\documentclass[11pt]{amsart}

\usepackage{amsmath}
\usepackage{amssymb}
\usepackage{graphics}
\usepackage[ps,dvips,xdvi,arrow,frame,matrix,curve,cmtip]{xy}
\newcommand{\N}{\mathbb{N}}
\newcommand{\intO}{{\int_{\Omega}}}

\newcommand{\dx}{{\, dx}}
\newcommand{\ds}{{\, ds}}

\newcommand{\sig}{\sigma}
\newcommand{\Sig}{\Sigma}
\newcommand{\signa}{{\rm sig}}

\newcommand{\gam}{\gamma}
\newcommand{\Gam}{\Gamma}
\newcommand{\Gamz}{\Gamma_0}
\newcommand{\DZ}{\D_6 \times \Z_2}
\newcommand{\lam}{\lambda}

\newcommand{\al}{\alpha}
\newcommand{\Om}{\Omega}
\newcommand{\R}{\mathbb{R}}
\newcommand{\C}{\mathbb{C}}
\newcommand{\Z}{\mathbb{Z}}
\newcommand{\D}{\mathbb{D}}

\DeclareMathOperator{\fix}{Fix}
\DeclareMathOperator{\stab}{Stab}
\DeclareMathOperator{\pstab}{pStab}

\renewcommand{\subsection}[1]{\bigskip\noindent{\bf #1.}}
\theoremstyle{plain} %% This is the default, anyway
\newtheorem{thm}{Theorem}[section]

\newtheorem{prop}[thm]{Proposition}

\theoremstyle{definition}
\newtheorem{definition}[thm]{Definition}

\theoremstyle{remark}

\newtheorem{example}[thm]{Example}

\newtheorem{algorithm}[thm]{Algorithm}

\begin{document}

\thanks{Partially supported by NSF Grant DMS-0074326}
\thanks{\today}

%\title[GNGA: Semilinear PDE on the Snowflake] {GNGA for Semilinear
%Elliptic PDE on a Fractal Region: \\ Symmetry and Automated Branch
%Following.}
\title[Symmetry and Automated Branch Following] {Symmetry and Automated Branch Following for
\\
a Semilinear Elliptic PDE on a Fractal Region}

\author{John M. Neuberger}
\author{N\'andor Sieben}
\author{James W. Swift}

\email{
John.Neuberger@nau.edu,
Nandor.Sieben@nau.edu,
Jim.Swift@nau.edu}

\address{
Department of Mathematics and Statistics,
Northern Arizona University PO Box 5717,
Flagstaff, AZ 86011-5717, USA
}


\begin{abstract} We apply the Gradient-Newton-Galerkin-Algorithm (GNGA) of
Neuberger \& Swift to find solutions to a semilinear elliptic Dirichlet
problem on the region whose boundary is the Koch snowflake.
In a recent paper, we described an accurate and
efficient method for generating a basis of eigenfunctions of the
Laplacian on this region. In that work, we used the symmetry of the snowflake
region to analyze and post-process the basis, rendering it
suitable for input to the GNGA. The GNGA uses Newton's method on
the eigenfunction expansion coefficients
to find
solutions to the semilinear problem.
This article introduces the {\em bifurcation digraph},
an extension of the lattice of isotropy subgroups.
For our example, the bifurcation digraph shows
the 23 possible symmetry types of solutions to the PDE
and the 59 generic symmetry-breaking bifurcations among these symmetry types.
Our numerical code uses continuation methods, and follows branches created at symmetry-breaking bifurcations,
so the human user does not need to supply initial guesses for Newton's method.
Starting from the known trivial solution, the code automatically
finds at least one solution with each of the symmetry types that we predict can exist.
Such computationally intensive investigations necessitated
the writing of automated branch following code, whereby symmetry
information was used to reduce the number of computations per GNGA
execution and to make intelligent branch following decisions at bifurcation
points. \end{abstract}

\maketitle

\tolerance=10000

\begin{section}{Introduction.}
We seek numerical solutions to the semilinear elliptic boundary
value problem
\begin{align}
\label{pde}
        \Delta u + f_\lam(u) &= 0  \textrm{ in } \Omega \nonumber \\
         u &= 0  \textrm{ on } {\partial \Omega},
\end{align}
where $\Delta$ is the Laplacian operator, $\Omega\subset\R^2$ is the
%(open)
region whose boundary $\partial\Omega$ is the Koch snowflake,
$u:\Omega \to \R$ is the unknown function, and $f_\lam:\R\to\R$ is a one-parameter family of
odd functions. For convenience, we refer to $\Om$ as the {\em Koch
snowflake region}. This article is one of the first to consider a
nonlinear PDE on a region with fractal boundary. In this paper, we
choose the nonlinearity to be
\begin{eqnarray}
\label{nonlinearity}
f_\lam(u)=\lambda u + u^3,
\end{eqnarray}
and treat  $\lam \in \R$ as the bifurcation parameter.
When the parameter is fixed, we will sometimes use $f$ in place of $f_\lam$.
Using this convention, note that $\lam = f'(0)$.
%$\lam = {f_\lam}'(0)$.

This paper exploits the hexagonal symmetry of the Koch snowflake region,
and the fact that $f$ is odd.
Our nonlinear code would work with any region
with hexagonal symmetry and any odd `superlinear' function $f$ (see \cite{ccn}),
and with minor modification for other classes of nonlinearities as well.
We chose to work with odd $f$ primarily because of the rich symmetry structure.
The explicit shape of $\Om$ represents a considerable technological challenge
for the computation of the eigenfunctions \cite{lnrg,nss}, which are required as input to the nonlinear code.

It is well known that the eigenvalues of the Laplacian under this boundary condition satisfy
\begin{eqnarray}
\label{evals}
0<\lambda_1<\lambda_2\leq\lambda_3\leq\cdots\to\infty,
\end{eqnarray}
and that the corresponding eigenfunctions
$\{\psi_j\}_{j\in\N}$
can be chosen to be an orthogonal basis
%of $C^{\infty}(\Omega)$ functions
for the Sobolev space $H=H_0^1(\Om)=W_0^{1,2}(\Om)$,
and an orthonormal basis for
the larger Hilbert space $L^2=L^2(\Om)$.
The inner products are
$$
  \langle u,v\rangle_H = \intO \nabla u \cdot \nabla v \dx\ \hbox{ and }
\ \langle u,v\rangle_2 = \intO        u \;           v \dx ,
$$
respectively (see \cite{adams, gt, lapidus, lappang}).
Theorem 8.37 and subsequent remarks in \cite{gt}
imply that the eigenfunctions are in $C^\infty(\Om)$.
In \cite{lappang}, properties of the gradients of eigenfunctions near boundary points
are explored in light of
the lack of regularity of $\partial\Om$.

Using the Gradient-Newton-Galerkin-Algorithm (GNGA, see
\cite{ns}) we seek approximate solutions $u=\sum_{j=1}^{M} a_j \psi_j$ to
(\ref{pde}) by applying Newton's method to the eigenfunction expansion
coefficients of the gradient $\nabla J(u)$ of a nonlinear functional $J$
whose critical points are the desired solutions. The definition of $J$,
the required variational equations, a description of the GNGA, and a brief
history of the problem are the subject of Section~\ref{GNGA_section}.

The GNGA requires as input a basis spanning a sufficiently large but
finite dimensional subspace $B_M = {\rm span}\{\psi_1,\ldots,\psi_M\}$,
corresponding to the first $M$ eigenvalues $\{\lambda_j\}_{j=1}^{M}$.
As described in \cite{nss}, a grid $G_N$  of $N$
carefully placed points is used to approximate the eigenfunctions.
These are the same grid points used for the
numerical integrations required by Newton's method.
Section~\ref{basis_section}
briefly describes the process we use for generating the eigenfunctions.

Section~\ref{symmetry_section} concerns the effects of
symmetry on automated branch following.
The symmetry theory for linear operators found in \cite{nss} is summarized
and then the extensions required for nonlinear operators are described.
Symmetry-breaking bifurcations are analyzed in a way that allows
an automated system to follow the branches created at the bifurcations.
As we develop the theory, we present specific examples applying the
general theory to equation (\ref{pde}) on the snowflake region.
In particular, we find that there are 23 different symmetry types of solutions
to (\ref{pde}), and 59 generic symmetry-breaking bifurcations.
The symmetry types and bifurcations among them are summarized in a
{\em bifurcation digraph}, which generalizes the well-known
lattice of isotropy subgroups (see \cite{vol2}).
As far as we know, the bifurcation digraph is a new way
to organize the information about the symmetry-breaking bifurcations.

Section \ref{sce_section} describes how understanding the symmetry allows
remarkable increases in the efficiency of the GNGA.
Section~\ref{branch_section} describes the automated branch following.
We use repeated executions of the GNGA or a slightly modified algorithm
(parameter-modified GNGA) to follow solution branches of (\ref{pde}, \ref{nonlinearity}).
The GNGA uses Newton's method, which is known to work well if it
has a good initial approximation.  The main shortcoming
of Newton's method is that is works poorly without a good
initial approximation.
We avoid this problem by starting with the trivial solution ($u = 0$).
The symmetry-breaking bifurcations of the trivial solution are
found by the algorithm and the primary branches are started.
The program
%then
%recursively
follows the branches by continuation methods,
and then follows the new branches created at symmetry-breaking bifurcations.
To follow an existing branch, we vary $\lambda$ slightly between executions.
To start new solution branches created at bifurcation points, we
treat $\lambda$ as a variable and fix one of the null eigenfunctions
of the Hessian evaluated at the bifurcation point.
The symmetry analysis tells which null eigenfunction to use.
In this way solutions with all
23 symmetry types are found automatically, starting from $u=0$,
without having to guess any approximations for Newton's method.

In our experiments, many bifurcation diagrams were generated by applying
the techniques mentioned above. A selection of these diagrams are provided
in Section~\ref{results_section}, along with contour plots of solutions to
(\ref{pde}) corresponding to each of the 23 symmetry types predicted to
exist. We include evidence of the convergence of our algorithm as the
number of modes $M$ and grid points $N$ increase.

Many extensions to our work are possible, including enforcing different
boundary conditions on the same region, solving similar semilinear
equations on other fractal regions, and applying the methodology to partial
difference equations (PdE) on graphs \cite{graph}.
Section~\ref{conclusion_section} discusses some of these possible
extensions.
In particular, we are in the process of re-writing the suite of programs.
We plan to be able to solve larger problems using a parallel environment.
We will be able to solve problems with larger symmetry groups
by automating the extensive group theoretic calculations.
This concluding section also has a discussion of the convergence of the GNGA.

\end{section} % END INTRODUCTION

%\label{GNGA_section}
%\include{gnga}

\begin{section}{GNGA.}
\label{GNGA_section}

We now present the variational machinery for studying (\ref{pde}) and
follow with a brief description of the general GNGA. Section
\ref{branch_section} contains more details of the implementation of the
algorithm for our specific problem. Let $F_\lam(u)=\int_0^u f_\lam(s)\ds$ for all
$u\in{\R}$ define the primitive of $f_\lam$.
We then define the action functional $J:\R \times H\to{\R}$ by
\begin{equation}
\label{Jdef}
J(\lam, u)=\intO \left \{{\textstyle \frac12}|\nabla u|^2 - F_\lam(u) \right \}\dx.
\end{equation}
We will sometimes use $J: H\to{\R}$ to denote $J(\lam, \cdot )$.
%
% jmn - Begin modified 6/2/06
%
The class of nonlinearities $f$ found in \cite{ccn, cdn2, graph, rab} imply that
$J$ is well defined and of class $C^2$ on $H$.
The choice (\ref{nonlinearity}) we make in this paper belongs to that class.
Critical points of $J$ are by definition weak solutions of (\ref{pde})
(see for example \cite{ccn, rab, gt}),
and clearly classical solutions are critical points.
The usual ``bootstrap'' argument of repeatedly applying Theorem 8.10 of \cite{gt} can be used in our case.
Specifically, $H^k_0$ is embedded in $L^q$ for all $q\geq2$ when the space diminsion $n$ is $2$, regardless of
the regularity of $\partial\Omega$ (due to the zero Dirichlet boundary condition, see \cite{adams}).
Hence $u\in H^k$ implies $f(u)\in H^k$ as well.
As a result, if $u$ is a critical point then $u\in C^\infty(\Omega)\cap C(\bar\Omega)$, hence a classical solution.
If one considered boundary conditions, space dimensions, and nonlinear terms other than the choices made in this paper,
it could happen that critical points  would be weak not classicial solutions.
Regardless, our approximations lie in $B_M\subset C^{\infty}$.
Here, the existence proofs for positive, negative, and sign-changing exactly once
solutions from \cite{ccn, rab} immediately give at least 3 nontrivial (classical) solutions for our specific
superlinear boundary value problem;
appealing to symmetry implies the existence of even more solutions (see for example \cite{graph}).
%One of our chief goals in this paper is to use our understanding of symmetry to accurately and automatically follow
%enough branches to obtain at least one approximate solution
%from each of the possible 23 types of symmetries.

The choice of $H$ for the domain is crucial to the analysis of the PDE
(see \cite{ccn, conm}, and references therein),
as well as for understanding the theoretical basis of effective steepest descent algorithms
(see \cite{cdn, jwn, num}, for example).
We will work in the coefficient space $\R^M \cong B_M$.
The {\em coefficient vector of} $u \in B_M$ is the vector $a\in\R^M$ satisfying $u=\sum_{j=1}^M a_j \psi_j$.
%We will work in the finite-dimensional subspace
%$B_M$
%of the smooth eigenfunctions $\{\psi_j\}$ are normalized in $L^2$,
%and in the corresponding coefficient space,
%where $u\in B_M$ if and only if $u=\sum_{j=1}^M a_j \psi_j$ for
%some coefficient
%vector $a\in\R^M$.
%
% jmn - End modified 6/2/06
%
Using the corresponding eigenvalues (\ref{evals}) and integrating by parts,
the quantities of interest are
\begin{equation}
\label{JprimeDef}
g_j = J'(u)(\psi_j) = \intO \{\nabla u \cdot \nabla \psi_j -  f(u) \, \psi_j \}
= a_j\lambda_j - \intO f(u) \, \psi_j , \quad \mbox{and}
\end{equation}
%
%and
%
\begin{equation}
\label{JprimePrimeDef}
h_{jk} =
J''(u)(\psi_j,\psi_k) = \intO \{\nabla \psi_j \cdot \nabla \psi_k -  f'(u) \, \psi_j\,\psi_k\}
= \lambda_j\delta_{jk} - \intO f'(u) \, \psi_j \,\psi_k ,
\end{equation}
where $\delta_{jk}$ is the Kronecker delta function.
Note that there is no need for numerical differentiation
when forming gradient and Hessian coefficient vectors and matrices
in implementing Algorithm \ref{algorithm};
this information is encoded in the eigenfunctions.

The vector $g\in\R^M$ and the $M\times M$ matrix $h$ represent suitable projections of the
$L^2$ gradient and Hessian of $J$, restricted to the subspace $B_M$, where all such
quantities are defined.
For example, for $u=\sum_{j=1}^M a_j \psi_j$, $v=\sum_{j=1}^M b_j \psi_j$, and
$w=\sum_{j=1}^M c_j \psi_j$, we have:
$$
P_{B_M}\nabla_2 J(u) = \sum_{j=1}^M g_j \psi_j, \
J'(u)(v) = g\cdot b, \ {\rm and} \
J''(u)(v,w) = hb\cdot c = b\cdot hc.
$$
We can identify $g$ with the approximation
$P_{B_M}\nabla_2 J(u)$ of $\nabla_2J(u)=\Delta u + f(u)$,
which is defined for $u\in B_M$.
The solution $\chi$ to the $M$-dimensional linear system $h\chi=g$ is then
identified with the (suitably projected) search direction
$(D^2_2J(u))^{-1}\nabla_2 J(u)$,
which is not only defined for $u\in B_M$, but is there
equal to $(D^2_HJ(u))^{-1}\nabla_H J(u)$.
We use the least squares solution of $h\chi=g$. In practice,
the algorithm works even near bifurcation points where the Hessian is not invertible.

The heart of our code is Newton's method in the space of eigenfunction coefficients:
\begin{algorithm}{(GNGA)}
\label{algorithm}
\begin{enumerate}
\tt
\item Choose initial coefficients $a=\{a_j\}_{j=1}^{M}$, and set $u=\sum
a_j \psi_j$.
% I took out the "put_a_on_S", mainly because we don't do it in the new code.
%\item Replace $a \leftarrow t a$ and $u \leftarrow t u$, where $t > 0$ is chosen so that $J'(t u)(u) = 0$.
\item Loop
        \begin{enumerate}
        \item Calculate the gradient vector $g=\{J'(u)(\psi_j)\}_{j=1}^{M}$
from equation (\ref{JprimeDef}).
        \item Calculate the Hessian matrix $h=\{J''(u)(\psi_j, \psi_k)\}_{j,\;k=1}^{M}$ from equation (\ref{JprimePrimeDef}).
        \item Exit loop if $|| g ||$ is sufficiently small.
        \item Solve $h \chi = g$ for the Newton search direction $\chi \in \R^M$.
        \item Replace $a \leftarrow a - \chi$ and update $u=\sum a_j \psi_j$.
        \end{enumerate}
\item Calculate $\signa(h)$ and $J$ for the approximate solution.
\end{enumerate}
\end{algorithm}
If Newton's method converges
then we expect that $u$ approximates a solution to the PDE (\ref{pde}),
provided $M$ is sufficiently large and the eigenfunctions and numerical integrations are sufficiently
accurate. See Section \ref{conclusion_section}.

Our estimate for the Morse index (MI) of the critical point of $J$ is the signature of $h$, denoted
$\signa(h)$, which is defined as the number of negative eigenvalues of $h$.  This measures the
number of linearly independent directions away from $u$ in which $J$ decreases quadratically.

The basic Algorithm \ref{algorithm} is modified to take advantage of the symmetry of our problem.
The $M$ integrations required in step (a) and the $M(M+1)/2$ integrations in step (b) are reduced
to fewer integrations if the initial guess has nontrivial symmetry.

%%%%%%%%%%%%%%%%%%%%

We often use a ``parameter-modified'' version of the GNGA (pmGNGA).
In this modification, $\lambda$ is treated as an unknown variable
and one of the $M$ coefficients $a_k$ is fixed. Along a given
branch, symmetry generally forces many coefficients to be zero. When
a bifurcation point is located by observing a change in MI, we can
predict the symmetry of the bifurcating branches using the symmetry
of the null eigenfunctions of the Hessian. By forcing a small
nonzero component in the direction of a null eigenfunction
(orthogonal to the old branch's smaller invariant subspace), we can
assure that the pmGNGA will not converge to a solution lying on the
old branch. Another benefit of the pmGNGA is that it can handle a
curve bifurcating to the right as well as one bifurcating to the
left. In our system, the branches that bifurcate to the right have
saddle node bifurcations where they turn around and go to the left.
The pmGNGA can follow such branches while the normal GNGA cannot.

The implementation of pmGNGA is not difficult.  The $M$ equations
are still
$$
g_i = J'(u) (\psi_i) = 0,
$$
but the $M$ unknowns are
$$
   \tilde a = (a_1,\ldots,a_{k-1},\lambda,a_{k+1},\ldots,a_M),
$$
and the value of one coefficient, $a_k$, is fixed.
Consequently, we replace the Hessian matrix $h$ with a new matrix $\tilde h$ where the
$k$-th column is set to ${\partial g_i}/{\partial \lam} = -a_i$:
$$
\tilde h_{ij} = \left \{
\begin{array}{rl}
h_{ij} & \mbox{if } j \neq k \\
- a_i & \mbox{if } j = k
\end{array}
\right .
.
$$
The search direction $\tilde \chi$ is the solution to the system
$\tilde h \tilde \chi = g$.
The pmGNGA step is
      $$
      \tilde a \leftarrow \tilde a -  \tilde \chi ,
      $$
and then $u$ and $\lambda$ are updated.  After Newton's method converges, the $k$-th column of the original $h_{ij}$
is calculated and the MI of the solution, $\signa(h)$, is computed.

%%%%%%%%%%%%%%%%%%%%

We conclude this section with a very brief history of the analytical and
numerical aspects of the research into (\ref{pde}) given our type of
nonlinearity $f$. Our introduction to this general subject was \cite{ccn},
where a sign-changing existence result was proven. This theorem is
extended in \cite{cdn2}; we indicate briefly in
Section~\ref{results_section} where this so-called CCN solution can be
found on our bifurcation diagrams.
The article
\cite{cdn} was our first success in using symmetry to find higher MI
solutions.
The GNGA was developed in
\cite{ns}, wherein a much more detailed description of the variational
structure and numerical implementation can be found.
%The computational efforts related to this current project are somewhat more sophisticated.
%The more important aspects of our improvements are explained in
%Sections~\ref{symmetry_section}~and~~\ref{branch_section}.
The first implementation of the
GNGA for regions where the eigenfunctions
are not known in closed form is in \cite{hn}, where the region
is a Bunimovich stadium.
The article \cite{conm} provides a historical overview of the authors'
experimental results using variants of the Mountain Pass Algorithm (MPA,
MMPA, HLA) and the GNGA, as well as recent analytical results and a list
of open problems; the references found therein are extensive.
%The details concerning the eigenfunctions on the snowflake region, needed
%for the current paper, are in \cite{nss}.

\end{section}

%\label{basis_section}
% \include{basis}

\begin{section}{The Basis of Eigenfunctions.}
\label{basis_section}

In \cite{nss}, we describe theoretical and computational results that lead
to the generation of a basis of eigenfunctions solving
\begin{equation}
\label{lpde}
\Delta u + \lambda u = 0 \ \textrm{in } \Omega, \ \
u = 0  \ \textrm{on }  {\partial \Omega}.
\end{equation}
That paper details the grid technique and symmetry analysis that
accompanied the effort; we briefly summarize those results in this
section.

The Koch snowflake is a well-known fractal, with Hausdorff dimension
$\log_3 4$. Following Lapidus, Neuberger, Renka, and Griffith \cite{lnrg},
we take our snowflake to be inscribed in a circle of radius
$\frac{\sqrt{3}}{3}$ centered about the origin.
%With this choice, the
%polygonal approximations used in the fractal construction have side length
%that are powers of $1/3$.
We use a triangular grid $G_N$ of $N$ points to
approximate the snowflake region.
Then, we identify $u: G_N \rightarrow \R$ with $u \in \R^N$, that is,
\begin{equation}
\label{ui}
u(x_i) = u_i
\end{equation}
at grid points $x_i \in G_N$.
Our paper \cite{nss} differs from \cite{lnrg} in that we use a different placement
of the grid points and a different method of enforcing the boundary condition,
resulting in more accurate eigenvalue estimates
with fewer points.
Figure~\ref{grid} depicts the
levels 2 and 3 grids in the family of grids
% $\ell = 2$ grid
used in \cite{nss} to compute eigenfunctions;
we used the first $M$ eigenfunctions computed at levels 4, 5, and 6 in our nonlinear experiments.
The number of grid points
at level $\ell$ is $N = (9^\ell - 4^\ell)/5$, and the spacing between grid points is $h = 2/3^\ell$.
%other levels are in Table~\ref{grid_table}.
%
\begin{figure}
\begin{center}
\scalebox{.7}{\includegraphics{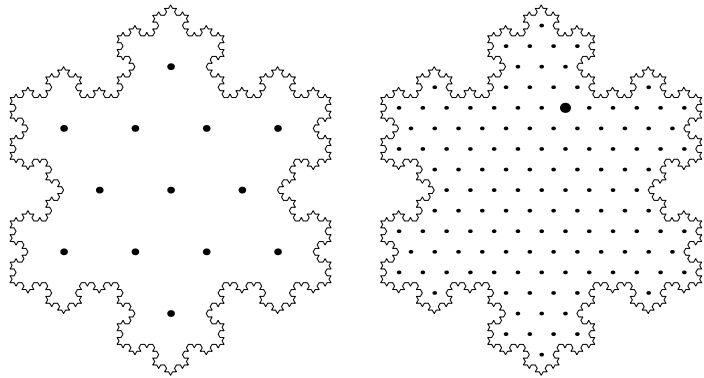}}
\caption{
%The Koch snowflake region $\Om$ with the 13 grid points at level $\ell = 2$, and 133 grid points at
%level $\ell = 3$.
%A generic point in the level three grid is indicated.
%Only at level 3 or greater are there grid points, line the one indicated, which are not on a line
%of reflection symmetry.
The Koch snowflake region $\Om$ with the grids $G_{13}$ and $G_{133}$ at levels $\ell = 2$ and 3, respectively.
A generic grid point (which is not on any line of reflection symmetry)
is indicated in the larger grid.
}
\label{grid}
\end{center}
\end{figure}
%
%In \cite{nss},
We computed the eigenvalues and
eigenfunctions for (\ref{lpde}) using ARPACK and this approximation to
the Laplacian with zero-Dirichlet boundary conditions:
%Our method of imposing the zero-Dirichlet boundary conditions can be summarized as
\begin{align}
\begin{split}
\label{template}
-\Delta u (x) \approx  \frac{2}{3h^2} \left(  (12-\hbox{number of neighbors})\, u(x) -
                       \sum \{\hbox{neighboring values of } u\} \right).
\end{split}
\end{align}

%\begin{table}
%\begin{center}
%
%\begin{tabular}{|c|c|c|c|c|c|c|}
%\hline
%$\ell$ & 1 & 2 & 3 & 4 & 5 & 6 \\
%\hline
%$N$ & 1 & 13 & 133 & 1261 & 11605 & 105469
%\\ \hline
%\end{tabular}
%\vspace*{.2in}
%\caption{
%The number $N$ of interior grid points as a function of the level $\ell$.
%The spacing between grid points is $h = h_{\rm NSS}(\ell) = 2/3^\ell$.
%We typically use level $\ell=5$ in our nonlinear experiments.
%}
%\label{grid_table}
%\end{center}
%\end{table}
%and the eigenvalues eigenvalues a
%Using the differencing scheme in (\ref{template}) and the grid
%depicted in Figure~\ref{grid},
%we computed \cite{nss} eigenvalues and
%eigenfunctions for (\ref{lpde}) using ARPACK.
%Table~\ref{Richard_D1} lists approximations of the first ten eigenvalues; these values are primary bifurcation points.
The ARPACK is based upon an algorithmic variant of the
Arnoldi process called the Implicitly Restarted Arnoldi Method (see
\cite{arpack}) and is ideally suited for finding the eigen-pairs of the
large sparse matrices associated with the discretization of the Laplacian.

\end{section}

\begin{section}{Symmetry: The Lattice of Isotropy Subgroups and The Bifurcation Digraph.}

\label{symmetry_section}

This section describes equivariant bifurcation theory
% (see, for example \cite{chossat} \cite{vol2} or \cite{tsp})
as it applies to the branching of solutions to equation (\ref{pde}),
see \cite{chossat, vol2, tsp, lauterbach}.
We are able to describe the expected symmetry types of solutions to (\ref{pde}),
as traditionally arranged in a lattice of isotropy subgroups.
We introduce the {\em bifurcation digraph},
a refinement of the lattice,
which shows every possible generic bifurcation from one symmetry type to another
as a directed edge which is labeled with information about the bifurcation.
The bifurcation digraph is of interest in its own right and summarizes the essential information required
by our automated branch following code.
In this project, GAP (Groups, Algorithms, and Programming, see \cite{GAP}) was used solely to verify
the symmetry analysis we did by hand.
In our continuing projects GAP is a useful tool since it can perform the tedious calculations
and write the information in a format that can be read by the branch following code.
Matthews \cite{matthews} has used GAP to do similar calculations.
We apply this methodology to the snowflake domain being considered in this paper.
The analysis shows that solutions fall into 23 symmetry types,
and that there are 59 types of generic symmetry breaking bifurcations.

\subsection{Group Actions and the Lattice of Isotropy Subgroups}
Let $\Gamma$ be a finite group and $V$ be a real vector space.
A {\em representation} of $\Gamma$
is a homomorphism $\alpha: \Gamma \rightarrow GL(V)$.
Where convenient, we identify $GL(V)$ with the set of invertible matrices with real coefficients.
Every representation $\alpha$ corresponds to
a unique {\em group action}
of $\Gamma$ on $V$ by the rule
$\gamma \cdot v := \alpha(\gamma)(v)$ for all $\gamma \in \Gamma$ and $v \in V$.
We will usually use the action rather than the representation.
The {\em group orbit} of $v$ is $\Gamma \cdot v = \{ \gamma \cdot v \mid \gam \in \Gam \}$.

\begin{example}
\label{d6actions}
Let
$$
\D_6 := \langle \rho, \sigma \mid \rho^6 = \sigma^2 = 1, \, \rho \, \sigma = \sigma \rho^5 \rangle
$$
be the dihedral group with 12 elements.  It is convenient to
define $\tau = \rho^3 \sig$.
It follows that $\sig \tau = \tau \sig = \rho^3$.
The group $\D_6$ is the symmetry of a regular hexagon, and of the Koch Snowflake region $\Omega$.
The standard $\D_6$ action on the plane is given by
\begin{align}
\begin{split}
\label{actionOnPlane} \rho \cdot(x,y)  &= \mbox{$
   \left ( \frac{1}{2} x + \frac{\sqrt{3}}{2} y, -\frac{\sqrt{3}}{2} x + \frac{1}{2} y \right )$} \\
\sig \cdot(x,y)  &= (-x, y) \\
\tau \cdot(x,y)  &= (x, -y) .
\end{split}
\end{align}
In this action, $\rho$ is a rotation by $60^\circ$, $\sig$ is a reflection across the $y$-axis,
and $\tau$ is a reflection across the $x$-axis.  These group actions are depicted in
Figure \ref{sols1}, near the end of the paper.

We will denote subgroups of $\D_6$ by listing the generators.
While any given subgroup of $\D_6$ can be defined using only $\rho$ and $\sig$, we find it
geometrically descriptive to use $\tau$ in certain cases. For example, we prefer $\langle \rho^2, \tau \rangle$
to the equivalent  $\langle \rho^2, \rho \sig \rangle$.
In order to make relationships among subgroups intuitive,
we often include $\tau$ when its membership is implied by the other generators
(see for example Figure \ref{lattice}).

The standard $\D_6$ group action (\ref{actionOnPlane}) is not the only action we consider.
For a function $u \in L^2(\Om)$ and group element $\gamma \in \D_6$,
we define $(\gamma \cdot u)(x) = u(\gamma^{-1}\cdot x)$.
In this paper, a vector $u$ defined by $u_i = u(x_i)$, for a given grid $G_N=\{x_i\}_{i=1}^N$,
is a discrete approximation of a function on $\Omega$.  The $\D_6$ group action on $u \in \R^N$
is a permutation of the components: $(\gamma \cdot u)_i = u(\gamma^{-1}\cdot x_i)$.
Given a function $u \in L^2(\Om)$ or $\R^N$, the group orbit $\D_6 \cdot u$ consists of
functions obtained from $u$ by a reflection or rotation.
\end{example}
\begin{example}
\label{dzaction}
The group $\DZ$, where $\Z_2 = \{1, -1\}$, acts on $L^2(\Om)$ in a natural way.
For all $(\gamma,z)\in \D_6\times\Z_2$, define
\[
(\gamma, z) \cdot u = z (\gamma \cdot u) .
\]
We will denote $(\gamma, 1) \in \DZ$ by $\gamma$ and $(\gamma, -1) \in \DZ$ by $-\gamma$.
With this natural notation $(-\gamma) \cdot u = -(\gamma \cdot u)$, which we
call simply $-\gamma \cdot u$.
\end{example}
Let us recall some facts about group actions, following
\cite{chossat,vol2,tsp}. The {\it isotropy subgroup} or {\it
stabilizer of $v \in V$ in $\Gamma$} is
\[
\stab(v, \Gam):= \{\gamma \in \Gamma \mid \gamma \cdot v = v\} .
\]
The isotropy subgroup measures the symmetry of $v$,
and is sometimes called the little group of $v$, or $\Gamma_v$.
If the group $\Gamma$ is understood, we may simply write $\stab(v)$ in place of $\stab(v, \Gam)$.
The {\em stabilizer of a subset
$W \subseteq V$ in $\Gam$}
is $\stab(W, \Gam):=\{ \gam \in \Gam \mid \gam \cdot W = W \}$.
This must be distinguished from the {\em point stabilizer of a subset}
$$
\pstab(W, \Gam) := \{ \gamma \in \Gamma \mid \gamma \cdot v = v \mbox{ for all } v \in W \} =
\bigcap\{\stab(v, \Gam) \mid v\in W \} .
$$
Another commonly used notation is $\Gamma_W$ for the stabilizer and $\Gamma_{(W)}$ for the point stabilizer.
Note that $\pstab(W, \Gam)$ is always normal in $\stab(W, \Gam)$,
and the {\em effective symmetry group acting on $W$} is
$\stab(W, \Gam)/\pstab(W, \Gam)$, which acts faithfully on $W$.

If $\Sig$ is a subgroup of $\Gam$ then the {\em fixed point subspace of $\Sig$ in $V$} is
\[\fix(\Sig, V):=\{v\in V\mid\ \gamma \cdot v = v \mbox{\rm \ for all }\gam \in \Sig\}.\]
Another notation for the fixed point subspace is $V_\Sig$.
We write $\fix(\Sig)$ when $V$ is understood.

An {\em isotropy subgroup} of the $\Gam$ action on $V$ is the stabilizer of some point $v \in V$.
For some group actions, not every subgroup of $\Gam$ is an isotropy subgroup.

\begin{example}
\label{d6onR12}
Consider the $\D_6$ action on the plane $\R^2$
described in equation (\ref{actionOnPlane}. It is well-known that $\langle \rho \rangle$
is not an isotropy subgroup of this action.

Now consider the $\D_6$ action on the function space $L^2(\Om)$.
We give a standard argument that every subgroup of $\D_6$ is an isotropy subgroup.
Start with a function $u^*$ that is zero everywhere except for a small region,
and suppose that the region is distinct from each of its nontrivial images under the $\D_6$ action.
Then for any subgroup $\Sig \leq \D_6$,
the {\em average of the function $u^*$ over $\Sigma$}, defined as
\begin{equation}
\label{averageOverSubgroup}
P_\Sig (u^*) = \frac{1}{| \Sig |}
\sum_{\gamma \in \Sig} \gamma \cdot u^*
\end{equation}
has isotropy subgroup $\Sigma$.
Therefore every subgroup of the $\D_6$ action on $L^2(\Om)$ is an isotropy subgroup.
The average over the group is an example of a Haar operator,
and $P_\Sig: V \rightarrow \fix(\Sig, V)$
is an orthogonal projection operator \cite{wangzhou2}.

Similarly, every subgroup of $\D_6$
is an isotropy subgroup of the $\D_6$ action on $\R^N$, the space of
functions on the grid $G_N$, provided $\ell \geq 3$.
This follows from averaging the function that is 1 at a generic lattice point, and 0 elsewhere.
Recall that a {\em generic point} is one whose isotropy subgroup is trivial.
Figure \ref{grid} shows that the level two grid $G_{13}$ does not have a generic point,
while the level three grid $G_{133}$ does.
%$\ell = 2$, shown in Figure \ref{grid}, does not have a generic lattice
%point, and
% For the action of $\D_6$ on $\R^{13}$, the space of functions on $G_{13}$,
%averaging any function over $\langle \rho \rangle$ gives a function with isotropy subgroup $\D_6$.
%Hence,
%$\langle \rho \rangle$ is not an isotropy subgroup of the $\D_6$ action on the
%corresponding function space $\R^{13}$.
%
%While the level 3 grid is the smallest of our grids
%with a generic lattice point, we can get by with fewer points.
Thus, the space of functions on $G_{133}$ has the same isotropy subgroups as $L^2(\Om)$, but
%$\D_6$ action on $\R^{133}$
a much smaller space has this same property.
Start with any generic point $x_1 \in \Om$.
Then $\D_6$ acts on
the space of functions on the 12 points
$\D_6 \cdot x_1$.
%$\{ \gam \cdot x_1 \mid \gam \in \D_6 \}$.
This $\D_6$ action on $\R^{12}$
has the same structure of isotropy subgroups
as the $\D_6$ action on $L^2(\Om)$, and
is the $\D_6$ action used in our GAP calculations.
The corresponding 12-dimensional representation is
the well-known {\em regular representation} of $\D_6$
(see \cite{scott, sternberg,tinkham}).
\end{example}
The symmetry of functions is described by two related concepts.
A function $q: V \rightarrow \R$ is $\Gam$-{\em invariant} if
$q(\gamma \cdot v)=q(v)$ for all $\gamma \in \Gam$ and all $v \in V$.
Similarly, an operator $T: V \rightarrow V$ is $\Gam$-{\em equivariant}
if $T(\gamma \cdot v)=\gamma\cdot T(v)$ for all $\gamma \in \Gam$ and all $v \in V$.
\begin{example}
The energy functional $J$ defined in equation (\ref{Jdef}) is $\DZ$-invariant.
The nonlinear PDE (\ref{pde}) can be written as $(\Delta+f)(u)=0$,
where $\Delta +f$ is a $\DZ$-equivariant operator.
(There are subtleties concerning the domain and range of $\Delta$.
See \cite{chossat, cdn} for a careful treatment of the function spaces.)
In particular, $\Delta +f$ is $\D_6$-equivariant because the snowflake
region $\Omega$ has $\D_6$ symmetry, and
$(\Delta+f)(-u) = -(\Delta+f)(u)$, since $f$ is odd.
As a consequence, if $u$ is a solution to (\ref{pde}), then so
is every element in its group orbit $(\DZ) \cdot u $.
\end{example}
The isotropy subgroups and fixed point subspaces are important because of the
following simple yet powerful results.  See \cite{chossat, vol2, tsp}.
\begin{prop}
\label{fps}
Suppose $\Gamma$ acts linearly on $V$, $T: V \rightarrow V$ is $\Gam$-equivariant and
$\Sig$ is an isotropy subgroup of $\Gam$.
\begin{itemize}
\item[(a)]
If $v\in \fix(\Sig)$ then
$T(v)\in \fix(\Sig)$.
%If $v\in \fix(\Sig, V)$ then $T(v)\in \fix(\Sig, V)$.
Thus, $T|_{\fix(\Sig)}: \fix(\Sig) \rightarrow \fix(\Sig)$ is defined.
\item[(b)]
$\stab(\fix(\Sig)) = N_\Gam(\Sig)$, the normalizer of $\Sig$ in $\Gam$,
and $\pstab(\fix(\Sig)) = \Sig$.
%$\stab(\Gam, \fix(\Sig)) = N_\Gam(\Sig)$, the normalizer of $\Sig$ in $\Gam$.
\item[(c)]
$T |_{\fix(\Sig)}$ is $N_{\Gam}(\Sig)$-equivariant.
%$T |_{\fix(\Sig, V)}$ is $N_{\Gam}(\Sig)$-equivariant.
\item[(d)]
$T |_{\fix(\Sig)}$ is $N_{\Gam}(\Sig) / \Sig$-equivariant,
and $N_\Gam(\Sig)/\Sig$ acts faithfully on $\fix(\Sig)$.
%$T |_{\fix(\Sig, V)}$ is $N_{\Gam}(\Sig) / \Sig$-equivariant, and $N_\Gam(\Sig)/\Sig$ acts faithfully on $\fix(\Sig, V)$.

\end{itemize}
\end{prop}
If $\Sigma$ is a subgroup of $\Gamma$, the normalizer of $\Sigma$ in $\Gamma$ is defined to be
$N_{\Gamma}(\Sigma) := \{ \gamma \in \Gamma \mid \gamma \Sigma = \Sigma \gamma \}$,
which is the largest subgroup of $\Gamma$ for which $\Sigma$ is a normal subgroup.
The presence of the normalizer in Proposition~\ref{fps}(b) is interesting, since
the normalizer is a property of the abstract groups, and is independent of the group action.
\begin{example}
As a consequence of Proposition \ref{fps}, we can solve the PDE (\ref{pde}), written as $(\Delta + f)(u)=0$,
by restricting $u$ to
functions in $\fix(\Sig, L^2(\Omega))$.  This leads to a simpler problem since
the function space $\fix(\Sig, L^2(\Omega))$ is simpler than $L^2(\Omega)$.  An example
of this is in Costa, Ding, and Neuberger \cite{cdn}.  The
techniques of that paper, applied to our problem, would
find sign-changing solutions with Morse index 2 within the space $\fix(\D_6, L^2(\Omega))$.
This space consists of all functions which are unchanged under all of the rotations
and reflections of the snowflake region.

Proposition~(\ref{fps}) also applies to the GNGA, since the Newton's method iteration
mapping is $\DZ$-equivariant.
If the initial guess is in a
particular fixed point subspace, all the iterates will be
in that fixed point subspace.
This fact can be used to speed numerical calculations,
as described in Section \ref{sce_section}.
\end{example}

Two subgroups $\Sig_1, \Sig_2$ of $\Gam$ are conjugate ($\Sig_1 \sim \Sig_2$) if
$\Sig_1 =\gam \Sig_2 \gam^{-1}$ for some $\gam\in \Gam$.
The {\em symmetry type} of $v \in V$ for the $\Gamma$ action is the conjugacy class of $\stab(v, \Gam)$.
Note that $\stab(\gamma \cdot v) = \gamma \stab (v) \gamma^{-1}$.
Thus, every element of a group orbit $\Gamma \cdot v$ has the same symmetry type.

Let $\mathcal S=\{S_i\}$ denote the set of all symmetry types of a $\Gam$ action on $V$.
The set $\mathcal S$ has a natural partial order,
with $S_i \leq S_j$ if
there exits $\Sig_i \in S_i$ and $\Sigma_j \in S_j$ such that $\Sig_i \leq \Sig_j$.
The partially ordered set $(\mathcal S, \leq)$ is called the
{\em lattice of isotropy subgroups} of the $\Gam$ action on $V$ \cite{vol2}.
The {\em diagram} of the lattice of isotropy subgroups is
a directed graph with vertices $S_i$ and arrows $S_i \leftarrow S_j$
if, and only if, $S_i \lneq S_j$ and there is no symmetry type between $S_i$ and $S_j$.

\begin{example}
The symmetry type of a solution $u$ to our PDE (\ref{pde}) for the $\DZ$ action
is the conjugacy class of $\stab(u, \DZ)$;
we refer to this as the symmetry type of $u$, without reference to $\DZ$.
The discussion of $\D_6$ acting on $L^2(\Om)$ in Example \ref{d6onR12} can easily be
extended to $\DZ$ acting on $L^2(\Om)$.  Note that if $-1 \in \Sigma \leq \DZ$, then the average
of any function over $\Sigma$ is $u = 0$.  Therefore the only isotropy subgroup of $\DZ$ which contains
$-1$ is $\DZ$ itself.  On the other hand, the argument in Example \ref{d6onR12} shows that
any subgroup of $\DZ$ which does not contain $-1$ is an isotropy subgroup.
Therefore, $\Sigma \leq \DZ$ is an isotropy subgroup
of this group action if and only if $\Sigma = \DZ$ or $-1 \notin \Sigma$.

This result allowed us to compute the isotropy subgroups by hand.
We verified our calculations using GAP.
There are exactly 23 conjugacy classes of isotropy subgroups for the $\DZ$ action on $L^2(\Omega)$,
shown in condensed form in Figure \ref{lattice}.
Thus, a solution to the PDE (\ref{pde}) has one of 23 different symmetry types.
\end{example}
\begin{figure}
\scalebox{1}{
\xymatrix@R=40pt@C=75pt{
 &\fbox{$ \Gam_0 = \langle \rho, \sig, \tau, -1 \rangle = \DZ
$}
%  \ar@{->}_<>(.1)4[d]^<>(.5){\Z_2}
  \ar@{->}_<>(.1)4[d]
\\
&
%   \fboxrule=1.5pt
   \fbox{$
   \begin{aligned}
%   \vspace*{10pt}
   &\Gam_1 = \langle\rho,\sigma,\tau\rangle = \D_6 \\
   &\Gam_2 = \langle\rho,-\sigma,-\tau\rangle \\ % &  V^{(2)}\\
   &\Gam_3 = \langle-\rho,\sigma,-\tau\rangle \\ % &  V^{(3)}\\
   &\Gam_4 = \langle-\rho,-\sigma,\tau\rangle\\ % &  V^{(4)}\\
   \end{aligned}
   $}
    \ar@{->}_<>(.1)2_<>(.9)2[d]
    \ar@{->}_<>(.95)2[rd]
    \ar@{->}[ld]
\\
%   \fboxrule=1.5pt
   \fbox{$
   \begin{aligned}
   &\Gam_5  =\langle\sigma,\tau\rangle   \\ % &  V^{(1)}\oplus V^{(5)}_{1} \\
   &\Gam_6  =\langle-\sigma,-\tau\rangle \\ % &  V^{(2)}\oplus V^{(5)}_{2}\\
   &\Gam_7  =\langle\sigma,-\tau\rangle  \\ % &  V^{(3)}\oplus V^{(6)}_{1}\\
   &\Gam_8 = \langle-\sigma,\tau\rangle  \\ % &  V^{(4)}\oplus V^{(6)}_{2}\\
   \end{aligned}
   $}
   \ar@{->}_<>(.1)2_<>(.9)2[d]
   \ar@{->}_<>(.95)2[rd]
 &
   \fbox{$
   \begin{aligned}
   &\Gam_9     =\langle\rho^2,\sigma\rangle  \\ % &  V^{(1)}\oplus V^{(3)}\\
   &\Gam_{10}  =\langle\rho^2,\tau\rangle    \\ % &  V^{(1)}\oplus V^{(4)}\\
   &\Gam_{11}  =\langle\rho^2,-\tau\rangle   \\ % &  V^{(2)}\oplus V^{(3)}\\
   &\Gam_{12}  =\langle\rho^2,-\sigma\rangle \\ % &  V^{(2)}\oplus V^{(4)}\\
   \end{aligned}
   $}
    \ar@{->}_<>(.85)4[rd]
    \ar@{->}[ld]
&
   \fbox{$
   \begin{aligned}
  &\Gam_{13}  =\langle\rho\rangle     \\ %       &  V^{(1)}\oplus V^{(2)} \\
  &\Gam_{14}  =\langle-\rho\rangle    \\ %       &  V^{(3)}\oplus V^{(4)}\\
    \end{aligned}
    $}
    \ar@{->}_<>(.9)2[d]
    \ar@{->}[ld]
\\
 \fbox{$
    \begin{aligned}
   &\Gam_{15}  =\langle\sigma\rangle  \\ %  &  V^{(1)}\oplus V^{(3)}\oplus V^{(5)}_1\oplus V^{(6)}_{1}\\
   &\Gam_{16}  =\langle\tau\rangle    \\ %  &  V^{(1)}\oplus V^{(4})\oplus V^{(5)}_1\oplus V^{(6)}_{2}\\
   &\Gam_{17}  =\langle-\tau\rangle   \\ %  &  V^{(2)}\oplus V^{(3)}\oplus V^{(5)}_2\oplus V^{(6)}_{1}\\
   &\Gam_{18}  =\langle-\sigma\rangle \\ %  &  V^{(2)}\oplus V^{(4)}\oplus V^{(5)}_2\oplus V^{(6)}_{2}\\
   \end{aligned}
 $}
    \ar@{->}_<>(.9)4[rd]
&
   \fbox{$
   \begin{aligned}
   &\Gam_{19}  =\langle\rho^3\rangle  \\ % &  V^{(1)}\oplus V^{(2)}\oplus V^{(5)}_1\oplus V^{(5)}_{2} \\
   &\Gam_{20}  =\langle-\rho^3\rangle \\ % &  V^{(3)}\oplus V^{(4)}\oplus V^{(5)}_2\oplus V^{(6)}_{2}\\
   \end{aligned}
   $}
    \ar@{->}_<>(.9)2[d]
&
   \fbox{$
 \Gam_{21} = \langle\rho^2\rangle % \quad V^{(1)}\oplus V^{(2)}\oplus V^{(3)}\oplus V^{(4)}
   $}
    \ar@{->}[ld]
\\
&
 \fbox{$
 \Gam_{22} = \langle 1 \rangle % \quad U
 $}
} }
\caption
{The condensed diagram of the isotropy lattice (see \cite{vol2}) for the
$\DZ$ action on $L^2(\Omega)$. The vertices of this diagram are the
symmetry types (equivalence
classes of isotropy subgroups).
We follow the convention \cite{chossat,vol2,tsp}
that one element $\Gam_i$ of each symmetry type $S_i = [ \Gam_i ]$ is listed.
The representatives $\Gam_i$ have the property that
$\Gam_i \leq \Gam_j$ iff $S_i \leq S_j$.
Contour plots of
solutions to PDE (\ref{pde}) with each of the 23 symmetry types are given
in Figures~\ref{sols1} and \ref{sols2}.
The diagram of the isotropy lattice is condensed
as in \cite{thomaswood}.
The small numbers on the edges tell the
number of connections emanating from each symmetry type in a box.
A missing small number means 1.
For example,
the two arrows representing $[\Gam_{21}] \leq [\Gam_{13}]$ and $[\Gam_{21}] \leq [\Gam_{14}]$
in the full diagram
are collapsed to a single arrow in the condensed diagram.
For $\Gam_0$ through $\Gam_4$, the $\tau$ generator is redundant since $\tau = \rho^3 \sig$, but its presence
makes the subgroups manifest.  For example,
$\Gam_2 = \langle \rho, -\sigma, -\tau \rangle = \langle \rho, -\sigma \rangle$, but the three generators
make it clear that
$\langle -\sigma, -\tau \rangle \leq \langle \rho, -\sigma, -\tau \rangle$.
}
\label{lattice}
\end{figure}

\subsection{Irreducible Representations and the Isotypic Decomposition}
In order to understand the symmetry-breaking bifurcations we
need to first understand irreducible representations and
the isotypic decomposition of a group action.
The information about the irreducible representations
is summarized in character tables \cite{scott, sternberg, thomaswood, tinkham}.
For our purposes, irreducible representations over the field $\R$
are required, see \cite{chossat, vol2, tsp}.
The {\em irreducible representations of $\Gamma$} are homomorphisms from $\Gamma$ to the space of
$d_j \times d_j$ real matrices: $\gamma \mapsto \alpha^{(j)}(\gamma)$, such that no proper
subspace of $\R^{d_j}$ is invariant under $\alpha^{(j)}(\gamma)$ for all $\gam \in \Gam$.
The {\em dimension of the irreducible representation $\alpha^{(j)}$} is $d_j$.
We call $W \subseteq V$ a {\em $\Gamma$-invariant subspace of $V$} if $\Gamma \cdot W \subseteq W$.
An {\em irreducible subspace of $V$} is an invariant subspace with no proper invariant subspaces.
Every irreducible subspace of the $\Gamma$ action on $V$ corresponds to a unique (up to similarity)
irreducible representation of $\Gamma$.
The dimension of the irreducible subspace is the same as the dimension
of the corresponding irreducible representation.

For each irreducible representation $\alpha^{(j)}$ of $\Gamma$,
the {\em isotypic component} of $V$ for the $\Gamma$ action,
denoted by $V_{\Gamma}^{(j)}$, is defined to be the direct sum of all of the
irreducible subspaces corresponding to the fixed $\alpha^{(j)}$
\cite{chossat, vol2, tsp, nss}.
The {\em isotypic decomposition} of $V$ is then
\begin{equation}
\label{isotypic}
V = \bigoplus_j V_{\Gamma}^{(j)}.
\end{equation}
Some of the isotypic components might be the single point at the origin.  These
can be left out of the isotypic decomposition.
A description of the isotypic components in terms of projection operators is
given in \cite{nss}.

For any group $\Gamma$, we denote the trivial representation by $\alpha^{(1)}$.
That is $\alpha^{(1)}(\gamma) = 1$ for all $\gamma \in \Gam$.  Thus,
if $\Gam$ is an isotropy subgroup of a $\Gamz$ action on $V$, then
$$
V^{(1)}_{\Gam} = \fix (\Gam, V) .
$$
\begin{example}
\label{isotypicD6}
Let us consider the $\D_6 = \langle \rho, \sig, \tau \rangle$ action on $L^2(\Omega)$.
We need to consider the six irreducible representations of $\D_6$,
which are listed in \cite{nss}, to find the isotypic decomposition of $L^2(\Omega)$.
Since these
isotypic components are central to our problem, we drop the $\D_6$ and
define $V^{(j)} := V^{(j)}_{\D_6}$, $j = 1, 2, \ldots, 6$ as follows:
\begin{align}
\label{v1-6}
V^{(1)} &=
\{u \in L^2(\Omega) \mid \rho \cdot u = u, ~ \sig \cdot u = u , ~ \tau\cdot u = u \} \\
V^{(2)} &=
\{u \in L^2(\Omega) \mid \rho \cdot u = u, ~\sig \cdot u = -u , ~ \tau \cdot u = -u \} \notag \\
V^{(3)} &=
\{u\in L^2(\Omega) \mid \rho \cdot u = -u, ~ \sig \cdot u = u , ~ \tau \cdot u = -u \} \notag \\
V^{(4)} &=
\{u \in L^2(\Omega) \mid \rho \cdot u =-u, ~ \sig \cdot u = -u , ~ \tau \cdot u = u \} \notag\\
V^{(5)} &=
\{u \in L^2(\Omega) \mid \rho^3 \cdot u = u, ~ u + \rho^2 \cdot u + \rho^4 \cdot u = 0 \} \notag\\
V^{(6)} &=
\{u \in L^2(\Omega) \mid \rho^3 \cdot u = -u, ~ u +\rho^2 \cdot u + \rho^4 \cdot u = 0 \} . \notag
\end{align}
%The last two components are subdivided in terms of a ``canonical basis'' as
%$V^{(5)} = V^{(5)}_1 \oplus V^{(5)}_2$
%and $V^{(6)} = V^{(6)}_1 \oplus V^{(6)}_2$, where
%\begin{align}
%\label{v56}
%V^{(5)}_1 &= \{u \in V^{(5)} \mid \sig \cdot u = u , ~ \tau \cdot u = u \} \\
%V^{(5)}_2 &=
%\{u \in V^{(5)} \mid \sig \cdot u = -u , ~ \tau \cdot u = -u\} \notag\\
%V^{(6)}_1 &=
%\{u \in V^{(6)} \mid \sig \cdot u = u , ~\tau \cdot u = -u \} \notag\\
%V^{(6)}_2 &=
%\{u \in V^{(6)} \mid \sig\cdot u = -u , ~ \tau \cdot u = u \} . \notag
%\end{align}
%As described in \cite{nss}, the canonical basis corresponds
%to a particular choice of the irreducible representation $\alpha^{(5)}$ or $\alpha^{(6)}$ within
%its equivalence class.
\end{example}

%\begin{example}
%\label{isotypicD2} Now, let us consider the isotypic decomposition
%of the $\Gamma_5 = \langle \sig, \tau \rangle \cong \Z_2 \times
%\Z_2$ action on $L^2(\Omega)$.  There are 4 irreducible
%representations of $\Z_2 \times \Z_2$, each of which is
%one-dimensional.
%The trivial representation is $\alpha^{(1)}$, and the others are
%defined by
%$\alpha^{(2)}(\sig) = -1$, $\alpha^{(2)}(\tau) = -1$,
%$\alpha^{(3)}(\sig) = 1$, $\alpha^{(3)}(\tau) = -1$,
%$\alpha^{(4)}(\sig) = -1$, and $\alpha^{(4)}(\tau) = 1$.
%The four isotypic components of $L^2(\Om)$ for the $\Gam_5$ action,
%written in terms of the spaces defined in (\ref{v1-6}) and
%(\ref{v56}), are
%$$
%\begin{array}{ll}
%V_{\langle \sig, \tau \rangle}^{(1)} =
% \{u \in L^2(\Omega) \mid \sig \cdot u = u, \tau \cdot u = u  \} &= V^{(1)} \oplus V^{(5)}_1  \\
%V_{\langle \sig, \tau \rangle}^{(2)} =
% \{u \in L^2(\Omega) \mid \sig \cdot u = -u, \tau \cdot u = -u  \} &= V^{(2)} \oplus V^{(5)}_2  \\
%V_{\langle \sig, \tau \rangle}^{(3)} =
% \{u \in L^2(\Omega) \mid \sig \cdot u = u, \tau \cdot u = -u  \} &= V^{(3)} \oplus V^{(6)}_1 \\
%V_{\langle \sig, \tau \rangle}^{(4)} =
% \{u \in L^2(\Omega) \mid \sig \cdot u = -u, \tau \cdot u = u  \} &= V^{(4)} \oplus V^{(6)}_2 .
%\end{array}
%$$
%Note that the canonical basis introduced in \cite{nss} is needed for an efficient description
%of the isotypic decomposition.  This is one example of how the canonical
%basis is the best basis of eigenfunctions to use with the GNGA.
%\end{example}

\begin{example}
\label{isotypicZ6}
The isotypic decomposition of $\Gam_{13} = \langle \rho \rangle \cong \Z_6$ illustrates some features
of {\em real} representation theory.  The irreducible representations of $\Z_6$ over $\C$ are all one-dimensional.
They are $\alpha^{(j)}(\rho) = (e^{i \pi/3})^{j-1}$ for
$j = 1, 2, \ldots, 6$.
Over the field $\R$, however, the one-dimensional irreducible representations of $\Z_6$ are given by
\begin{equation}
\label{1d_irreps}
\alpha^{(1)}(\rho) = 1 , \ \ \alpha^{(2)}(\rho) = -1,
\end{equation}
and the two-dimensional irreducible representations of $\Z_6$,
up to similarity transformations,
are given by
\begin{equation}
\label{2d_irreps}
\alpha^{(3)}(\rho) = \left (
\begin{array}{cc}
-\frac{1}{2} & \frac{\sqrt{3}}{2} \\
-\frac{\sqrt{3}}{2} & -\frac{1}{2} \ \
\end{array}
\right )
, \ \
\alpha^{(4)}(\rho) = \left (
\begin{array}{cc}
\frac{1}{2} & \frac{\sqrt{3}}{2} \\
-\frac{\sqrt{3}}{2} \ \ & \frac{1}{2}
\end{array}
\right ) .
\end{equation}
Note that $\alpha^{(3)}(\rho)$ is matrix for a rotation by $120^\circ$ and $\alpha^{(4)}(\rho)$ is a
$60^\circ$ rotation matrix.

An irreducible representation over $\R$ is called {\em absolutely irreducible} if it is also irreducible over
$\C$.  For example, all of the irreducible representations of $\D_6$ listed in \cite{nss} are absolutely irreducible,
as are the one-dimensional irreducible representations of $\Z_6$ in equation (\ref{1d_irreps}).
On the other hand, the two-dimensional irreducible representations of $\Z_6$ in equation (\ref{2d_irreps})
are {\em not} absolutely irreducible.

The four isotypic components of the $\langle \rho \rangle$ action on $L^2(\Om)$ are
$$
V^{(1)}_{\langle \rho \rangle} = \{u \in L^2(\Om) \mid \rho \cdot u = u \} = V^{(1)} \oplus V^{(2)}
$$
$$
V^{(2)}_{\langle \rho \rangle} = \{u \in L^2(\Om) \mid \rho \cdot u = -u \} = V^{(3)} \oplus V^{(4)}
$$
$$
V^{(3)}_{\langle \rho \rangle} = V^{(5)}, \ \mbox{and} \ V^{(4)}_{\langle \rho \rangle} = V^{(6)} .
$$
If we had used the complex irreducible representations, some of the corresponding isotypic components
would contain complex-valued functions.
It is more natural to use real irreducible representations, and consider only real-valued functions.
The price we pay is that most of the representation theory found in books,
and built into GAP, is done
for complex irreducible representations.
\end{example}
%As illustrated by these examples,
The isotypic decomposition for each of the 23 isotropy subgroups, $\Gam_i$, of $\DZ$
can be written as a direct sum of some subset of the eight spaces
$V^{(j)}$, for $j = 1, \ldots, 4$, and $V^{(j)}_1$ and $V^{(j)}_2$ for $j = 5, 6$
defined in (\ref{v1-6}) and \cite{nss}.
%(\ref{v56}).
%Example \ref{isotypicD2} shows that the subdivision of $V^{(5)}$ and $V^{(6)}$, described in
%(\ref{v56}), is needed.
%This is the reason why we prefer the canonical basis described in \cite{nss},
%which was not used in \cite{lnrg}.
%Each of the eigenfunctions is an element of one of these eight spaces,
%and this information is known by the C++ program.
%The $\DZ$ action on the Galerkin space $B_M = \{ \sum_{i=1}^M a_i \psi_i \} \cong \R^M$
%is quite simple.
%For example,
%$$
%V^{(5)}_1 \cap B_M =
%\left \{ \sum_{i=1}^M a_i \psi_i \mid a_i = 0 \mbox{ if } \psi_i \notin V^{(5)}_1 \right \}.
%$$
%Thus, the
The C++ program can easily check if a function is in any of the isotypic components
$V^{(j)}_{\Gam_i}$ of $B_M$ for each of the $\Gam_i$, $i = 0, 1, \ldots, 22$, actions.

\subsection{Symmetry-Breaking Bifurcations}
The fact that there are 23 {\em possible} symmetry types of solutions
to the PDE (\ref{pde}) begs the question, do solutions with each of these symmetry types exist?
Clearly the trivial solution $u=0$, with symmetry type $S_0$, exits.
Our procedure for finding approximate solutions with each of these symmetry types is
to start with the trivial solution and recursively
follow solution branches created at symmetry-breaking bifurcations.

Let us start by abstracting the PDE defined by (\ref{pde}),
which depends on the real parameter $\lambda$.
Let $V$ be an inner product space and
$J:\R\times V\to\R$ be a family of $\Gamz-$invariant functions that depends on a parameter
$\lambda$.  That is, $J(\lam, \gam \cdot u) = J(\lam, u)$ for all $\gam \in \Gamz$ and $u \in V$.
It is understood that $\Gamz$ is the {\em largest} known group for which $J$ is invariant;
of course $J$ is also invariant under any subgroup of $\Gamz$.
We will use $\Gam$, or $\Gam_i$, to refer to an isotropy subgroup of the ``full'' group $\Gamz$.
Consider the steady-state bifurcation problem $g(\lambda,u)=0$, where
$g(\lambda, u)=\nabla J(\lambda,u)$.  Throughout this paper, the gradient $\nabla$ acts on the $u$ component.
The solutions to $g(\lam, u) = 0$ are critical points of
$J$, so we use the terms ``solution" and ``critical point"
interchangeably.
Note that $g:\R\times V\to V$ is a family of
$\Gamz-$equivariant gradient operators on $V$.  That is,
$g(\lambda,\gamma\cdot u) = \gamma\cdot g(\lambda,u)$.
For our PDE, $\Gamma_0 = \DZ$.
In the numerical implementation, $V = \R^M \cong B_M$ and $g$
is defined in (\ref{JprimeDef}).

%The specific calculations for this case are the material of the next subsection.
We define a {\it branch of solutions} to be a connected component of
$\{ (\lambda,u)\in \R\times L^2(\Omega) \mid g(\lam, u) = 0, \ \stab(u) = \Gam \}$,
where $\Gam$ is called the isotropy subgroup, or symmetry, of the branch.
A branch of solutions $B_1$ has a {\em symmetry-breaking bifurcation} at the {\em bifurcation point}
$(\lam^*, u^*) \in B_1$ if a branch of solutions, $B_2$, with a different symmetry,
has $(\lam^*, u^*)$  as a limit point
but $(\lam^*, u^*) \notin B_2$.  We say that branch $B_2$ is {\em created} at this bifurcation, and often
refer to $B_1$ as the {\em mother branch} and $B_2$ as the {\em daughter branch}.
The symmetry of the daughter branch is always a proper
subgroup of the symmetry of the mother branch.
That is, the daughter has less symmetry than the mother.

The main tool for finding bifurcation points is the Hessian of the
energy functional, $h$.
%the linear map $h(\lam, u): V \rightarrow V$
%defined by $h(\lam, u) = D^2 J(\lam, u) = D_u g(\lam, u)$. The
%Hessian represents $J''(u)$ (for $\lam$ fixed) in that $J''(u)(v, w)
%= \langle h(u) v, w \rangle$ for all $v, w \in V$. The Hessian is a
%symmetric linear map, i.e., $\langle h(u) v, w \rangle = \langle v,
%h(u) w \rangle$. In Section \ref{GNGA_section}, the vector $g$ and
%the matrix $h$ represent these objects in the coefficient space
%$B_M$.
If $(\lam^*, u^*)$ is a bifurcation point, then $h(\lam^*,
u^*)$ is not invertible, since otherwise the implicit function
theorem would guarantee the existence of a unique local solution
branch.
The {\em Morse index} (MI) of a critical point $(\lam, u)$
is defined to be the number of negative eigenvalues of $h(\lam, u) = D^2 J (\lam, u)$,
provided no eigenvalue is 0.
The Hessian is symmetric, so all of its eigenvalues are real.
The MI on a branch of solutions typically changes at a bifurcation point.
\begin{example}
The trivial solution to (\ref{pde}, \ref{nonlinearity}) is $u = 0$,
and the {\it trivial branch} is $\{ (\lam, 0) \mid \lam \in \R \}$.
Since $h(\lam, 0)(v) = \Delta v + \lam v$,
the bifurcation points of the trivial branch are $( \lam_i, 0)$, where $\lam_i, i \in \N$,
are the eigenvalues (\ref{evals}).
% of the Laplacian on $\Om$ with 0 Dirichlet boundary condition.
If $\lam_i < \lam < \lam_{i+1}$, then the MI of the trivial solution $(\lam, 0)$ is $i$.
The $i$-th {\it primary branch} is created at the bifurcation point $(\lam_i, 0)$ on the trivial branch.
In cases with double eigenvalues
there are two branches created at the same point in our problem.
For example, the second and third primary branches are created at $\lam_2 = \lam_3$.
Near $(\lam_i, 0)$, the solutions on the $i$-th primary branch are approximately
some constant times the $i$-th
eigenfunction of the Laplacian, $\psi_i$.
\end{example}

We define a {\em degenerate critical point}, or
a {\em degenerate solution}, to be a point $(\lam^*, u^*)$ which satisfies
$g(\lam^*, u^*) = 0$ and $\det h(\lam^*, u^*) = 0$.
Thus, every bifurcation point is a degenerate critical point.
Some degenerate critical points are not bifurcation points. For example,
when a branch folds over and is not monotonic in $\lambda$, the {\em fold point}
is degenerate, but is not a bifurcation point as we have defined it.
(Note that we avoid the term ``saddle-node bifurcation'' since there is really no bifurcation.)

%We will ignore the rare cases where the Morse index is the same on both sides of a degenerate
%critical point, and we will assume that the degenerate critical points are isolated.  In particular,
%we assume that the group $\Gamz$ is finite, e.g., $\DZ$.  If $\Gamz$ is infinite, then some
%critical points lie on manifolds with the same dimension as $\Gamz$.  These critical points
%are always degenerate since infinitesimal motions in $\Gamz$ correspond to eigenvectors
%of $D^2 J(\lam^*, u^*)$ with zero eigenvalue.
%The GNGA can handle such cases (see the works in progress \cite{nst, thompson}).

Let us develop some notation to talk about bifurcations.
Suppose that $(\lam^*, u^*)$ is an isolated degenerate critical point of a $\Gamz$-equivariant system
$g(\lam, u) = 0$.
%where $\Gamz$ is the largest known symmetry group of the system.
Let $\Gamma = \stab(u^*, \Gamz)$, and
define $L := h(\lam^*, u^*)$.
%If we translate $u^*$ to the origin
%the system becomes $g(\lam, u - u^*) = 0$ which is $\Gam$-equivariant but not $\Gamz$-equivariant.
Note that $\Gam$, not $\Gamz$, is important as far as the bifurcation of $(\lam^*, u^*)$ is concerned.
Let $E$ be the null space of the $\Gam$-equivariant operator $L$.  We call $E$ the {\em center eigenspace}.
Let $\Gamma'$ be the point stabilizer of $E$.  The definitions are repeated here for reference:
\begin{equation}
\label{definitions}
\Gamma := \stab(u^*, \Gamz) ,  ~ ~ L := h(\lam^*, u^*), ~ ~ E := N(L), ~ ~ \Gamma^\prime := \pstab(E, \Gam) .
\end{equation}
%where it is understood that $(\lam^*, u^*)$ is an isolated degenerate critical point:
%$\Gamz \cdot u^*$ is a finite set, $g(\lam^*, u^*) = 0$ and $L$ is singular.

If $e \in E$, then $L(e) = 0$ by definition.  For any $\gamma \in \Gamma$,
$\gamma\cdot e \in E$ since the $\Gam$-equivariance of $L$ implies that
$L(\gamma \cdot e) = \gamma \cdot L(e) = 0$.
Hence,
$$
\stab(E, \Gam) = \Gamma .
$$
Note that $\stab(E,\Gam)/\pstab(E, \Gam) = \Gamma/\Gamma'$ acts faithfully on $E$.
In the usual case where $(\lam^*, u^*)$ is a bifurcation point,
not just a degenerate critical point,
we say that $\Gamma/\Gamma'$ is the {\em symmetry group of the bifurcation}, or
that $(\lam^*, u^*)$ undergoes a {\em bifurcation with $\Gamma/\Gamma'$ symmetry}.

In the notation of (\ref{definitions}),
$L$ sends each of the isotypic components
$V^{(j)}_{\Gamma}$ to itself \cite{nss, sternberg, tinkham}.
Barring ``accidental degeneracy,'' the center eigenspace $E$ is a $\Gamma$-irreducible subspace.
Thus, $E$ is typically a subspace of exactly one isotypic
component $V^{(j)}_{\Gamma}$,
and $\dim(E)$ is the dimension $d_j$ of the corresponding
corresponding irreducible representation, $\alpha^{(j)}$.
Furthermore, the point stabilizer of $E$
is the kernel of $\alpha^{(j)}$ and can be computed without knowing $E$.
In summary,
at a generic bifurcation point there is some irreducible representation $\al^{(j)}$
of $\Gam$ such that:
%the following typically holds for the center eigenspace at
%a degenerate critical point with isotropy subgroup $\Gam$:
%$$
%E \mbox{ is } \Gam-\mbox{irreducible, so}
%$$
%$$E \subseteq V^{(j)}_{\Gamma} \mbox{ for some } j , \quad
%\dim(E) = \Delta MI = d_j , \quad \mbox{ and }
%\Gamma' = \{ \gamma \in \Gamma \mid \alpha^{(j)}(\gamma) = I \} .
%$$
$$
E \mbox{ is } \Gam\mbox{-irreducible}, \quad E \subseteq V^{(j)}_{\Gamma}, \quad
\dim(E) = \Delta MI = d_j , \quad
\Gamma' = \{ \gamma \in \Gamma \mid \alpha^{(j)}(\gamma) = I \} .
$$
Accidental degeneracy is discussed in \cite{nss, sternberg, tinkham}.
We did not encounter any accidental degeneracy in our numerical investigation of (\ref{pde}, \ref{nonlinearity}),
so we will not discuss it further here.

We finally have the background to describe the bifurcations which occur in equivariant systems.
The goal is to predict what solutions will be created
at each of the symmetry breaking bifurcations, and know what vectors in $E$ to use to
start these branches using the pmGNGA.
While such a prediction is impossible for some complicated groups, we can determine how
to follow all of the bifurcating branches in the system
(\ref{pde}, \ref{nonlinearity}).
We follow the treatment and notation of \cite{vol2, tsp}.
At a symmetry-breaking bifurcation we can translate $(\lam^*, u^*)$ to the origin, and
we could, in principle,
do an equivariant Liapunov-Schmidt reduction or center manifold reduction
to obtain reduced bifurcation equations $\tilde{g}: \R \times E \rightarrow E$ where
$\tilde{g}(0,0) = 0$, $D\tilde{g}(0,0) = 0$, and $\tilde{g}$ is $\Gamma := \stab(u^*)$-equivariant.
It is important to realize that we do not actually need to perform the Liapunov-Schmidt reduction.

The most powerful tool for understanding symmetry breaking bifurcations is the Equivariant Branching Lemma.
Recall that absolutely irreducible representations were defined in Example \ref{isotypicZ6}.
See \cite{chossat, vol2, tsp} for a thorough discussion of the Equivariant Branching Lemma,
including further references.
\begin{thm}{\bf Equivariant Branching Lemma (EBL)}
Suppose $\Gamma$ acts absolutely irreducibly on the space $E$, and
let $\tilde g: \R \times E \rightarrow E$ be $\Gamma$-equivariant.
Assume that $\Gamma$ acts nontrivially, so $\tilde g(\lam, 0) = 0$.
Since $\Gam$ acts absolutely irreducibly,
$D\tilde g(\lam, 0) = c(\lam) I_d$ for some function $c: \R \rightarrow \R$,
where $I_d$ is the identity matrix of size $d = \dim(E)$.
Assume that $c(0) = 0$ and $c'(0) \neq 0$.
Let $\Sigma$ be an isotropy subgroup of the $\Gamma$ action on $E$
with $\dim \fix(\Sigma, E) = 1$.
Then there are at least two solution branches of $\tilde g(\lam, u) = 0$
with isotropy subgroup $\Sigma$ created at $(0,0)$.
\end{thm}

The EBL, combined with Liapunov-Schmidt theory,
implies that there are at least two solution branches of the full problem $g(\lam, u) = 0$
with isotropy subgroup $\Sigma$
created at the bifurcation point $(\lam^*, u^*)$.
We call these newly created branches {\em EBL branches} since their existence can be predicted by
the EBL.
Other branches created at a bifurcation are called {\em non-EBL branches}.

Following \cite{chossat, vol2, tsp}, we define a {\em maximal isotropy subgroup}
of a $\Gam$ action on $V$ to be an isotropy subgroup $\Sig \neq \Gam$ with the property
that if $\Theta$ is an isotropy subgroup such that $\Sig \leq \Theta$,
then $\Theta = \Sig$ or $\Theta = \Gam$.  In other words, a maximal isotropy subgroup
is a maximal proper isotropy subgroup.
%In the notation of the EBL,
If $\dim(\fix(\Sig, E)) = 1$, then $\Sig$ is
a maximal isotropy subgroup of the $\Gam$ action on $E$.  The converse, however, is not true.

In gradient systems, for example the PDE (\ref{pde}), more can be said.  If $\Sigma$ is any
maximal isotropy subgroup of the $\Gamma$ action on $E$, then there is typically a solution branch created
at the bifurcation with isotropy subgroup $\Sigma$.  If $\dim \fix(\Sigma, E) \geq 2$, the branch created
is an example of a non-EBL branch.
See \cite{smoller} for a discussion of bifurcations in gradient systems.

By Proposition \ref{fps},
the effective symmetry group of $\tilde g$,
restricted to $\fix(\Sigma, E)$, is $N_{\Gamma}(\Sigma)/\Sigma$.
This effective symmetry group determines
how solutions with symmetry $\Sig$ bifurcate.

\begin{example}
\label{G1bifs}
\begin{figure} % figure bifD6
\scalebox{1}{
\xymatrix@C=10pt{
\langle \rho,\sigma,\tau \rangle
&
\langle \rho,\sigma,\tau \rangle \ar[d]
&
\langle \rho,\sigma,\tau \rangle \ar[d]
&
\langle \rho,\sigma,\tau \rangle \ar[d]
&
\langle \rho,\sigma,\tau \rangle \ar[d]
&
&
\langle \rho,\sigma,\tau \rangle \ar[dl] \ar[rd]
&
&
\\
&
\langle \rho \rangle
&
\langle \rho^2,\sig \rangle
&
\langle \rho^2,\tau \rangle
&
\langle \sigma,\tau \rangle \ar[d]
&
\langle \sig \rangle \ar[rd]
&
&
\langle \tau \rangle \ar[dl]
&
\\
&&&&
\langle \rho^3 \rangle &
&
\langle 1 \rangle  &
         &
}
}
\caption{
\label{bifD6}
Diagrams of the six isotropy lattices for the actions of $\D_6 = \langle \rho,\sigma,\tau \rangle$
on each of the six isotypic components $V^{(j)}$ of the
$\D_6$ action on $L^2(\Om)$.
This describes the six possibilities (barring accidental degeneracy) for
the $\D_6$ action on the center eigenspace $E$ at a degenerate critical point.
}
\end{figure}
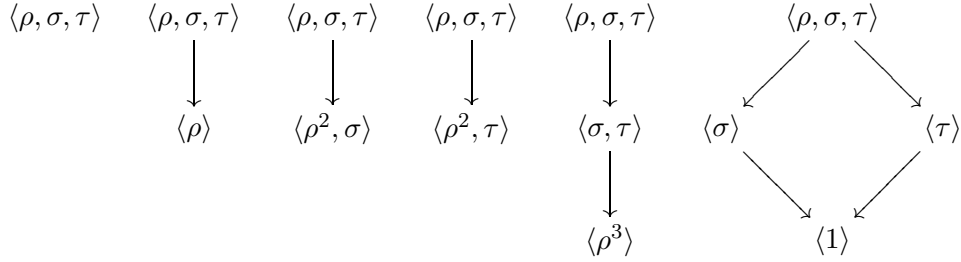
Consider a degenerate critical point with isotropy subgroup $\Gamma_1 = \D_6 = \langle \rho,\sigma,\tau \rangle$.
Barring accidental degeneracy, the center eigenspace $E$ is a subspace of one of the 6 isotypic components of the
$\D_6$ action on $L^2(\Omega)$ described in Example \ref{isotypicD6}.
Figure \ref{bifD6} shows the lattice of isotropy subgroups for $\D_6$
acting on each of these 6 isotypic components $V^{(j)}$.
These 6 cases can be distinguished by determining
which isotypic component an arbitrary eigenfunction in $E$ belongs to.
We shall go through each of these six cases, and describe the resulting bifurcation.
Recall that $\Gam = \Gam_1 =\D_6$ for each of these six cases, and
$\Gam'=\pstab(E, \Gam)$.
$$
\begin{aligned}
E \subseteq V^{(1)} & \Rightarrow &&\Gamma' = \Gam_1 = \langle \rho, \sig, \tau \rangle,
&&\dim E = 1, &&\Gam/{\Gamma}' \cong \langle 1 \rangle \\
E \subseteq V^{(2)} & \Rightarrow &&\Gam' = \Gam_{13} = \langle \rho \rangle,
&&\dim E = 1, && \Gam/\Gam' \cong \Z_2 \\
E \subseteq V^{(3)} & \Rightarrow &&\Gam' = \Gam_{9} = \langle \rho^2,\sig \rangle,
&&\dim E = 1, && \Gam/\Gam' \cong \Z_2 \\
E \subseteq V^{(4)} & \Rightarrow &&\Gam' = \Gam_{10}= \langle \rho^2,\tau \rangle,
&&\dim E = 1, && \Gam/\Gam' \cong \Z_2 \\
E \subseteq V^{(5)} & \Rightarrow &&\Gam' = \Gam_{19}= \langle \rho^3 \rangle,
&&\dim E = 2, && \Gam/\Gam' \cong \D_3 \\
E \subseteq V^{(6)} & \Rightarrow &&\Gam' = \Gam_{22}= \langle 1 \rangle,
&&\dim E = 2, && \Gam/\Gam' \cong \D_6 . \\
\end{aligned}
$$
The first case, $E \subseteq V^{(1)} = \fix(\Gam_1, L^2(\Omega))$,
does not lead to a symmetry-breaking bifurcation.
The $\D_6$ action on $E$ is trivial, so the EBL does not apply.
The degenerate critical point $(u^*, \lam^*)$ is typically a fold
point (or saddle-node), not a bifurcation point.
In the neighborhood of the fold point there is only one solution branch,
with isotropy subgroup $\Gam_1$, and the branch lies to one side of $\lam = \lam^*$ or the other.

The next three cases, with $\Gam/\Gam' \cong \Z_2$ symmetry, are called {\em pitchfork bifurcations}.
Clearly, the only maximal isotropy subgroup is $\Gam'$ in each case, and the EBL applies.
The effective symmetry group acting on $E$ is $\Z_2$,
so there are two conjugate solution branches created at the bifurcation.
In the branch following code we follow
one of these branches using the pmGNGA
starting with any eigenvector $e \in E$.

The next case, with $E \subseteq V^{(5)}$, is a bifurcation with $\D_3$ symmetry.
The maximal isotropy subgroup $\Gam_5 = \langle \sig, \tau \rangle$ satisfies
$$
\dim \fix(\Gam_5, E) = 1, \mbox{ and } N_{\Gam_1}(\Gam_5)/\Gamma_5 = \langle 1 \rangle .
$$
Our branch following code uses a projection operator
to find an eigenvector $e \in E$ with $\stab(e, \Gam_1)= \Gam_5$.
The pmGNGA using this eigenvector $e$ will follow one of
the solution branches created at the bifurcation, and the
pmGNGA using the negative eigenvector $-e$ will find a branch that is
not conjugate to the first.  Bifurcations with $\D_3$ symmetry are
typically transcritical, and two $\D_3$-orbits of branches are created
at the bifurcation \cite{vol2, tsp}.

The last case, with $E \subseteq V^{(6)}$, is a bifurcation with $\D_6$ symmetry.
There are two maximal symmetry types, the conjugacy classes of
$\Gam_{15}$ and $\Gam_{16}$.  A calculation shows that
$$
\dim \fix(\Gam_{15}, E) = \dim \fix(\Gam_{16}, E) = 1, \mbox{ and }
N_{\Gam_1}(\Gam_{15})/\Gamma_{15} = N_{\Gam_1}(\Gam_{16})/\Gamma_{16} = \Z_2 .
$$
To follow one branch from each of the group orbits of solution branches created at this bifurcation,
it suffices to use the pmGNGA twice, with the eigenvectors $e_1, e_2 \in E$,
where $\stab(e_1, \Gam_1) = \Gam_{15}$ and $\stab(e_2, \Gam_1) = \Gam_{16}$.
It is well-known that these EBL-branches are typically the only branches created
at a bifurcation with $\D_6$ symmetry \cite{vol2, tsp}.
\end{example}
\begin{example}
\label{G13bifs}
\begin{figure} % figure bifD6
\scalebox{1}{
\xymatrix@C=10pt{
\langle \rho \rangle
&
\langle \rho \rangle \ar[d]
&
\langle \rho \rangle \ar[d]
&
\langle \rho \rangle \ar[d]
&
\\
&
\langle \rho^2 \rangle
&
\langle \rho^3 \rangle
&
\langle 1 \rangle
&
\\
}
}
\caption{
\label{bifG13}
The diagrams of the four isotropy lattices for the actions of $\Gam_{13} = \langle \rho \rangle$
on each of the four isotypic components $V^{(j)}_{\langle \rho \rangle}$ of the
$\Gam_{13}$ action on $L^2(\Om)$.
This describes the four possibilities (barring accidental degeneracy) for
the $\Gam_{13}$ action on the center eigenspace $E$ at a degenerate critical point.
}
\end{figure}
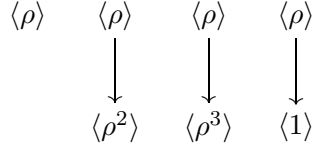
Consider a degenerate critical point with isotropy subgroup $\Gamma_{13} =\langle \rho \rangle \cong \Z_6$.
Barring accidental degeneracy, the center eigenspace $E$ is a subspace of one of the 4 isotypic components
$V^{(j)}_{\langle \rho \rangle}$ defined in Example \ref{isotypicZ6}.
Figure \ref{bifG13}
shows the lattice of isotropy subgroups for $\Gamma_{13}$ acting on each of these
4 isotypic components.
Recall that $\Gam = \Gam_{13} =\langle \rho \rangle$ for each of these cases,
and the minimal isotropy subgroup is $\Gam'=\pstab(E, \Gam)$.
We shall go through each of the four cases, and describe the resulting bifurcation:
$$
\begin{array}{lllll}
\vphantom{A_{B_\frac{C}{D}}}
E \subseteq V^{(1)}_{\langle \rho \rangle} = V^{(1)} \oplus V^{(2)} & \Rightarrow
&\Gamma' = \Gamma_{13} = \langle \rho \rangle, &\dim E = 1, &\Gam/{\Gamma}' \cong \langle 1 \rangle \\
\vphantom{A_{B_\frac{C}{D}}}
E \subseteq V^{(2)}_{\langle \rho \rangle} = V^{(3)} \oplus V^{(4)} & \Rightarrow
&\Gamma' = \Gam_{21} = \langle \rho^2 \rangle,
&\dim E = 1, & \Gam/\Gam' \cong \Z_2 \\
\vphantom{A_{B_\frac{C}{D}}}
E \subseteq V^{(3)}_{\langle \rho \rangle} = V^{(5)}
& \Rightarrow &\Gam' = \Gam_{19} = \langle \rho^3 \rangle,
&\dim E = 2, & \Gam/\Gam' \cong \Z_3 \\
\vphantom{A_{B_\frac{C}{D}}}
E \subseteq V^{(4)}_{\langle \rho \rangle} = V^{(6)}
& \Rightarrow &\Gam' = \Gam_{22}= \langle 1 \rangle,
&\dim E = 2, & \Gam/\Gam' \cong \Z_6 . \\
\end{array}
$$
The first two cases are analogous to the
first two cases in Example \ref{G1bifs}.
When $\Gam/\Gam' \cong \langle 1 \rangle$ there is a fold point, but no symmetry breaking bifurcation.
There is a pitchfork bifurcation when $\Gam/\Gam' \cong\Z_2$.
The next two cases are interesting because $\Gam_{13}$ does not act absolutely irreducibly
on $E$, and the EBL does not apply.  In both cases $\Gam'$ is a maximal isotropy subgroup.

In the third case, where $E \subseteq V^{(3)}_{\langle \rho \rangle}= V^{(5)}$,
every eigenfunction in the 2-dimensional $E$ has isotropy subgroup $\Gam_{19}$.
Since we have a gradient system, we know that solution branches with isotropy subgroup $\Gam_{19}$
are created at this bifurcation with $\Z_3$ symmetry.  The bifurcation is well-understood,
and it looks like a bifurcation
with $\D_3$ symmetry, except that the ``angle'' of the bifurcating solutions in the $E$ plane is arbitrary.
This means that trial and error is needed, in general, to find eigenfunctions in $E$
for which the pmGNGA will converge.
%Eigenfunctions with several angles are used to start the pmGNGA.
%If the angle is wrong, Newton's method will not converge to find the first solution on the branch.
If a branch is found for a starting eigenfunction $e$, then the eigenfunction $-e$ is used to find the other
solution branch.

In the fourth case, where $E \subseteq V^{(4)}_{\langle \rho \rangle}= V^{(6)}$,
every eigenfunction in $E$ has the same isotropy subgroup: $\Gam_{22} = \langle 1 \rangle$.
Gradient bifurcations with $\Z_6$ symmetry look like bifurcations with $\D_6$ symmetry,
except that the angle in the $E$ plane is arbitrary.
Again, trial and error is needed to find starting eigenfunctions for which the pmGNGA converges.
\end{example}

\subsection{The Bifurcation Digraph}
A calculation
similar to those summarized in Examples \ref{G1bifs} and \ref{G13bifs} was done for each of the isotropy subgroups
of the $\DZ$ action on $L^2(\Om)$.
The calculations were done by hand, and verified with GAP.
There are 59 generic symmetry-breaking bifurcations, one for each
isotypic component $V^{(j)}_{\Gam_i}$ on which $\Gam_i$ acts
nontrivially.
The amount of information is overwhelming, so we display the essential results
in what we call a bifurcation digraph.
\begin{definition}
\label{digraphdef}
The {\em bifurcation digraph} of the $\Gamz$ action on a real vector space $V$
is a directed graph with labelled arrows. The vertices are the symmetry types
(equivalence classes of isotropy subgroups).
Given $\Sig \leq \Gam$, two isotropy subgroups of the $\Gamz$ action on $V$,
we draw an arrow from $[\Gam]$ to $[\Sig]$
iff $\Sig$ is a maximal isotropy subgroup of the
$\Gam$ action on some isotypic component $V^{(j)}_{\Gam}$ of $V$.
Each arrow has the label $\Gam/\Gam'$, where $\Gam'$ is the kernel
of the $\Gam$ action on $V^{(j)}_\Gam$.
Furthermore, each arrow is either
solid, dashed or dotted. The arrow is
\begin{align}
\mbox{solid if }
& \dim \fix(\Sigma, E) = 1 \mbox{ and } N_{\Gam}(\Sigma)/\Sigma = \Z_2, \nonumber \\
\mbox{dashed if }
&\dim \fix(\Sigma, E) = 1 \mbox{ and } N_{\Gam}(\Sigma)/\Sigma = \langle 1 \rangle, \mbox{ and}\nonumber\\
\mbox{dotted if }
&\dim \fix(\Sigma, E) \geq 2 \nonumber,
\end{align}
where $E$ is any irreducible subspace contained in $V^{(j)}_\Gam$.
\end{definition}
Note that if $\dim \fix(\Sigma, E) = 1$, then $ N_{\Gamma}(\Sigma)/\Sigma$ is either
$\Z_2$ or $\langle 1 \rangle$, since these are the only linear group actions on $E \cong \R^1$.
Thus, the three arrow types (solid, dashed, and dotted) exhaust all possibilities.

\begin{thm}
\label{thm1}
For a given $\Gamz$ action on $V$, every arrow in the diagram of the isotropy lattice
is an arrow in the bifurcation digraph.
\end{thm}
\begin{proof}
Suppose $[ \Gam ] \rightarrow [ \Sig ]$ is an arrow in the diagram of the isotropy lattice.
Then some $\Sig^* \in [\Sig]$ is a maximal isotropy subgroup of the $\Gam$ action on $V$.
Choose $u^* \in V$ such that $\stab(u^*, \Gam) = \Sig^*$.  Such a $u^*$ exists
since $\Sig^*$ is an isotropy subgroup.
Now consider the isotypic decomposition $\{V^{(j)}_\Gam \}_{j \in J}$ of $V$.
We can write $u^* = \sum_{j \in J} u^{(j)}$, where $u^{(j)} \in V^{(j)}_\Gam$ are
uniquely determined.  Let $\gam$ be any element of $\Sig^*$.
Then $\gam \cdot u^* = \sum_{j \in J} \gam \cdot u^{(j)} = u^*$.
Since each of the components $V^{(j)}_\Gam$ is $\Gam$-invariant,
$\gam \cdot u^{(j) } = u^{(j)}$ for each $j \in J$.
Thus $\Sig^* \leq \stab(u^{(j)}, \Gam)$ for each $j \in J$.
Either $\stab(u^{(j)}, \Gam) = \Gam$ or $\stab(u^{(j)}, \Gam) = \Sig^*$,
since $\Sig^*$ is a maximal isotropy subgroup of the $\Gam$ action on $V$.
If $\stab(u^{(j)}, \Gam) = \Gam$ for all $j \in J$, then $\stab(u^*, \Gam) = \Gam$.
But $\stab(u^*, \Gam) \neq \Gam$, so $\stab(u^{(j)}, \Gam) = \Sig^*$
for some $j \in J$, and $\Sig^*$ is a maximal isotropy subgroup of the $\Gam$
action on this component $V^{(j)}_\Gam$ of $V$.
Therefore the bifurcation digraph has an
arrow from $[ \Gam]$ to $[ \Sig^* ] = [\Sig]$.
\end{proof}
Theorem \ref{thm1} says that the bifurcation digraph is an extension of
the diagram of the isotropy lattice.
The bifurcation digraph has more arrows, in general.
As with the lattice of isotropy subgroups, we usually draw a single element
$\Gam$ of the equivalence class $[\Gam]$ for each vertex of the bifurcation digraph.

An arrow from $\Gam$ to $\Sig$ in the bifurcation digraph
indicates that a $\Gamz$-equivariant gradient system $g(\lam, u) = 0$
can have a generic symmetry-breaking bifurcation
where a mother branch with isotropy subgroup $\Gam$ creates a daughter
branch with isotropy subgroup $\Sig$.
The symmetry group of the bifurcation is $\Gam/\Gam'$, and the center eigenspace at the bifurcation
point is the $\Gam$-irreducible space $E$.  The information encoded in the label and arrow type
is used by the heuristics of our branch-following algorithm.
A solid arrow indicates that every $e$ in the one-dimensional space $\fix(\Sigma, E)$
satisfies $\gamma \cdot e = - e$ for some $\gam \in \Sigma$.
Thus, there is typically a pitchfork bifurcation in the space $\fix(\Sigma, E)$.
A dashed arrow indicates that $\gamma \cdot e = e$ for all $e \in \fix(\Sigma, E)$ and $\gamma \in \Sig$.
Thus, the daughter branches bifurcating in the directions $e$ and $-e$ are not conjugate.
A dotted arrow indicates that the EBL does not apply to this
bifurcation.  As mentioned above, branching of solutions corresponding to a dotted arrow
is generic in gradient systems \cite{smoller, vol2}.

A condensed bifurcation digraph for the $\DZ$ action on $L^2(\Omega)$ is shown in Figure \ref{digraph}.
The calculations for the directed edges coming from $\Gamma_1$ and $\Gamma_{13}$
are described in examples \ref{G1bifs} ane \ref{G13bifs}, respectively.
The digraph has 65 directed edges, but there are only
5 possibilities for the symmetry group of the bifurcation:
$\Gamma/\Gamma' = \Z_2$, $\Z_3$, $\Z_6$, $\D_3$, or $\D_6$.
The symmetry-breaking bifurcation with each of these symmetries is well understood \cite{vol2, tsp},
and each is described briefly in Example \ref{G1bifs} or \ref{G13bifs}.
This digraph is of great help in writing an automated code for branch following.

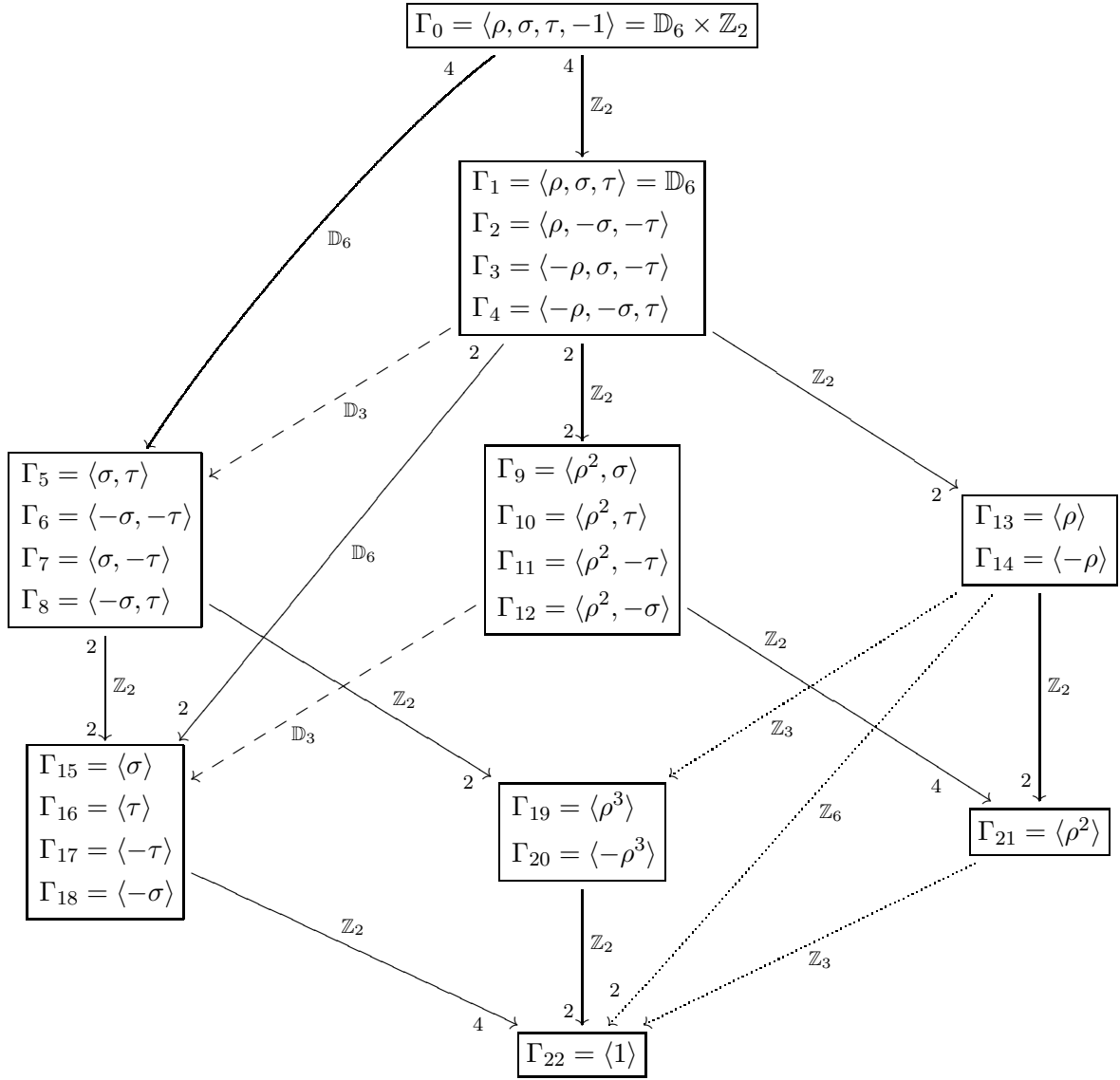
\begin{figure}
\scalebox{1}{
\xymatrix@R=40pt@C=75pt{
 &\fbox{$ \Gam_0 = \langle \rho, \sig, \tau, -1 \rangle = \DZ
$}
  \ar@{->}_<>(.1)4[d]^<>(.5){\Z_2}
  \ar@(ul,u)@{->}_<>(.1)4[ddl]^{\D_6}
\\
&
%   \fboxrule=1.5pt
   \fbox{$
   \begin{aligned}
   &\Gam_1  = \langle\rho,\sigma,\tau\rangle = \D_6 \\ % &  V^{(1)}\\
   &\Gam_2  = \langle\rho,-\sigma,-\tau\rangle \\ % &  V^{(2)}\\
   &\Gam_3  = \langle-\rho,\sigma,-\tau\rangle \\ % &  V^{(3)}\\
   &\Gam_4  = \langle-\rho,-\sigma,\tau\rangle\\ % &  V^{(4)}\\
   \end{aligned}
   $}
    \ar@{->}_<>(.1)2_<>(.9)2[d]^{\Z_2}
    \ar@{->}_<>(.95)2[rd]^{\Z_2}
    \ar@{-->}[ld]^{\D_3}
    \ar@{->}_<>(.05)2_<>(.95)2[ldd]^{\D_6}
\\
%   \fboxrule=1.5pt
   \fbox{$
   \begin{aligned}
   &\Gam_5  = \langle\sigma,\tau\rangle   \\ % &  V^{(1)}\oplus V^{(5)}_{1} \\
   &\Gam_6  = \langle-\sigma,-\tau\rangle \\ % &  V^{(2)}\oplus V^{(5)}_{2}\\
   &\Gam_7  = \langle\sigma,-\tau\rangle  \\ % &  V^{(3)}\oplus V^{(6)}_{1}\\
   &\Gam_8  = \langle-\sigma,\tau\rangle  \\ % &  V^{(4)}\oplus V^{(6)}_{2}\\
   \end{aligned}
   $}
   \ar@{->}_<>(.1)2_<>(.9)2[d]^{\Z_2}
   \ar@{->}_<>(.95)2[rd]^<>(.65){\Z_2}
 &
   \fbox{$
   \begin{aligned}
   &\Gam_9     = \langle\rho^2,\sigma\rangle  \\ % &  V^{(1)}\oplus V^{(3)}\\
   &\Gam_{10}  = \langle\rho^2,\tau\rangle    \\ % &  V^{(1)}\oplus V^{(4)}\\
   &\Gam_{11}  = \langle\rho^2,-\tau\rangle   \\ % &  V^{(2)}\oplus V^{(3)}\\
   &\Gam_{12}  = \langle\rho^2,-\sigma\rangle \\ % &  V^{(2)}\oplus V^{(4)}\\
   \end{aligned}
   $}
    \ar@{->}_<>(.85)4[rd]^<>(.25){\Z_2}
    \ar@{-->}[ld]^<>(.65){\D_3}
&
   \fbox{$
   \begin{aligned}
  &\Gam_{13}  = \langle\rho\rangle     \\ %       &  V^{(1)}\oplus V^{(2)} \\
  &\Gam_{14}  = \langle-\rho\rangle    \\ %       &  V^{(3)}\oplus V^{(4)}\\
    \end{aligned}
    $}
    \ar@{->}_<>(.9)2[d]^{\Z_2}
    \ar@{.>}[ld]^<>(.65){\Z_3}
    \ar@{.>}_<>(.95)2[ldd]^{\Z_6}
\\
 \fbox{$
    \begin{aligned}
   &\Gam_{15}  = \langle\sigma\rangle  \\ %  &  V^{(1)}\oplus V^{(3)}\oplus V^{(5)}_1\oplus V^{(6)}_{1}\\
   &\Gam_{16}  = \langle\tau\rangle    \\ %  &  V^{(1)}\oplus V^{(4})\oplus V^{(5)}_1\oplus V^{(6)}_{2}\\
   &\Gam_{17}  = \langle-\tau\rangle   \\ %  &  V^{(2)}\oplus V^{(3)}\oplus V^{(5)}_2\oplus V^{(6)}_{1}\\
   &\Gam_{18}  = \langle-\sigma\rangle \\ %  &  V^{(2)}\oplus V^{(4)}\oplus V^{(5)}_2\oplus V^{(6)}_{2}\\
   \end{aligned}
 $}
    \ar@{->}_<>(.9)4[rd]^{\Z_2}
&
   \fbox{$
   \begin{aligned}
   &\Gam_{19}  = \langle\rho^3\rangle  \\ % &  V^{(1)}\oplus V^{(2)}\oplus V^{(5)}_1\oplus V^{(5)}_{2} \\
   &\Gam_{20}  = \langle-\rho^3\rangle \\ % &  V^{(3)}\oplus V^{(4)}\oplus V^{(5)}_2\oplus V^{(6)}_{2}\\
   \end{aligned}
   $}
    \ar@{->}_<>(.9)2[d]^{\Z_2}
&
   \fbox{$
 \Gam_{21} = \langle\rho^2\rangle % \quad V^{(1)}\oplus V^{(2)}\oplus V^{(3)}\oplus V^{(4)}
   $}
    \ar@{.>}[ld]^{\Z_3}
\\
&
 \fbox{$
 \Gam_{22}= \langle 1 \rangle % \quad U
 $}
} }
\caption{The bifurcation digraph for the $\DZ$ action on $L^2(\Om)$
extends the diagram of the isotropy lattice.
The digraph shown is condensed as in
Figure~\ref{lattice}.
The arrows indicate generic symmetry breaking bifurcations.
The Morse index of the mother branch
changes by 1 at bifurcations with $\Z_2$ symmetry, and it changes by 2 at
all other bifurcations shown here.
}
\label{digraph}
\end{figure}
In our problem
%, where $\Gamz = \DZ$ acts on $L^2(\Om)$,
the label $\Gam/\Gam'$ and arrow type are sufficient to characterize the bifurcation completely.
For more complicated groups, the label may need to contain more information about the
action of $\Gam$ on $E$.  For example the label $\Gam/\Gam' = {\mathbb S}_4$ would be
ambiguous, since ${\mathbb S}_4$ has two faithful irreducible representations
with different lattices of isotropy subgroups.
\end{section}

\begin{section}{Symmetry and Computational Efficiency.}
\label{sce_section}

Several modifications of the GNGA (\ref{algorithm}) take advantage
of symmetry to speed up the calculations.
The symmetry forces many of the components of the gradient and Hessian to be zero.
We identified these zero components and avoided doing the time-consuming numerical integrations
to compute them.
At the start of the C++ program, the isotropy subgroup, $\Gam_i$, of the
initial guess is computed.  Recall that there are $M$ modes in the Galerkin space $B_M$,
so $\dim (B_M) = M$.
Define $M_i := \dim( \fix(\Gam_i, B_M))$.
We chose the representatives $\Gam_i$ within each conjugacy class so that
%In our PDE with $\DZ$ symmetry,
$\fix(\Gam_i, B_M)$ is a coordinate subspace of $B_M$.
Thus, $M-M_i$ components of the gradient $g(\lam,u)$ are zero if $\fix(u) = \Gam_i$.
The numerical integrations in (\ref{JprimeDef}) are done only for the $M_i$ potentially nonzero components of $g$.
%The numerical integrations are not done for the components of $g$ which must be zero.
Similarly, $M_i(M_i+1)/2$ rather than $M(M+1)/2$ numerical integrations are needed to compute
the part of the Hessian matrix $h$ needed by the GNGA algorithm:
The numerical integrations in (\ref{JprimePrimeDef}) are
done only if $\psi_j$ and $\psi_k$ are both in
$\fix(\Gam_i, B_M)$.
The system $h \chi = g$ for the Newton step $\chi$ reduces to a system of $M_i$ equations
in $M_i$ unknowns.
%, but
%$\lam_j \, \delta_{j k}$
%is included for all $j \in \{1, \ldots , M \}$
%so that the Hessian is nonsingular.
%For simplicity in implementing the GNGA,
%we solve the linear system $h \chi = g$ with $M$ equations and $M$ unknowns.
%In future work we plan to solve a reduced system with $M_i$ equations and $M_i$
%unknowns.
After Newton's method converges to a solution, the full Hessian needs to be calculated
in order to compute the MI.  Here, too, we can take advantage of the symmetry:
Since $h$ is $\Gam_i\,$-equivariant,
%the Hessian is block diagonal in the isotypic decomposition of $B_M$ for the $\Gam_i$ action.
%In other words,
$h_{j\,k} = 0$ if $\psi_j$ and $\psi_k$
are in different isotypic components $V^{(j)}_{\Gam_i}$ of $B_M$.

As an example, consider the execution time for approximating a solution with
$\Gam_1$ symmetry using $M = 300$ modes and a level $\ell = 5$ grid on a 1GHz PC.
Our C++ code uses only $M_1 = 30$ modes, and takes about 1.5 seconds per Newton step,
compared to 44 seconds when the symmetry speedup is not implemented.
\end{section}

%\label{branch_section}
% \include{branch}

\begin{section}{Automated Branch Following.}
\label{branch_section}
The branch following code is a complex collection of about a dozen Perl scripts,
{\it Mathematica} and Gnuplot scripts, and a C++ program.
These programs write and call each other fully automatically and communicate through output
files, pipes and command line arguments.  A complete bifurcation
diagram can be produced by a single call to the main Perl script.

\begin{figure}
\begin{center}
\rotatebox{-90}{\scalebox{.93}{\includegraphics{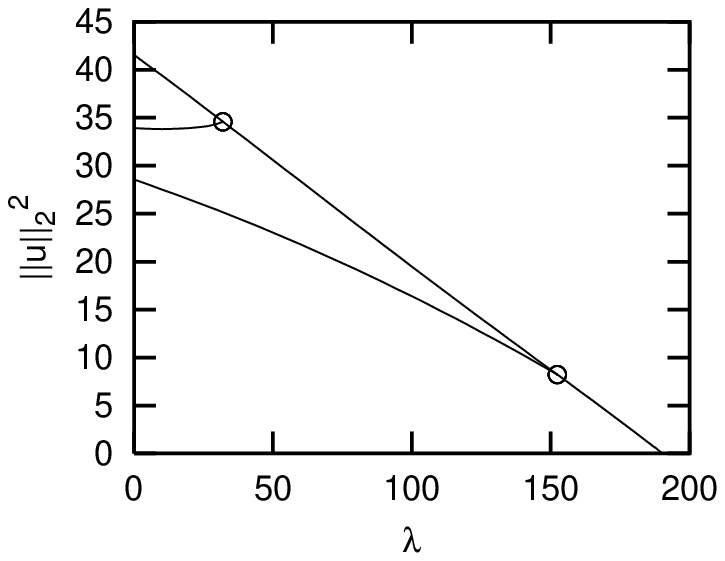}}}
\rotatebox{-90}{\scalebox{.93}{\includegraphics{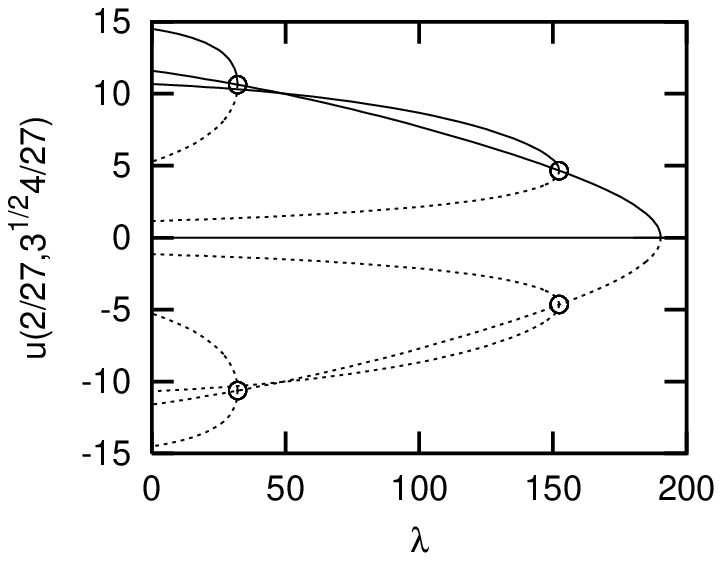}}}
\caption{
Bifurcation diagrams of the sixth primary branch (which bifurcates from $\lam_6$), showing
$||u||_2^2$ and $u(2/27, 4 \sqrt{3}/27)$ as a function of $\lam$.
Since $||u||_2^2$ is a $\D_6 \times \Z_2$-invariant function of $u$, each
group orbit of solution branches is shown as one curve on the left.  The disadvantage
of plotting $||u||_2^2$  is that the curves in many bifurcation diagrams are not well separated.
The point $(2/27, 4 \sqrt{3}/27)$ is
not on any of the reflection axes of the snowflake region.
There are 2 primary
branches with symmetry $S_1$, four secondary branches with
symmetry $S_9$, and four secondary branches with symmetry
$S_{10}$.  Our choice for the bifurcation diagrams in this paper combines the
advantages of both views: $u(2/27, 4 \sqrt{3}/27)$ is plotted as a function of $\lam$ for
exactly one branch (the solid lines) from each group orbit.
Unless indicated otherwise, all figures were produced with level $\ell = 5$ and $M = 300$ modes.
}
\label{fig.verbose}
\end{center}
\end{figure}

Two choices for the function of $u$ plotted vs. $\lam$ are shown
in Figure \ref{fig.verbose}.
In most bifurcation diagrams we plot approximate solutions
$u$ evaluated at a generic point $(2/27 , 4 \sqrt{3} / 27)$ (the big dot in Figure \ref{grid})
versus the parameter
$\lambda$;
other choices for the vertical axis such as $J(u)$ or $\|u\|_\infty$
lead to less visible separation of branches.
Two conjugate solutions can have different values at
the generic point, but since our program follows only one
branch in each group orbit this does not cause a problem.

The C++ program implements the GNGA algorithm. Its input
is a vector of coefficients $a \in \R^M$ for an initial guess in Newton's method,
an interval for $\lambda$, a stepsize for $\lambda$
%, symmetry restrictions
and several other parameters such as the grid level.
% and the
%number of modes used in the expansion of solutions.
It finds solutions on a
single branch of the bifurcation diagram. Every solution is written as a single line
in an output file. This line contains all the information about the solution,
and can be used to write an input file for a subsequent
call to the same C++ program.

The C++ program finds one branch (referred to as the main branch) and a short segment of
each of the daughter branches created at bifurcations of the main branch.
The coefficients approximating the first solution on the branch are supplied to
the C++ program.  Newton's method is used to find this first solution, then $\lam$
is incremented and the next solution is found.
The program attempts to follow the main branch all the way to
the final $\lam$, usually 0.
Heuristics are used to double or halve the $\lam$ stepsize when needed, keeping the stepsize
in the interval from the initial stepsize (input to the C++ program) to $1/32$ of the initial stepsize.
For example, the stepsize is halved
if Newton's method does not converge, if it converges to a solution with
the wrong symmetry, or if more than one bifurcation is detected in one $\lam$ step.

The Morse index is computed at each $\lam$ value on the main branch.
When the MI changes a subroutine is called to handle the bifurcation before the main
branch is continued.
If the MI changes from $m_1$ to $m_2$, we define $m = \max\{m_1, m_2\}$.
Then the bifurcation point is approximated
by using the secant method to set the $m$-th eigenvalue of the Hessian $h(u)$ to zero
as a function of $\lam$.  The GNGA is needed at each step of the secant method
to compute $u = u (\lam)$.
We find that the GNGA works well even though we are
approximating a solution for which the Hessian is singular.

After the bifurcation point is approximated, a short segment of each bifurcating branch is computed
and one output file is written for each branch, using Algorithm \ref{follow_branch}.
If the Equivariant Branching Lemma (EBL)
holds, then we know exactly which critical eigenvector to use for each branch.
%Let the coefficient vector of the solution at the bifurcation point be
%$a$, let the normalized critical eigenvector be $e \in \R^M$,
%and let $k$ be defined by $| e_k| \geq | e_i| $ for all $i$.
%We then use the pmGNGA with the initial guess
%$a + t \, e$, keeping the $k$-th component fixed and solving for $\lambda$ and the other
%$M-1$ components of $a$.  We start with $t = 0.1$, but this is decreased if Newton's method
%does not converge.  More points on the bifurcating branch are computed in the same way,
%except that $a$ is the last solution found on the branch.
%This short segment of the bifurcating branch ends when $\lambda$ reaches the bifurcation value $\lam_b$
%minus the stepsize, or when the pmGNGA does not converge even when $t$ is extremely small, or when a maximum
%number of points on the branch is computed.
\begin{algorithm}{(\verb"follow_branch")}
\label{follow_branch}
\begin{enumerate}
\tt
\item Input: bifurcation point $(\lam, a)$,
one
critical eigenvector $e \in \R^M$, \\
and stepsize $\Delta \lam < 0$.
Output: A file is written for one
%The subroutine writes a file with the first part of one
daughter branch.
\item Write $(\lam, a)$ to output file.  Set $t = 0.1$. Set $\lam_b = \lam$.
\item Compute index $k$ so that $| e_k | \geq | e_i |$ for all $i \in \{ 1, \ldots, M \}$.
\item Repeat until $ \lam_b - \lam < \Delta \lam$, or $t < 0.1/32$ or some maximum number
of points have been written to the file.
        \begin{enumerate}
        \item Do the pmGNGA with initial guess $(\lam, a + t \, e)$,
            fixing coefficient $k$.
        \item If Newton's method converges replace $(\lam, a)$ by the solution found and write this point to the file,
               else $t \leftarrow t/2$.
        \end{enumerate}
\end{enumerate}
\end{algorithm}

\begin{figure}
\begin{center}
\rotatebox{-90}{\scalebox{.93}{\includegraphics{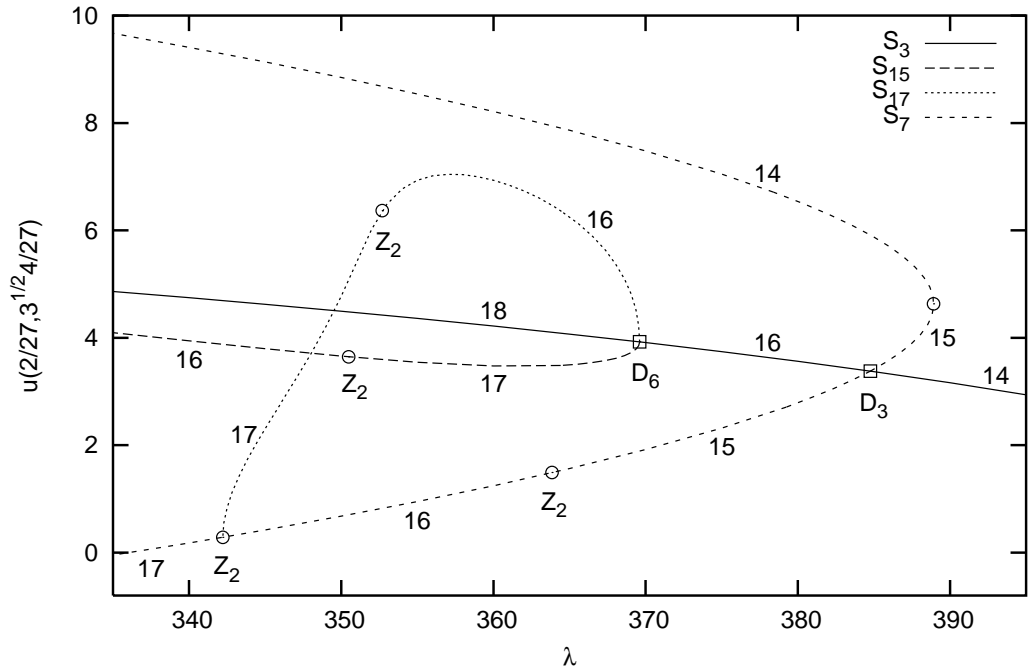}}}
\caption{A partial bifurcation diagram of the 14-th primary branch showing
a $\D_6$, a $\D_3$ and several $\Z_2$ bifurcations.
At the $\D_6$ bifurcation, 12 branches in two different group orbits
are born.
In accordance with Figure \ref{fig.verbose}, only two branches are followed and shown
on this bifurcation diagram.
An animation showing the followed branch with symmetry type $S_{15}$ is shown in {\tt s3s15.gif},
and an animation of the followed branch with symmetry type $S_{17}$ is in {\tt s3s17s7.gif}.
Note that this branch with $S_{17}$ symmetry ``dies'' at a bifurcation with $\Z_2$ symmetry,
showing that we cannot always make a consistent distinction between secondary and tertiary branches.
At the $\D_3$ bifurcation,
6 branches in two different group orbits are born.
As before, only two branches are followed.
An animation showing the ``upper'' branch with symmetry type $S_7$, through the bifurcation
point and continuing to the ``lower'' branch with symmetry type $S_7$ is shown in {\tt s7s3s7.gif}.
For clarity, the branches bifurcating from 3 of the $\Z_2$ bifurcations
are not shown.
The numbers next to a branch indicate the MI of the solution.
The MI changes by 2 at a square, and by 1 at a circle.
}
\label{fig.d3_d6}
\end{center}
\end{figure}

\begin{figure}
\begin{center}
\rotatebox{-90}{\scalebox{.93}{\includegraphics{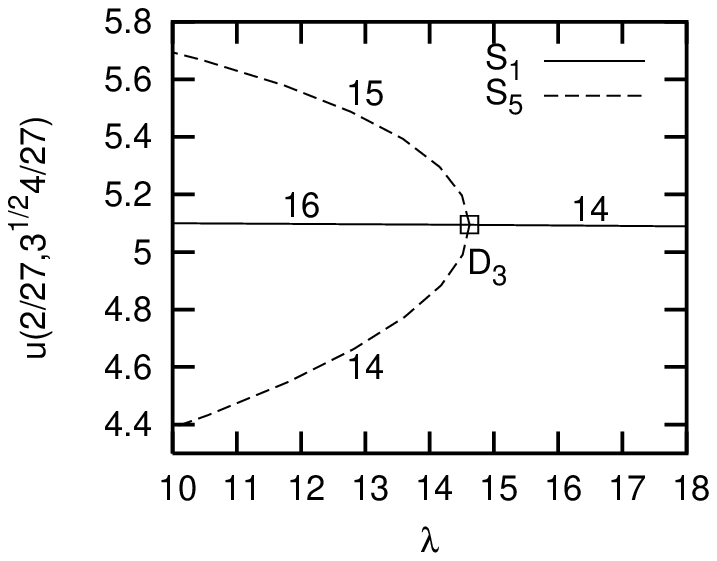}}}
\rotatebox{-90}{\scalebox{.93}{\includegraphics{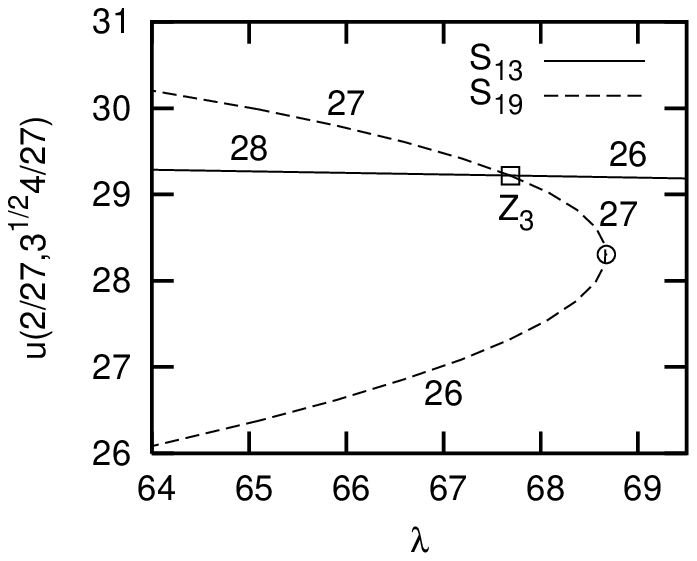}}}
\caption{
The $\D_3$ bifurcation of the 13-th primary branch is on the left. This is the only
observed $\D_3$ bifurcation that is not transcritical.
An animation of the upper branch with symmetry type $S_5$, through the bifurcation
point and continuing with the lower branch is shown in {\tt s5s1s5.gif}.
A $\Z_3$ bifurcation of a daughter of the 24-th primary branch is shown on the right.
The branches created at this bifurcation are not described by the EBL.
An animation of the branches with symmetry type $S_{19}$ is shown in {\tt s19s13s19.gif}.
%The bifurcation digraph in
%Figure~\ref{digraph} indicates that there are possible bifurcations with $\Z_6$ symmetry.
%We did not observed such a bifurcation with level $\ell = 5$ and $M = 300$ modes.
}
\label{fig.d3b}
\end{center}
\end{figure}

Note that the pmGNGA can follow a branch that bifurcates
to the right or the left.  Those that bifurcate to the right usually turn over in a saddle-node
``bifurcation" that does not offer any difficulty for the pmGNGA.
Figures \ref{fig.d3_d6} and \ref{fig.d3b} show several examples of bifurcations.

The EBL does not hold at bifurcations with $\Z_3$ and $\Z_6$ symmetry in our problem, since
the 2-dimensional center eigenspace does not have a 1-dimensional subspace with more symmetry.
Figure \ref{fig.d3b} shows one of the few bifurcations with $\Z_3$ symmetry that we observed.
By good fortune, the branches with symmetry type $S_{19}$
were successfully followed using the same eigenvectors one would choose for a
bifurcation with $\D_3$ symmetry.
A better method for following bifurcating solutions that are
not predicted by the EBL would be to use the pmGNGA with random (normalized) eigenvectors in $E$ repeatedly
until it appears that all equivalence classes of solutions have been found.

The branch following code is called recursively by a main Perl script. Initially, the C++ program follows
the trivial branch on a given $\lambda$ range. This results in an output
file for the trivial branch and another output file for each bifurcating primary
branch. Then the short parts of the primary branches are followed with more calls
to the C++ program. Any bifurcating branch results in a new output file, and
the Perl script makes another call to the C++ program to continue that branch.
The main Perl script's most important job is book keeping. It saves the output files
with distinct names, and calls the branch following code to continue
each of the new branches. The process stops when all the branches are fully followed
within the given $\lam$ range.

In this way, a complete bifurcation diagram is produced by a single invocation
of the main Perl script.  There is no need to guess initial conditions for input
to Newton's method, since the trivial solution is known exactly ($a = 0$) and all
the other solutions are followed automatically.

The main Perl script calls several other smaller scripts. For example, there is a script
which extracts solutions from output files and feeds them to the branch following code as input.
Another script creates Gnuplot scripts on the fly to generate bifurcation diagrams.
%Branch following results in a great number of output files. The organization is an important task.
Perl scripts are used to automatically number and store the output files and create
human readable reports about them.

\end{section}

%\label{results_section}
% \include{results}

\begin{section}{Numerical Results.}
\label{results_section}

\def\bifsize{.93}
\begin{figure}
\begin{center}
\rotatebox{-90}{\scalebox{\bifsize}{\includegraphics{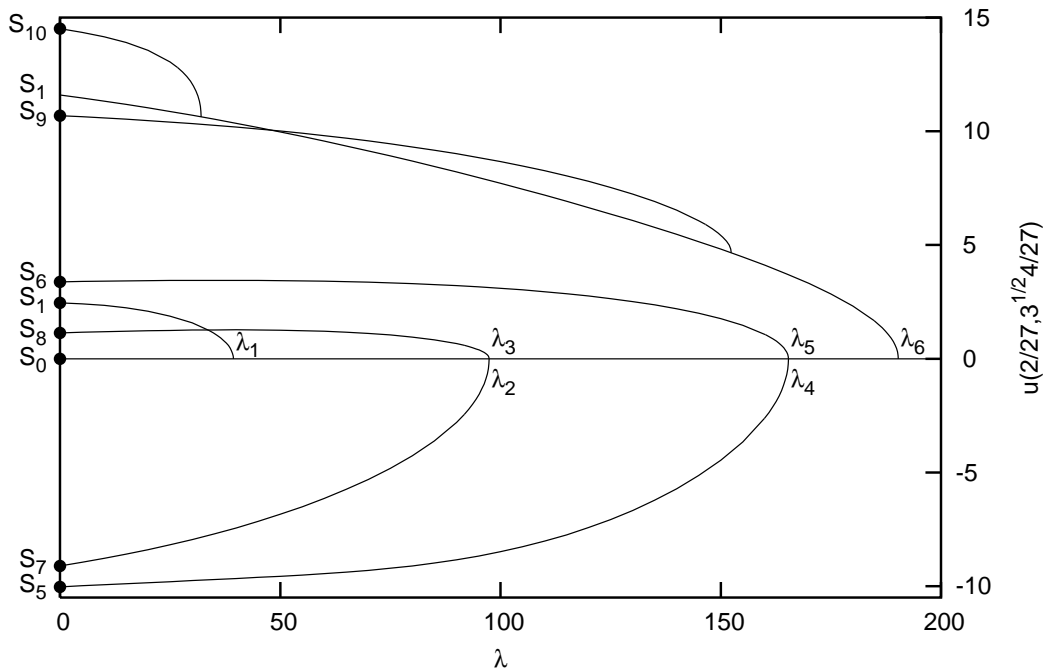}}}
\caption{The complete bifurcation diagram for the first six primary branches bifurcating
from the trivial branch.
%Primary branch $j$ is labelled by the eigenvalue $\lam_j$ at which it bifurcates.
%The branches born at multiple eigenvalues are labelled so that the
%$j$-th primary branch has MI $j$ near the bifurcation.
The second branch, with symmetry $S_7$, contains the CCN solution.
The dots at $\lambda=0$ in Figures~\ref{bif1-6}--\ref{bifD} correspond to
solutions depicted in Figures~\ref{sols1}~and~\ref{sols2}.
We used the level 5 grid with 300 modes in creating all bifurcation diagrams.
In Figure~\ref{levelmode} convergence data for the solution of symmetry type $S_{10}$ at
$\lambda=0$ is provided.
}
\label{bif1-6}
\end{center}
\end{figure}

\begin{figure}
\begin{center}
\rotatebox{-90}{\scalebox{\bifsize}{\includegraphics{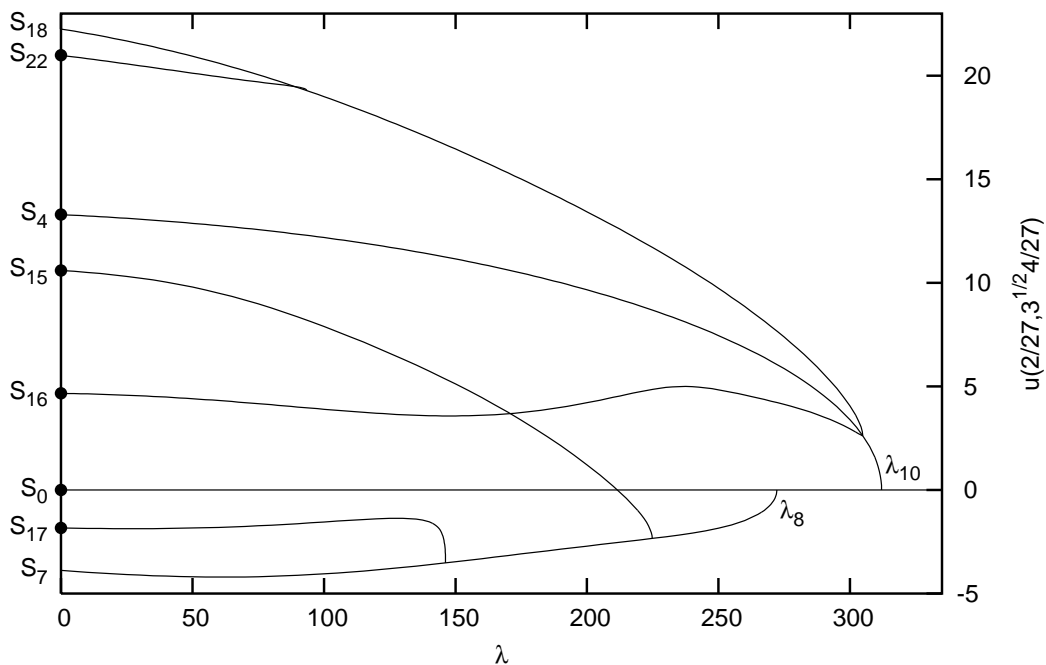}}}
% Note: The latest version is only on webwork1/snow/newton_copy/rbauto.
\caption{
A partial bifurcation diagram showing some of the solutions bifurcating from
the 8-th and 10-th primary branches.
Again, the dots at $\lam = 0$ indicate solutions shown
in Figures~\ref{sols1}~and~\ref{sols2}.
The contour plots as a function of $\lam$ are animated for the branches
ending with the dots indicating symmetry types $S_{15}$ ({\tt s7s15.gif}),
$S_{17}$ ({\tt s7s17.gif}), $S_{16}$ ({\tt s4s16.gif}), and $S_{22}$ ({\tt s4s18s22.gif}).
}
\label{bifB}
\end{center}
\end{figure}

\begin{figure}
\begin{center}
\rotatebox{-90}{\scalebox{\bifsize}{\includegraphics{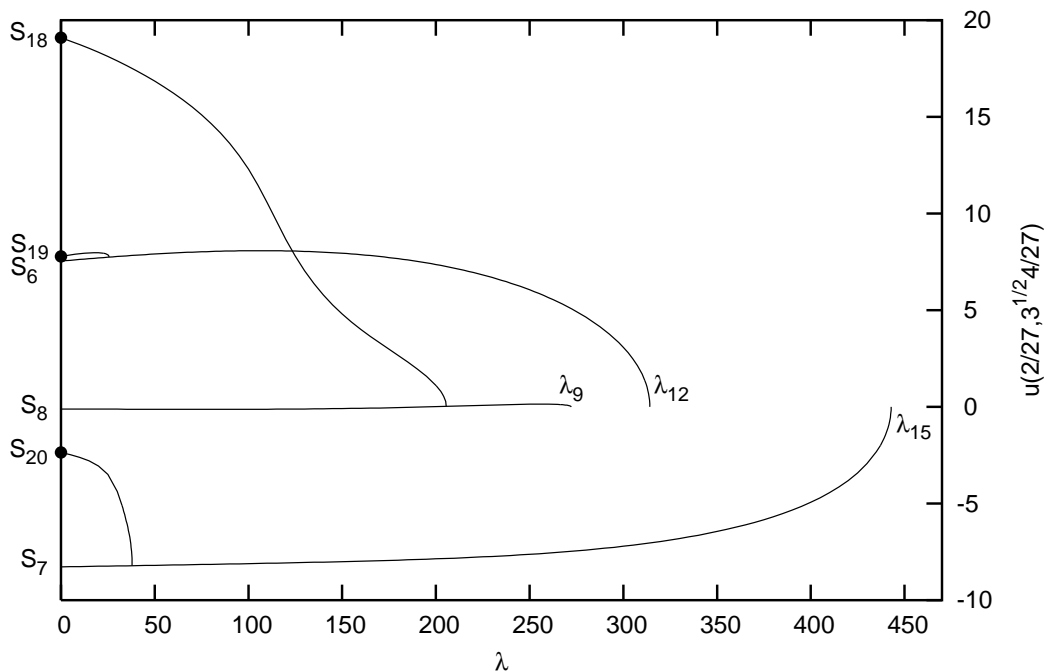}}}
% Note: The latest version is only on webwork1/snow/newton_copy/rbauto.
\caption{
A partial bifurcation diagram providing three additional symmetry types.
For clarity, the trivial branch is not shown in this and the next figure.
%Using the partial
%order of the lattice (see Figure~\ref{lattice}), we can determine a restricted set of
%symmetry types which may lead to bifurcating branches of any desired symmetry.
%Our branch following code is capable of using this information to restrict the search.
}
\label{bifC}
\end{center}
\end{figure}

\begin{figure}
\begin{center}
\rotatebox{-90}{\scalebox{\bifsize}{\includegraphics{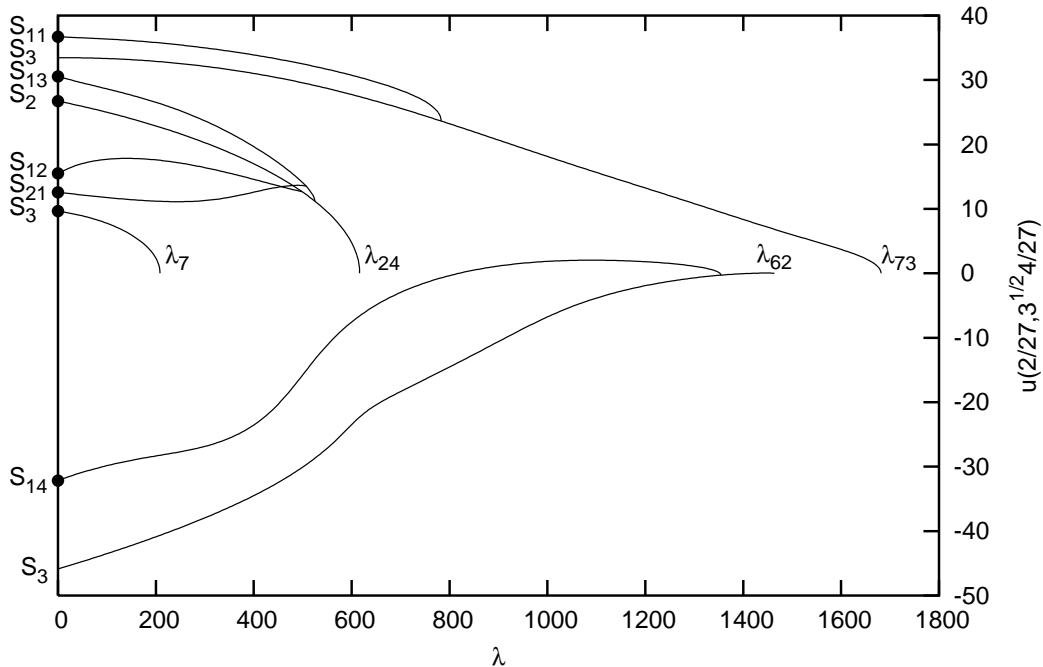}}}
\caption{
A partial bifurcation diagram containing solutions of the seven remaining symmetry types.
Primary branch 24 is the first branch with symmetry type $S_{2}$.
The symmetry types $S_{14}$ and $S_{11}$ were found by searching the first one hundred
primary branches, following only those branches which can lead to solutions with
the desired symmetry.  These two solutions are included for completeness, but
their existence for the PDE would have to be confirmed with more modes and
a higher level approximation of the eigenfunctions.
}
\label{bifD}
\end{center}
\end{figure}

\def\sbsize{0.3}
\def\figheight{4.5cm}

\begin{figure}
\begin{center}
\begin{tabular}{|c|c|c|}
\hline
\vphantom{{\rule{0cm}{\figheight}}}
   \scalebox{.38}{\includegraphics{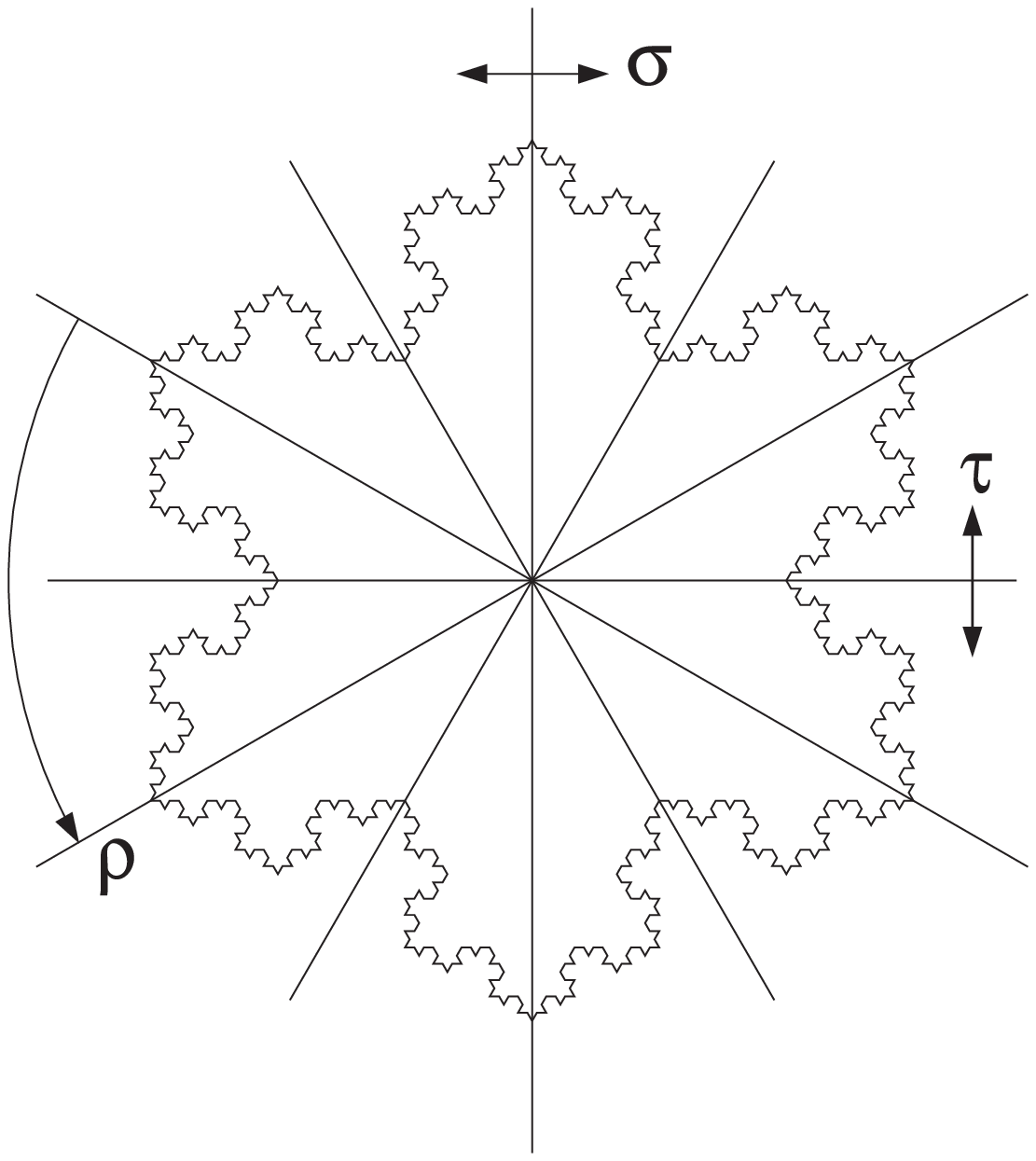}}
&
    \scalebox{\sbsize}{\includegraphics{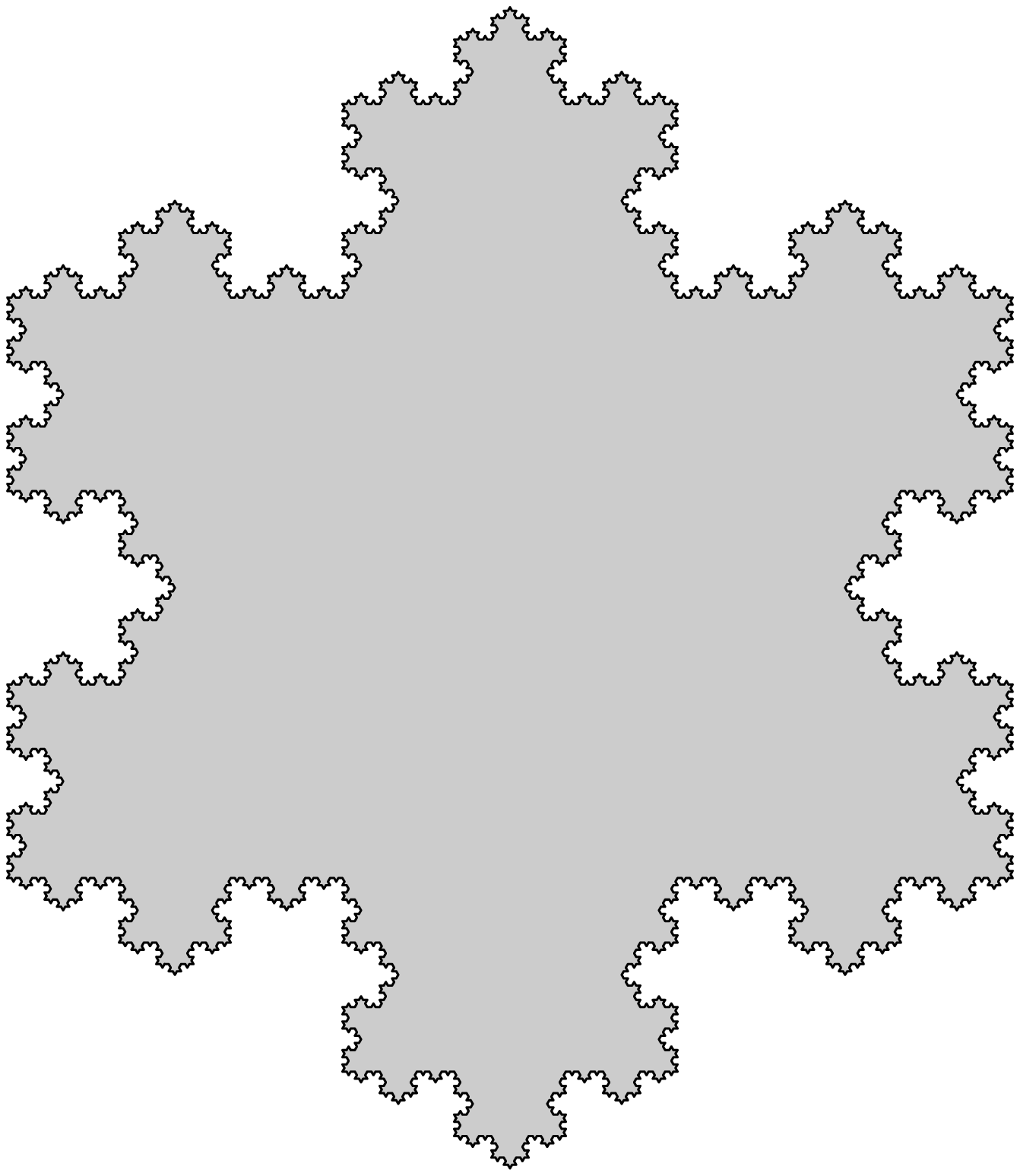}}
&
    \scalebox{\sbsize}{\includegraphics{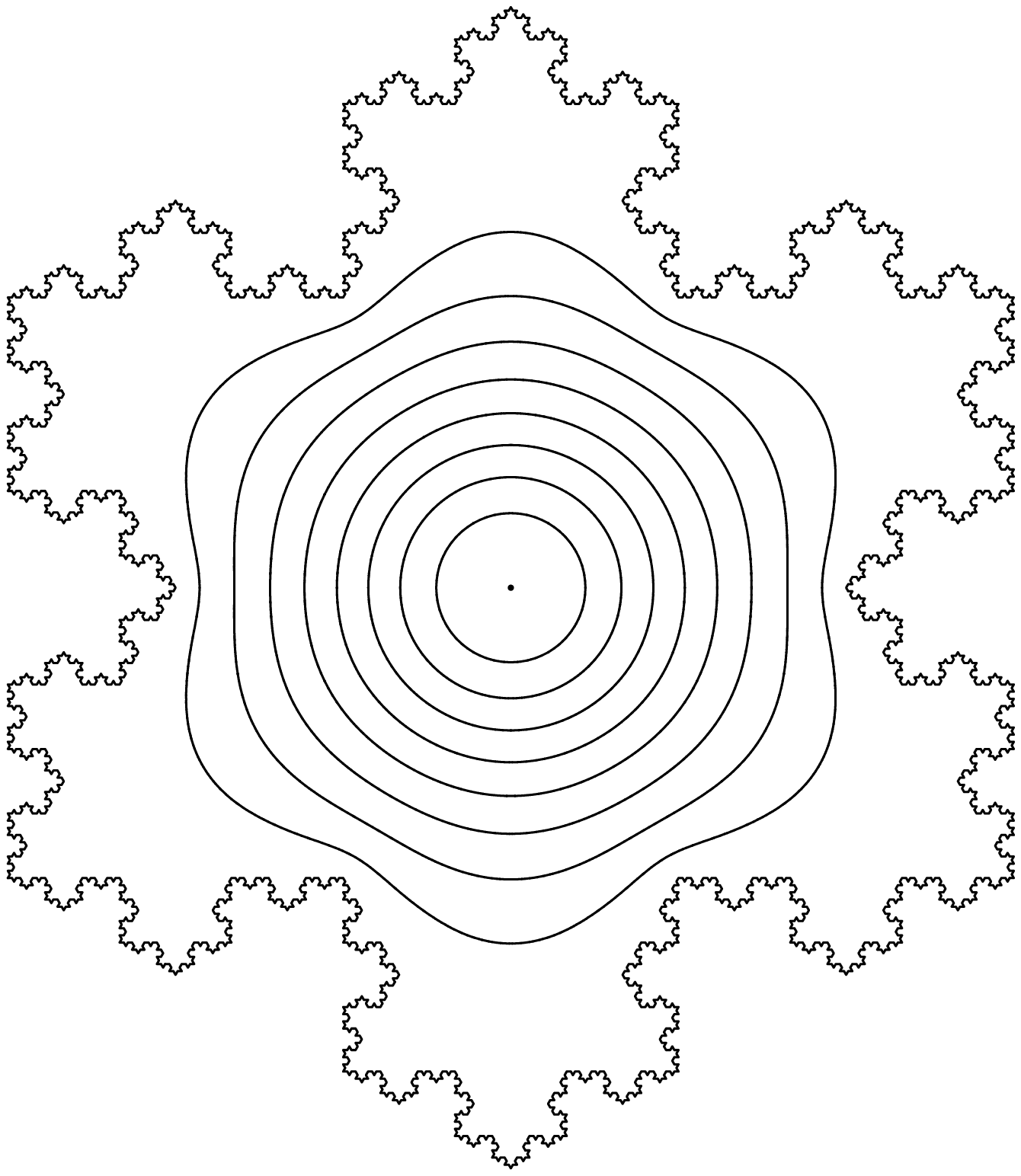}}
 \\
Action of $\rho$, $\sig$, and $\tau$.&
%$\Gam_0 = \langle\rho,\sigma, \tau, -1 \rangle$ &
$\Gam_0 = \langle\rho,\sigma, \tau, -1 \rangle = \DZ$ &
$\Gam_1 = \langle\rho,\sigma, \tau \rangle = \D_6$
\\

\hline
\vphantom{{\rule{0cm}{\figheight}}}
    \scalebox{\sbsize}{\includegraphics{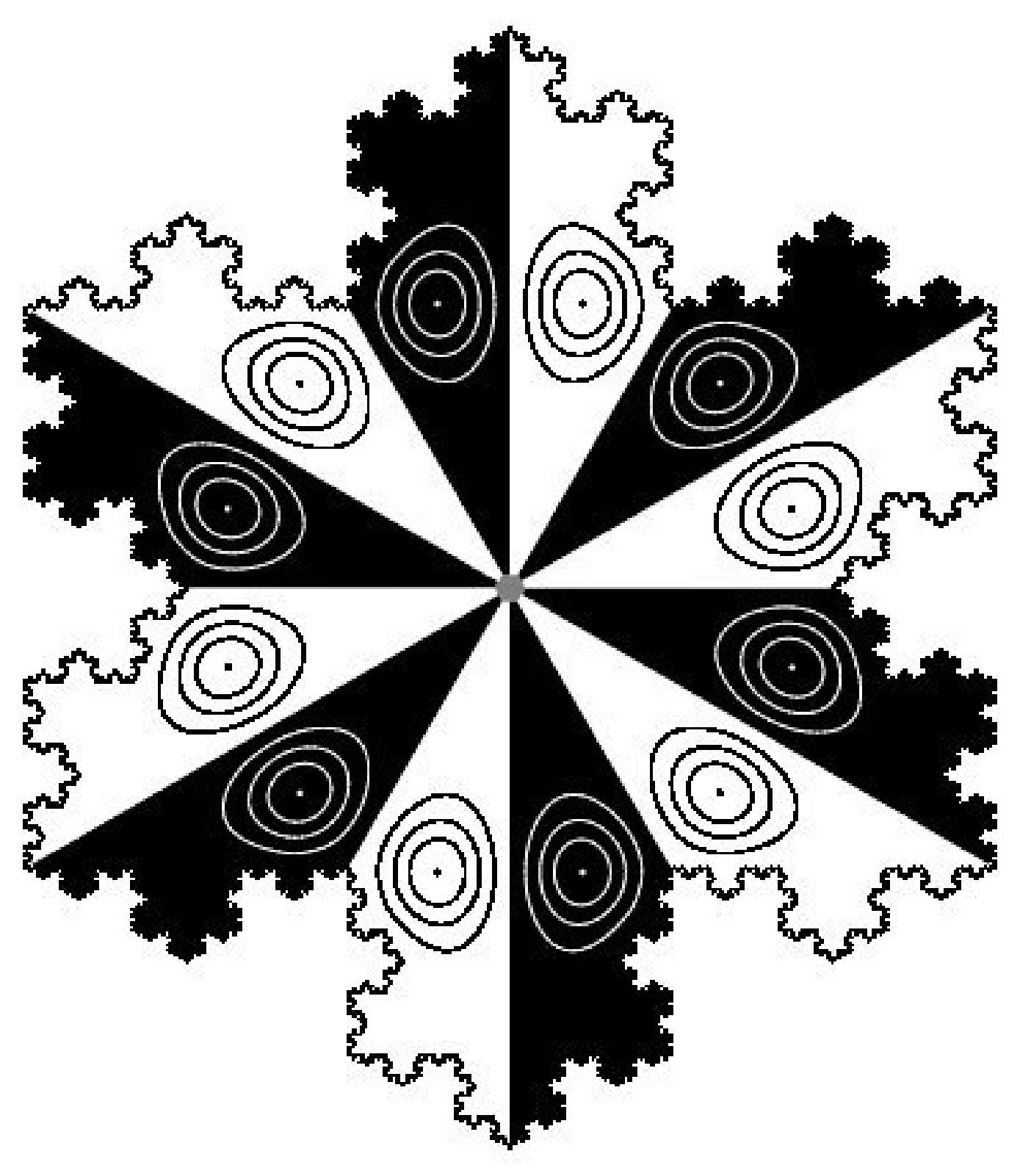}}
&
    \scalebox{\sbsize}{\includegraphics{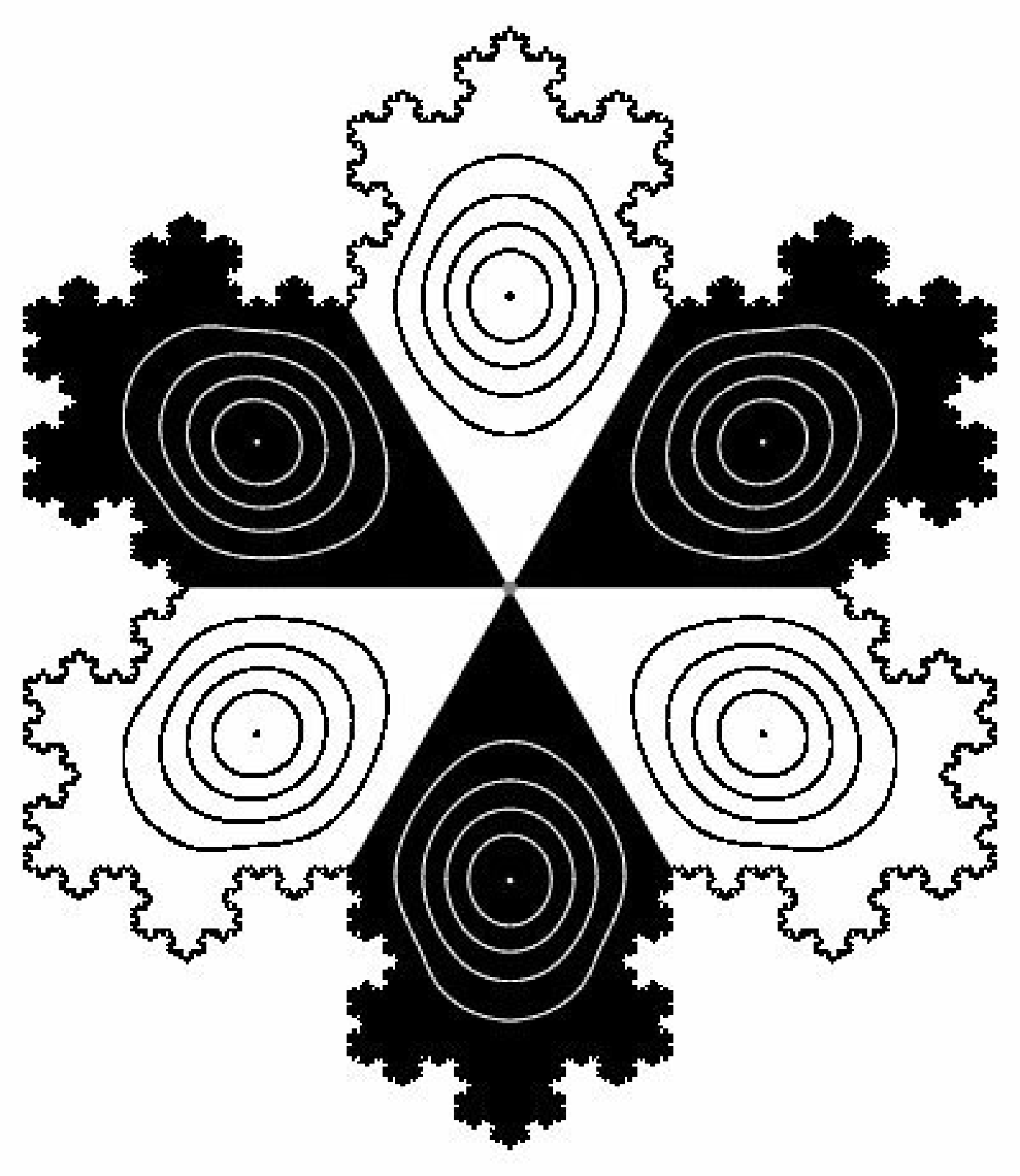}}
&
    \scalebox{\sbsize}{\includegraphics{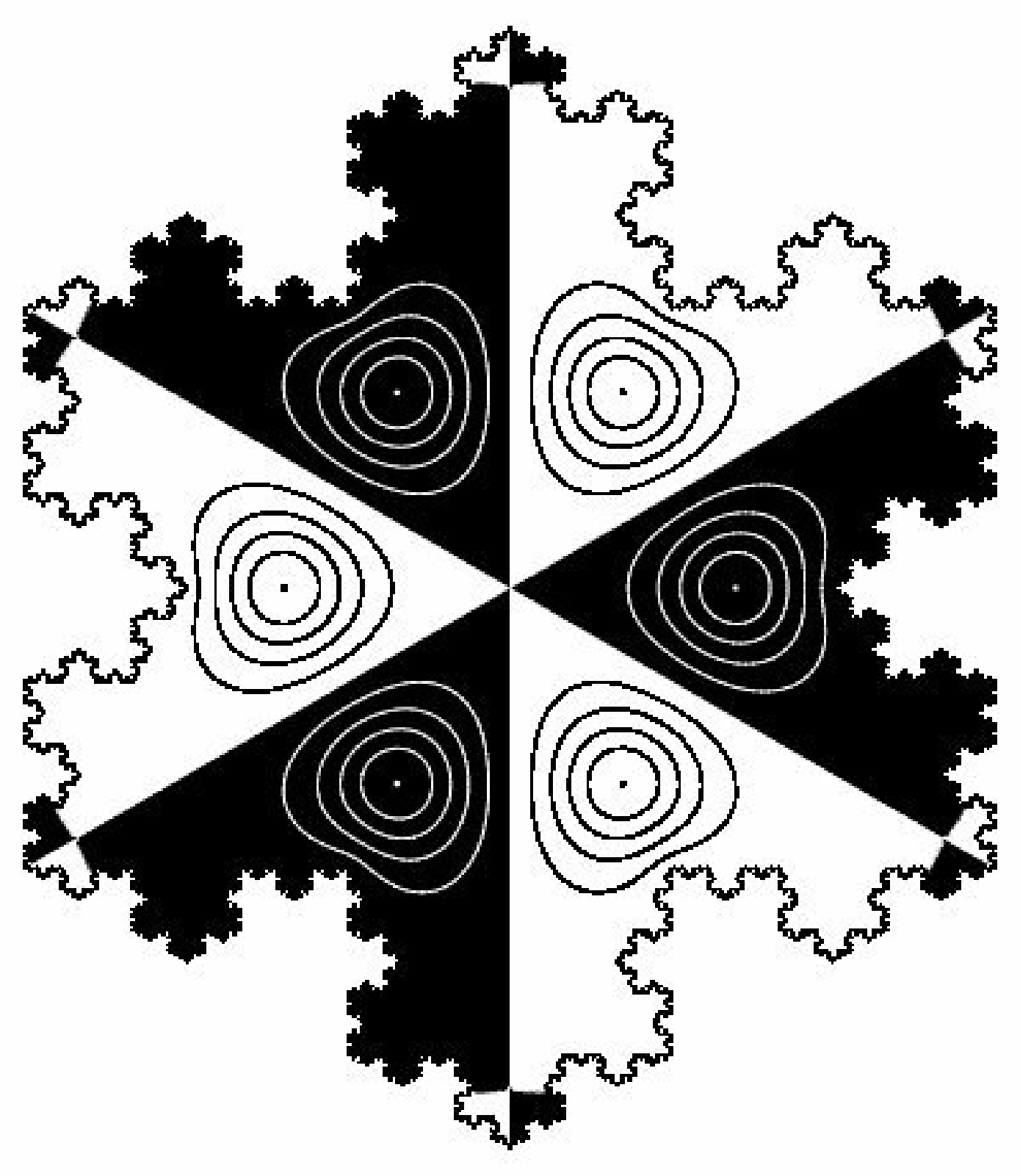}}

    \\

$\Gam_2 = \langle\rho,-\sigma, -\tau \rangle \cong \D_6$ &
$\Gam_3 = \langle-\rho,\sigma -\tau \rangle \cong \D_6$ &
$\Gam_4 = \langle-\rho,-\sigma,\tau\rangle \cong \D_6$
\\

\hline
\vphantom{{\rule{0cm}{\figheight}}}
    \scalebox{\sbsize}{\includegraphics{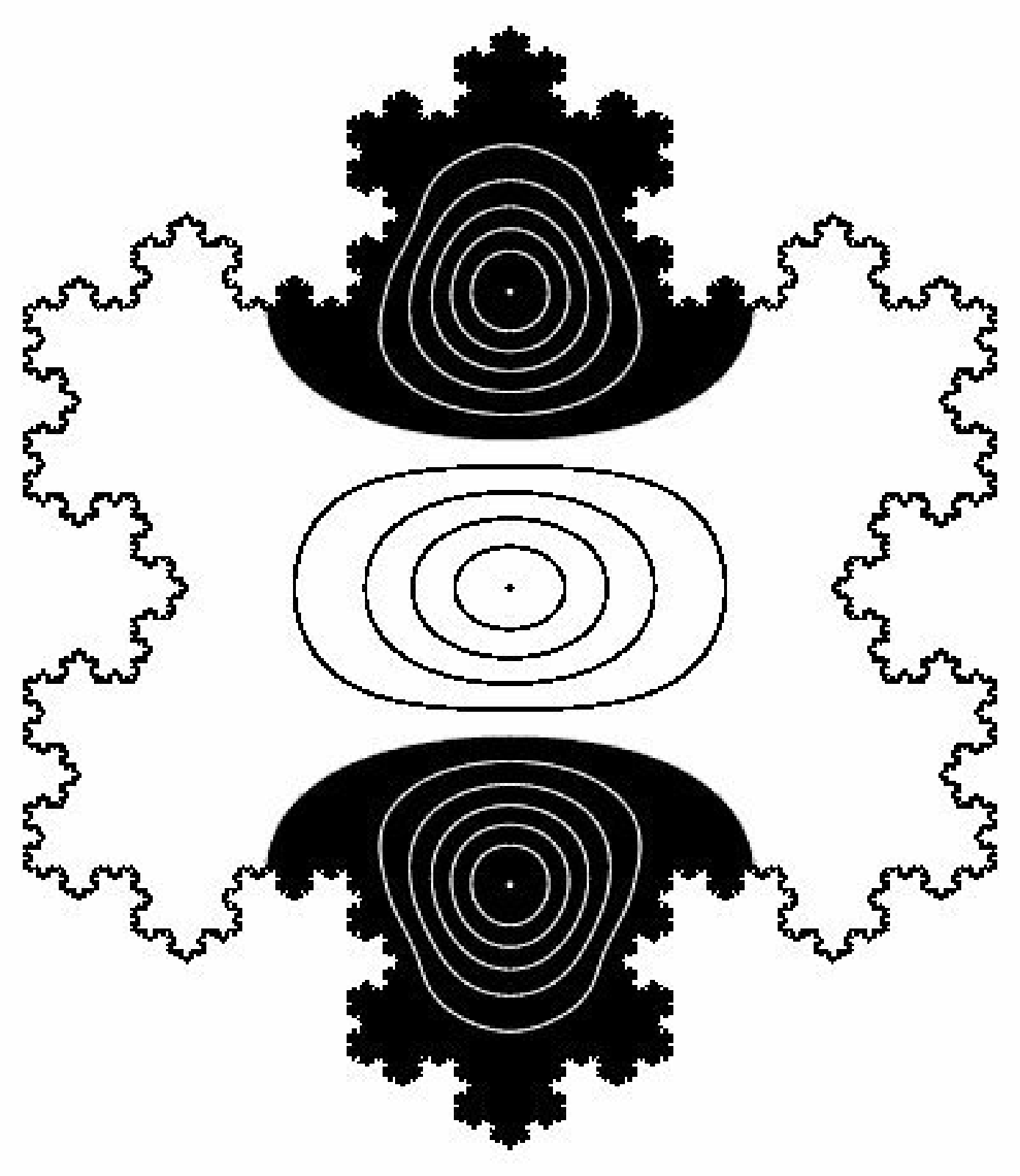}}
&
    \scalebox{\sbsize}{\includegraphics{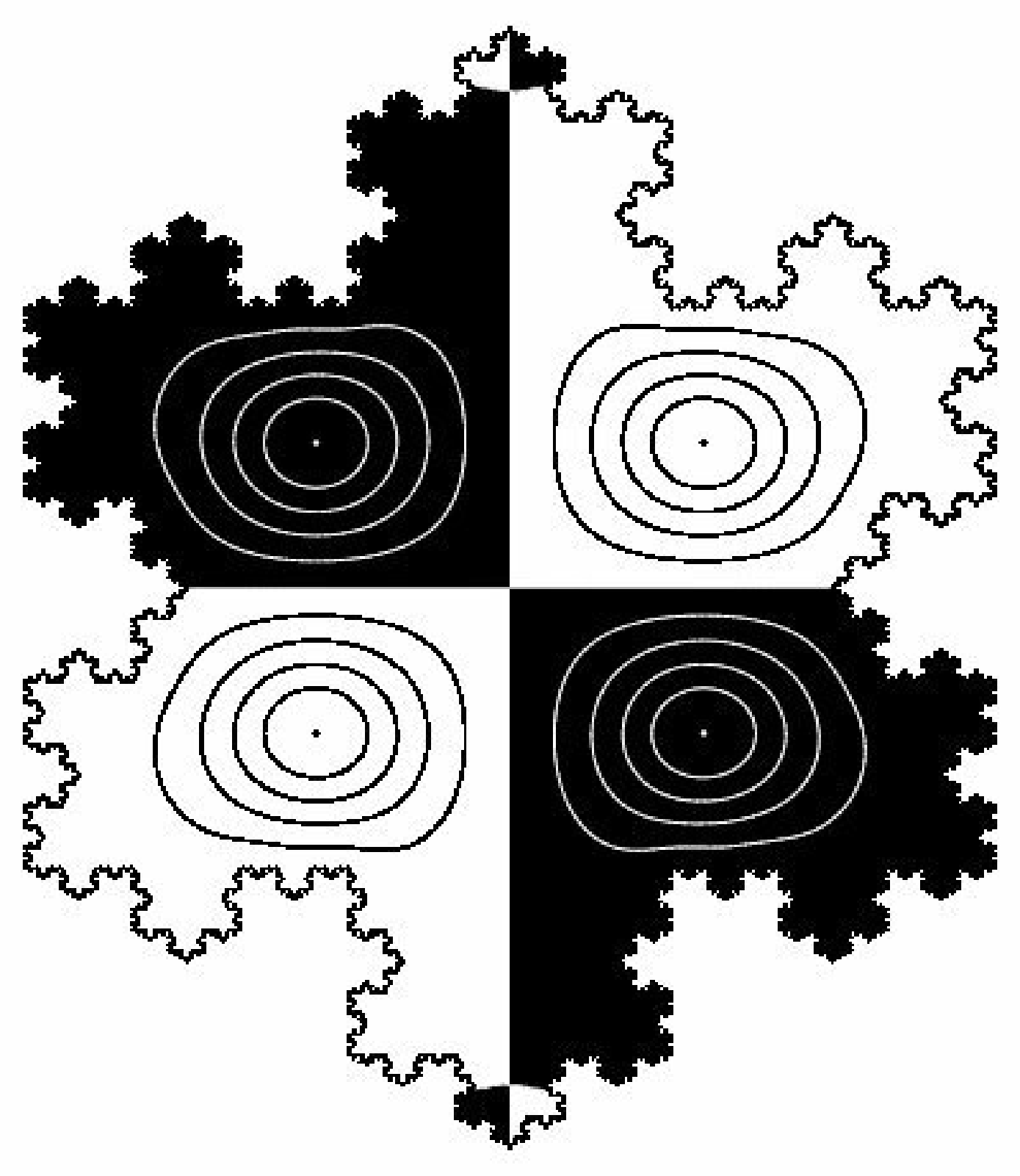}}
&
    \scalebox{\sbsize}{\includegraphics{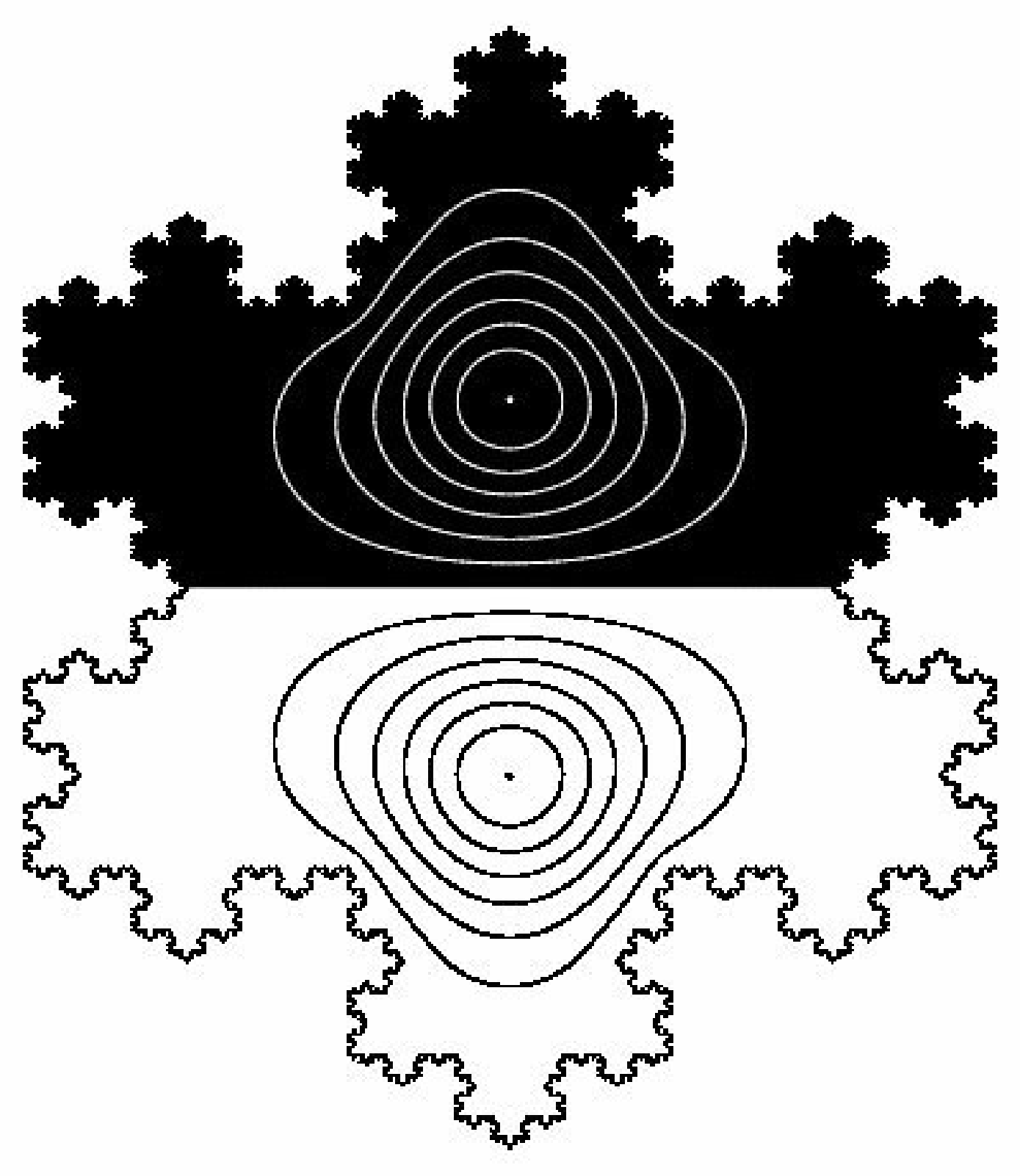}}
    \\
$\Gam_5 = \langle\sigma,\tau\rangle \cong \Z_2 \times \Z_2$  &
$\Gam_6 = \langle-\sigma,-\tau\rangle \cong \Z_2 \times \Z_2$ &
$\Gam_7 = \langle\sigma,-\tau\rangle \cong \Z_2 \times \Z_2$
\\

\hline
\vphantom{{\rule{0cm}{\figheight}}}
    \scalebox{\sbsize}{\includegraphics{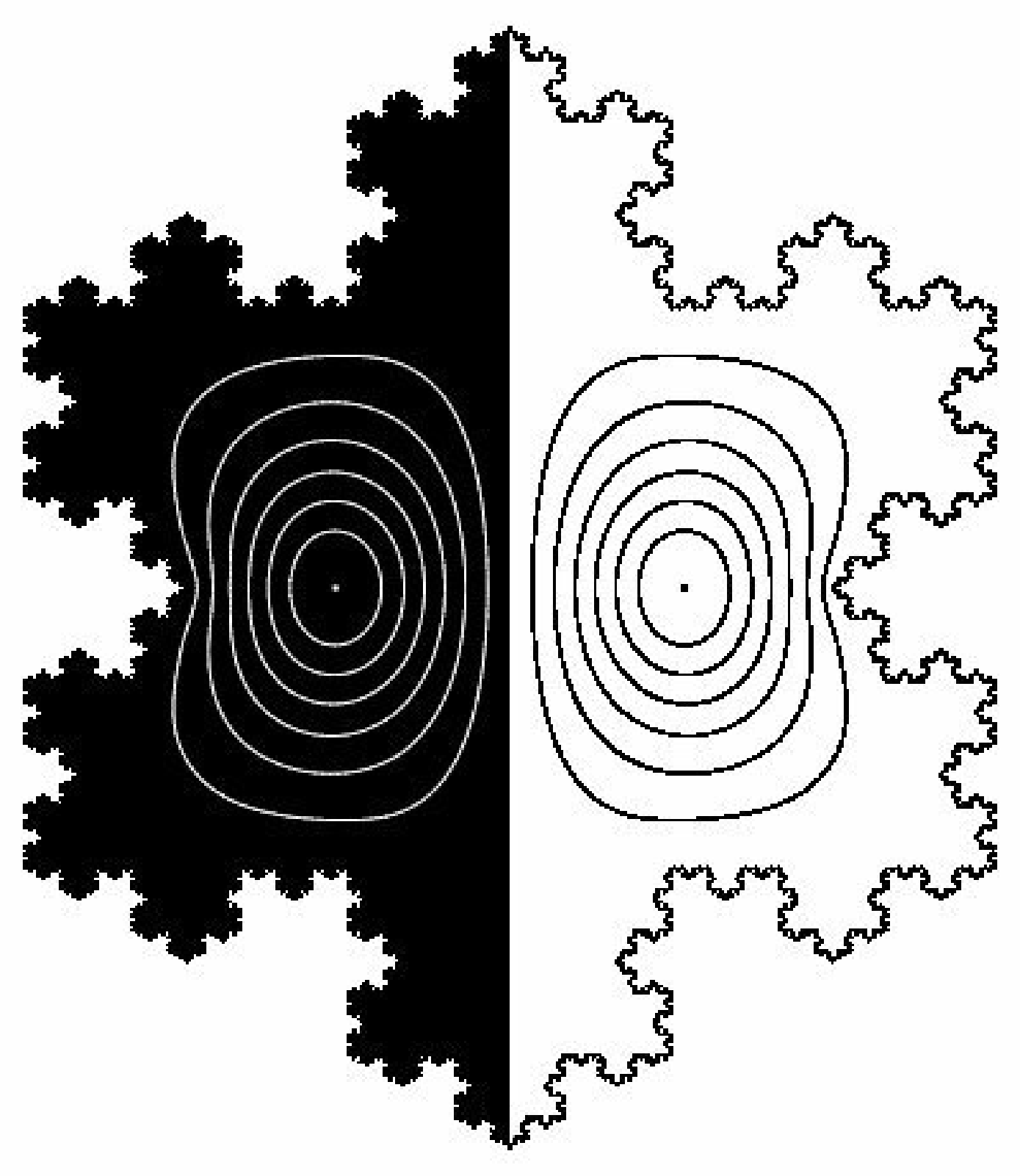}}
&
    \scalebox{\sbsize}{\includegraphics{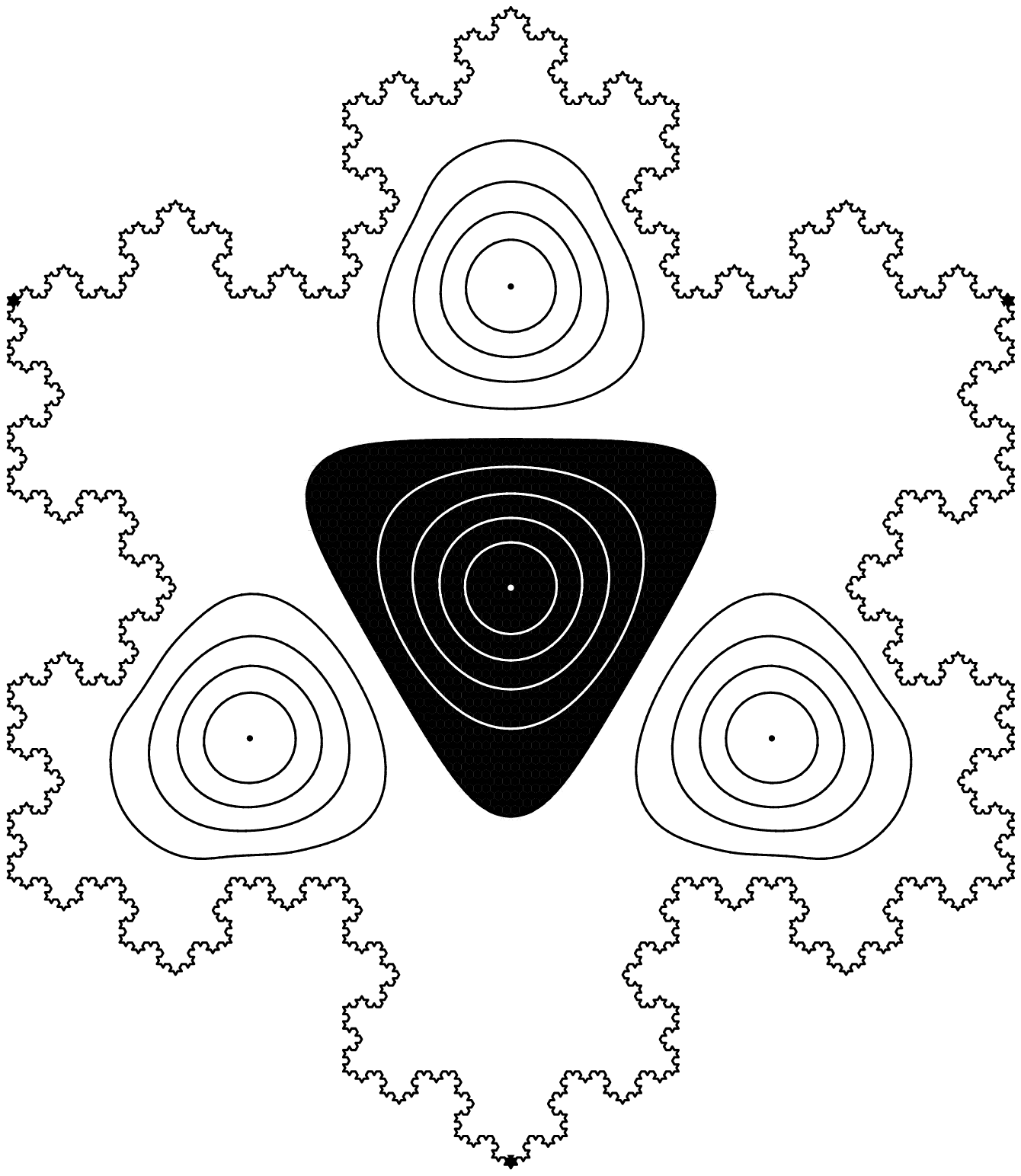}}
&
    \scalebox{\sbsize}{\includegraphics{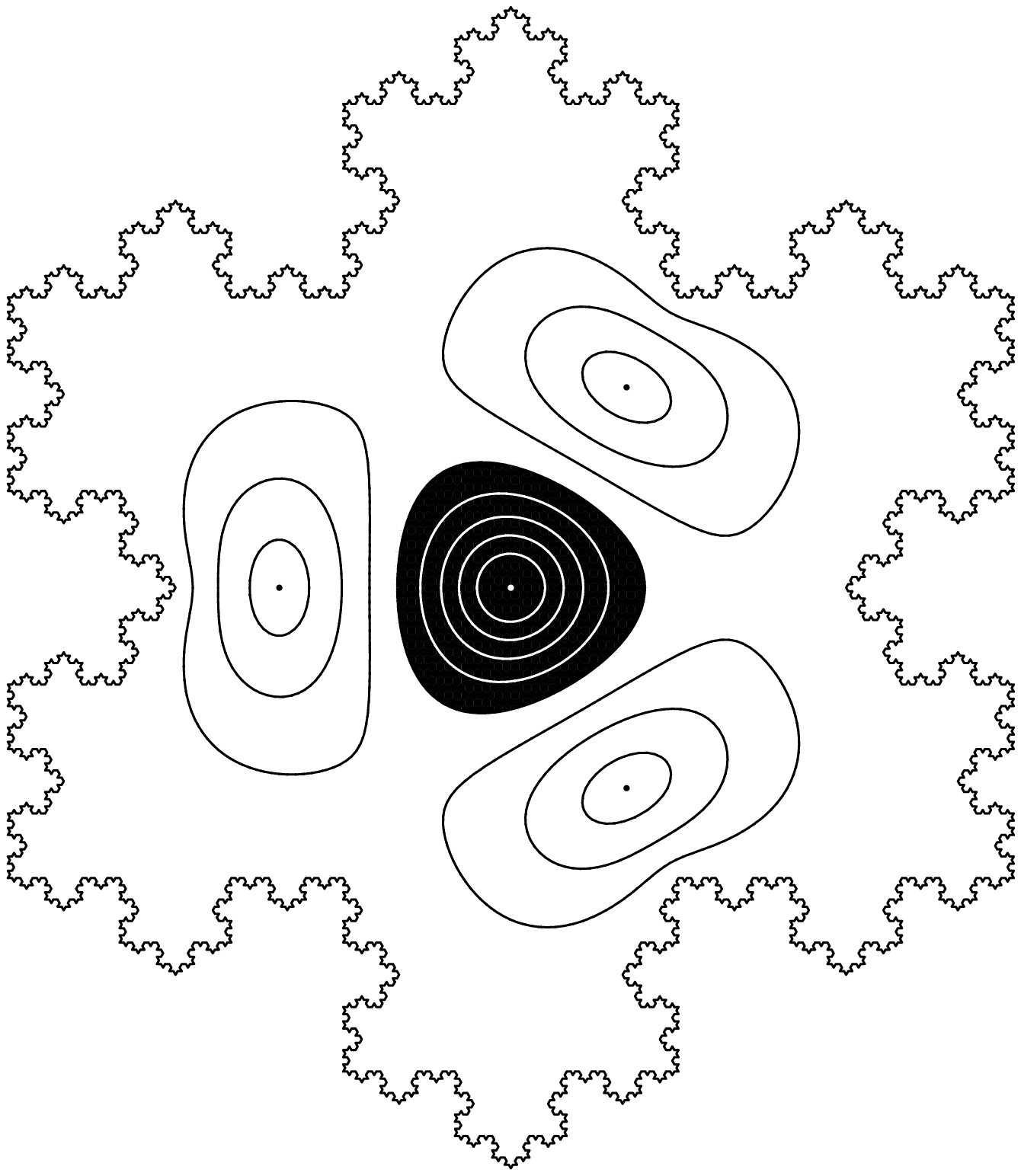}}
    \\
$\Gam_8 = \langle-\sigma,\tau\rangle \cong \Z_2 \times \Z_2$ &
$\Gam_9 = \langle\rho^2,\sigma\rangle \cong \D_3$ &
$\Gam_{10} = \langle\rho^2,\tau\rangle \cong \D_3$
 \\
\hline

%\hline
%
\end{tabular}
%
%\vspace*{.1in}
\caption{The action of the generators of $\D_6$ on the plane, along with
contour plots of solutions with symmetry types $S_0,\ldots,S_{10}$ at $\lambda=0$.
Recall that $S_i = [\Gam_i]$.}
\label{sols1}
\end{center}
\end{figure}

%%%%%%%%%%%%%%%%%%%%%%%%%%%

\begin{figure}
\begin{center}
\begin{tabular}{|c|c|c|c|c|}
\hline
\vphantom{{\rule{0cm}{\figheight}}}
    \scalebox{\sbsize}{\includegraphics{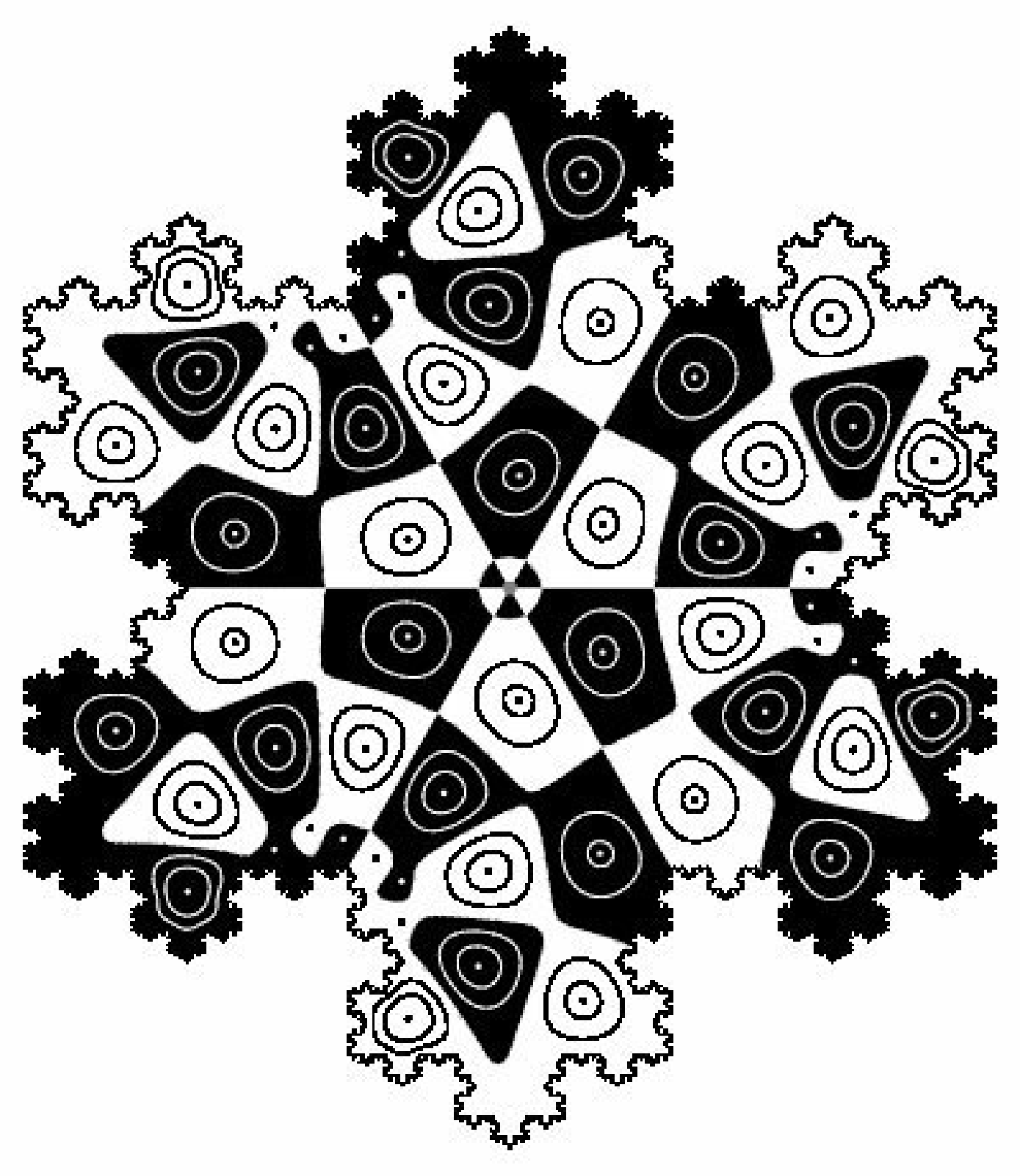}}
&
    \scalebox{\sbsize}{\includegraphics{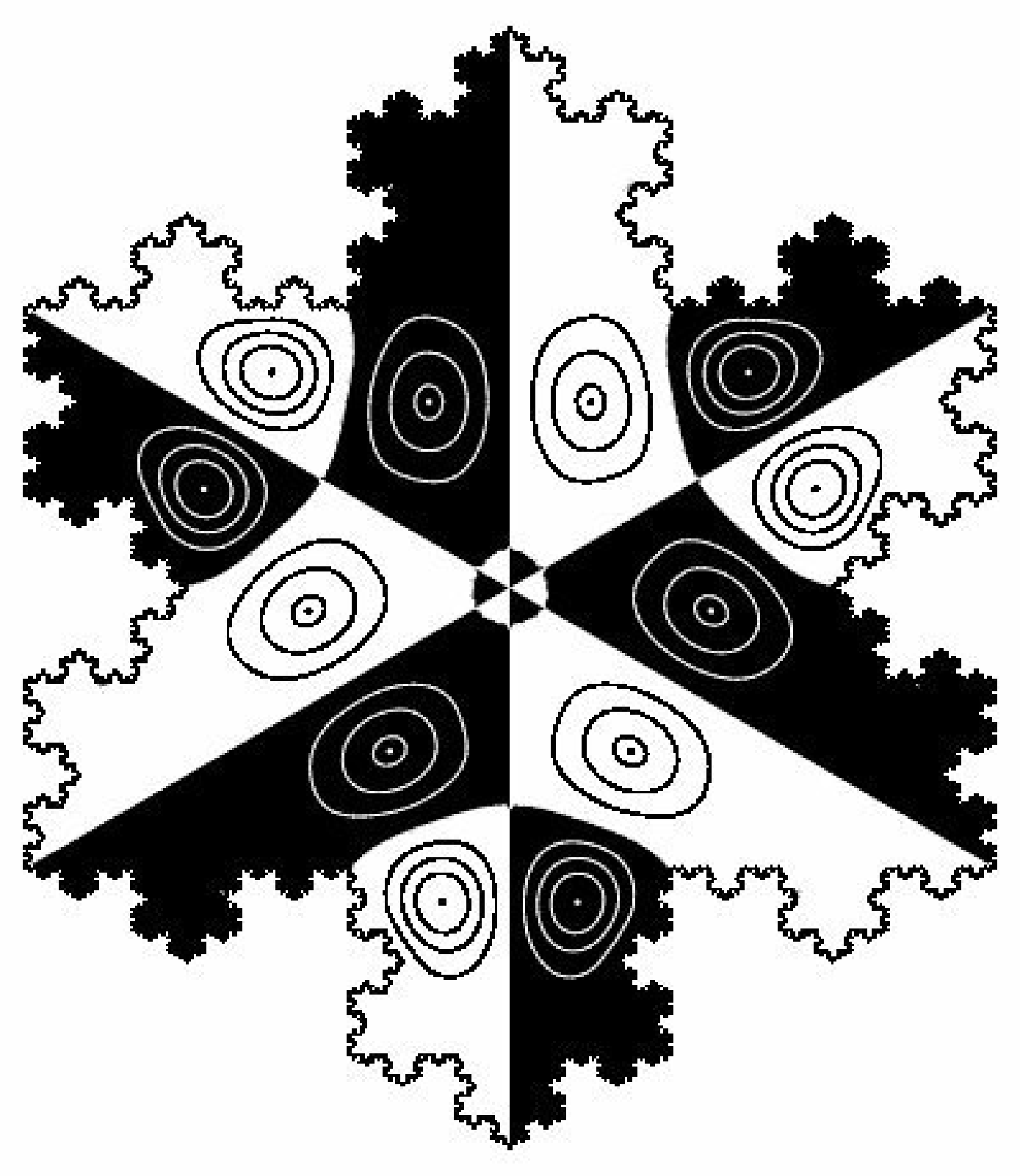}}
&
    \scalebox{\sbsize}{\includegraphics{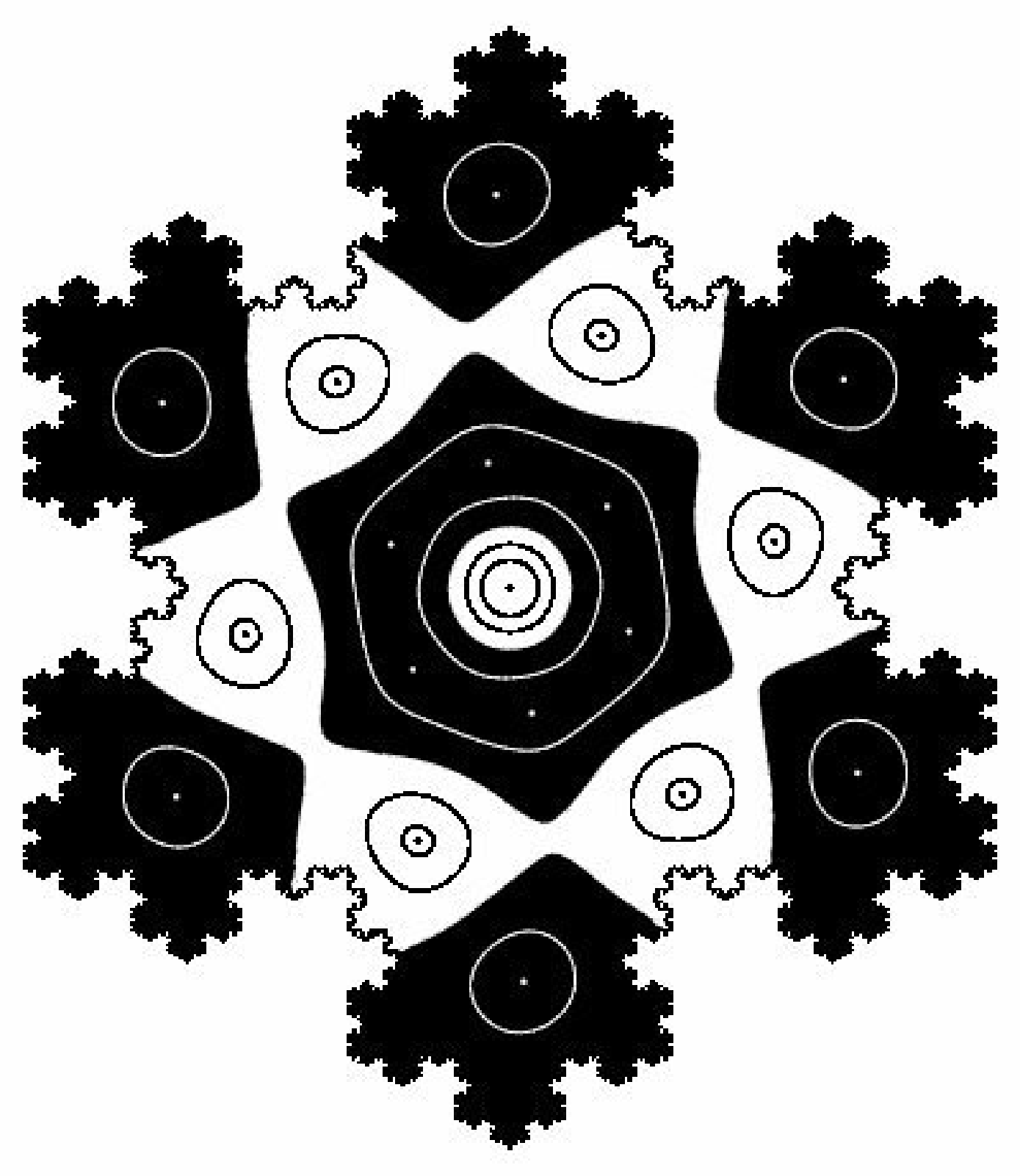}}
\\
%&
$\Gam_{11} = \langle\rho^2,-\tau\rangle \cong \D_3$ &
$\Gam_{12} = \langle\rho^2,-\sigma\rangle \cong \D_3$ &
$\Gam_{13} = \langle\rho\rangle \cong \Z_6$
\\
\hline
\vphantom{{\rule{0cm}{\figheight}}}
    \scalebox{\sbsize}{\includegraphics{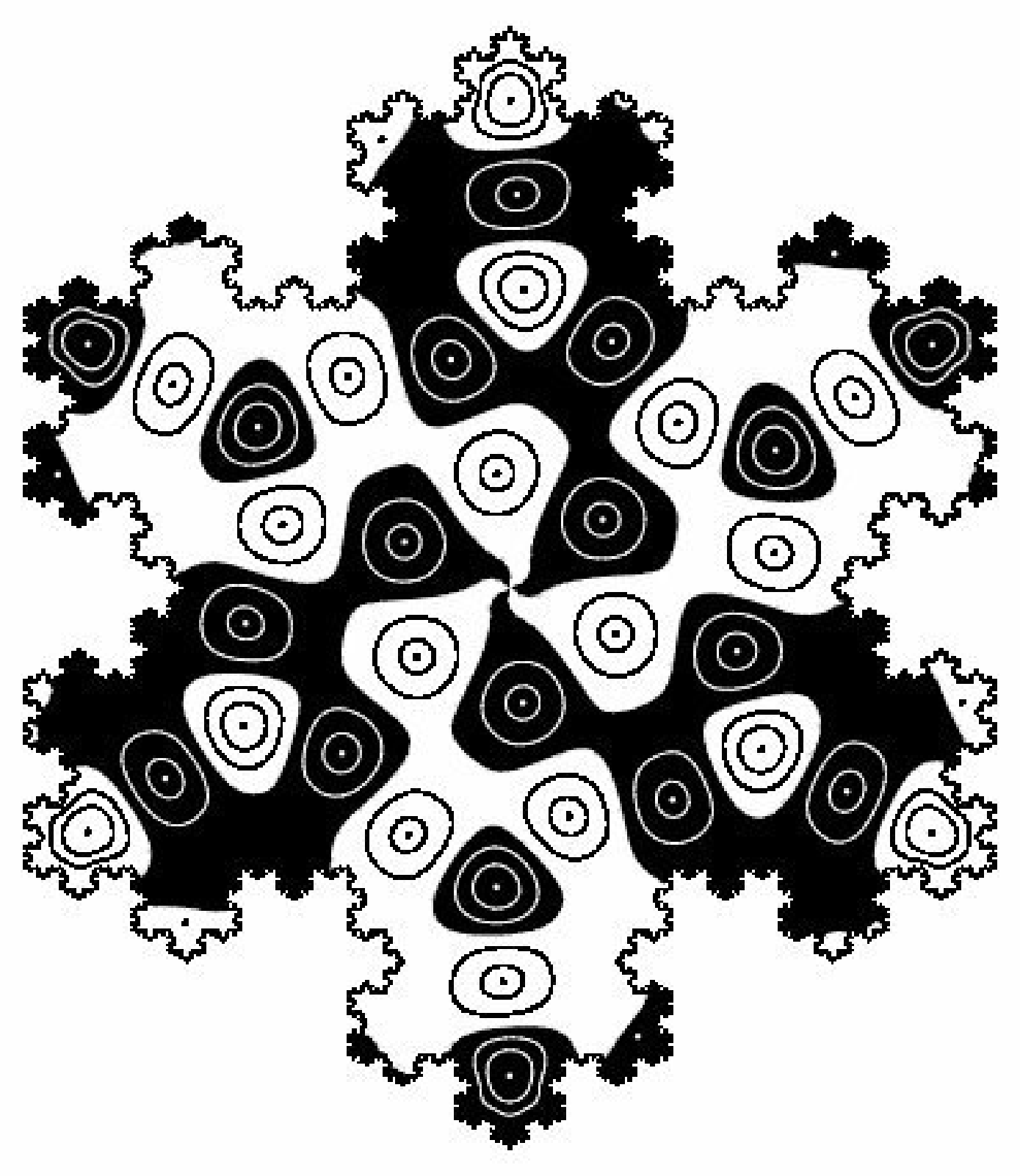}}
&
    \scalebox{\sbsize}{\includegraphics{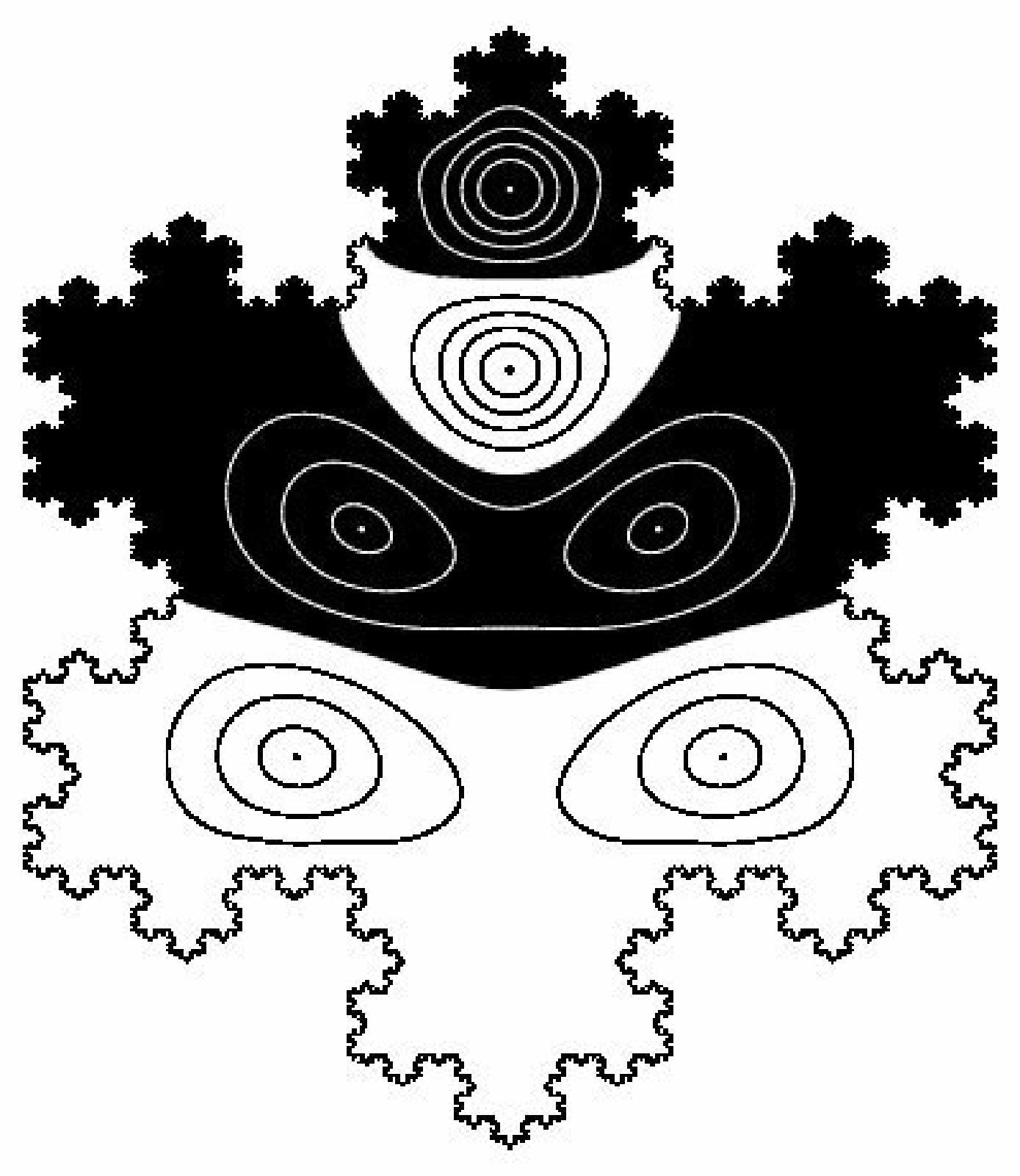}}
&
    \scalebox{\sbsize}{\includegraphics{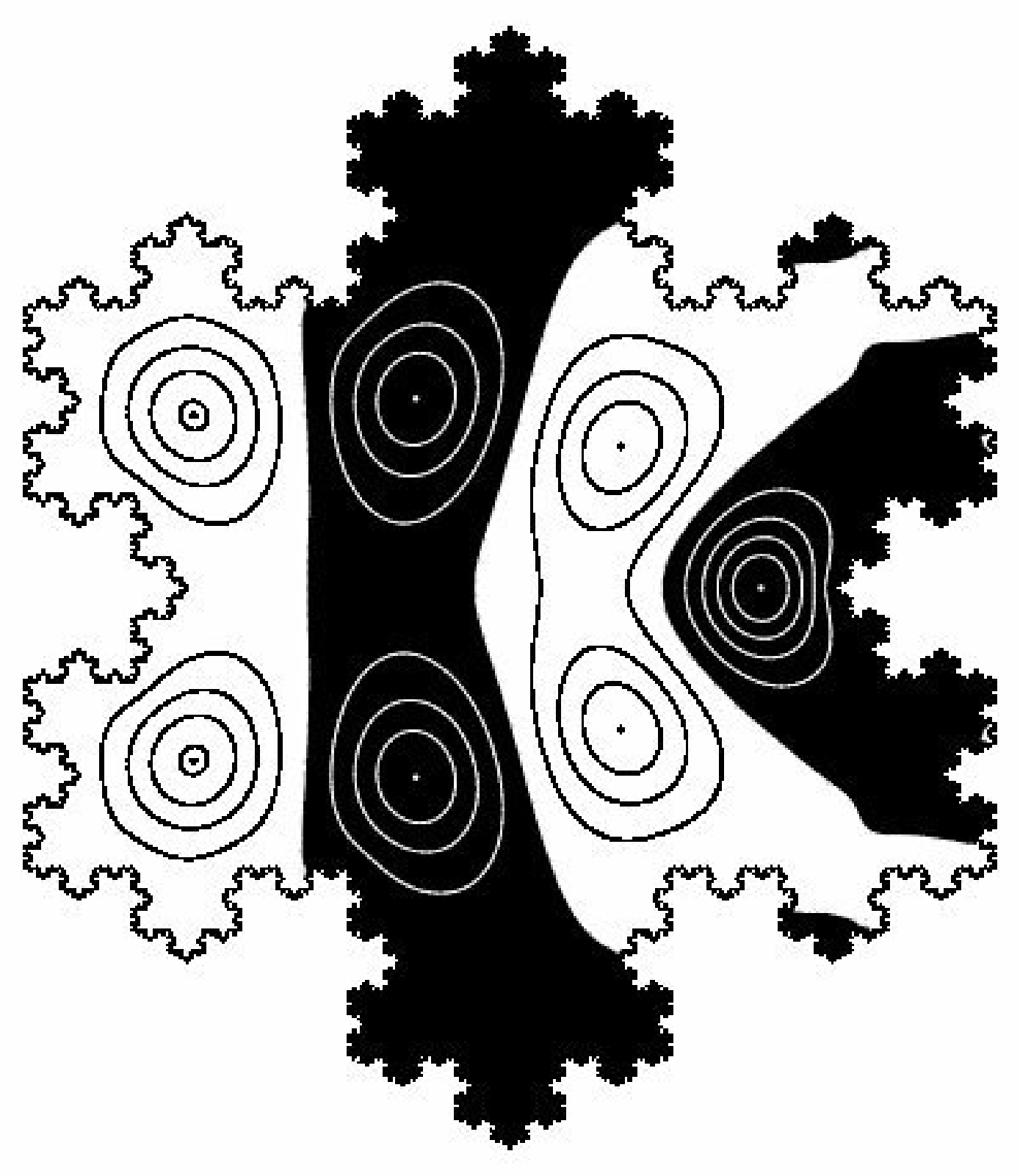}}
\\
$\Gam_{14} = \langle-\rho\rangle \cong \Z_6$ &
$\Gam_{15} = \langle\sigma\rangle \cong \Z_2$ &
$\Gam_{16} = \langle\tau\rangle \cong \Z_2$
\\
\hline
\vphantom{{\rule{0cm}{\figheight}}}
    \scalebox{\sbsize}{\includegraphics{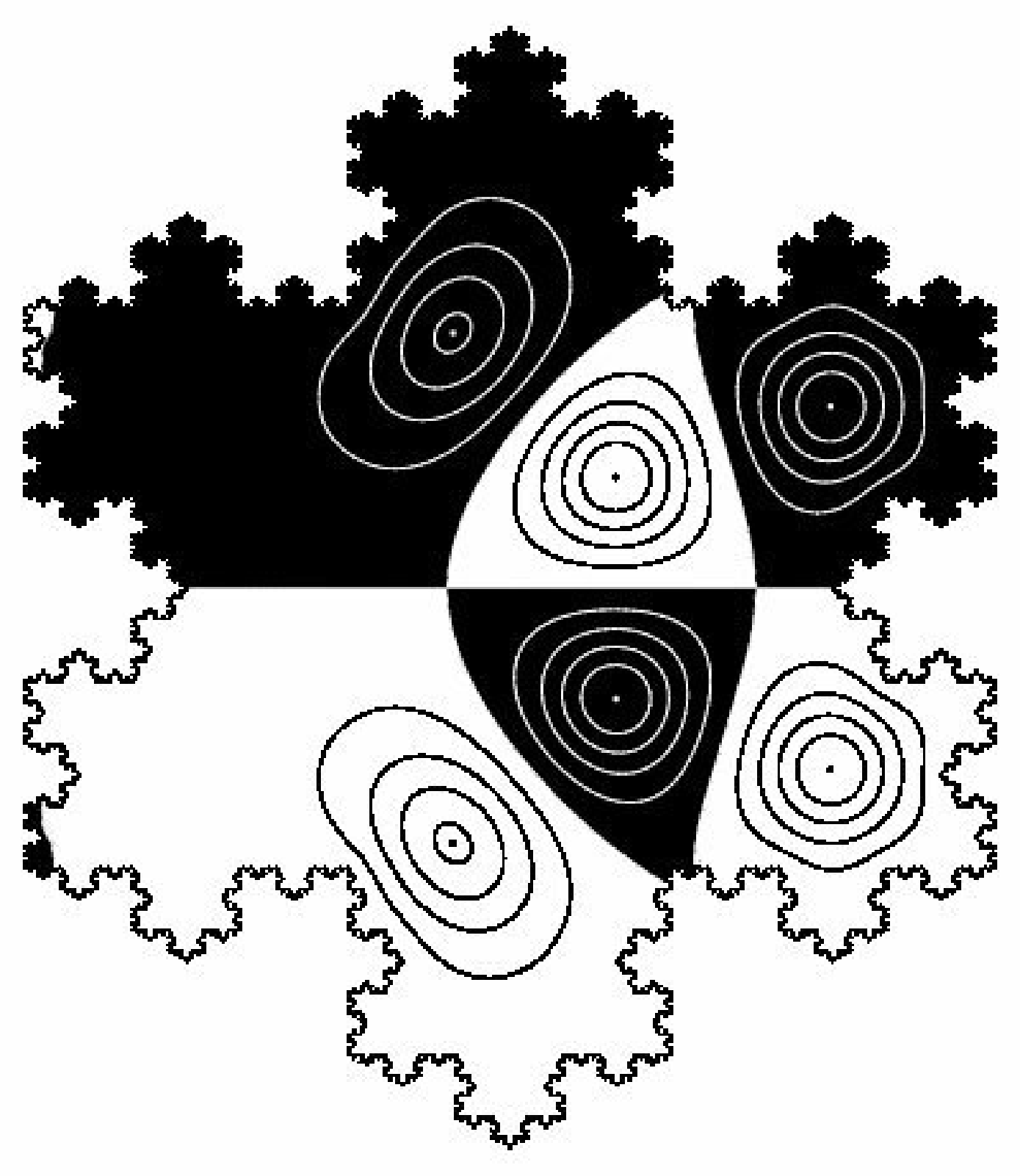}}
&
    \scalebox{\sbsize}{\includegraphics{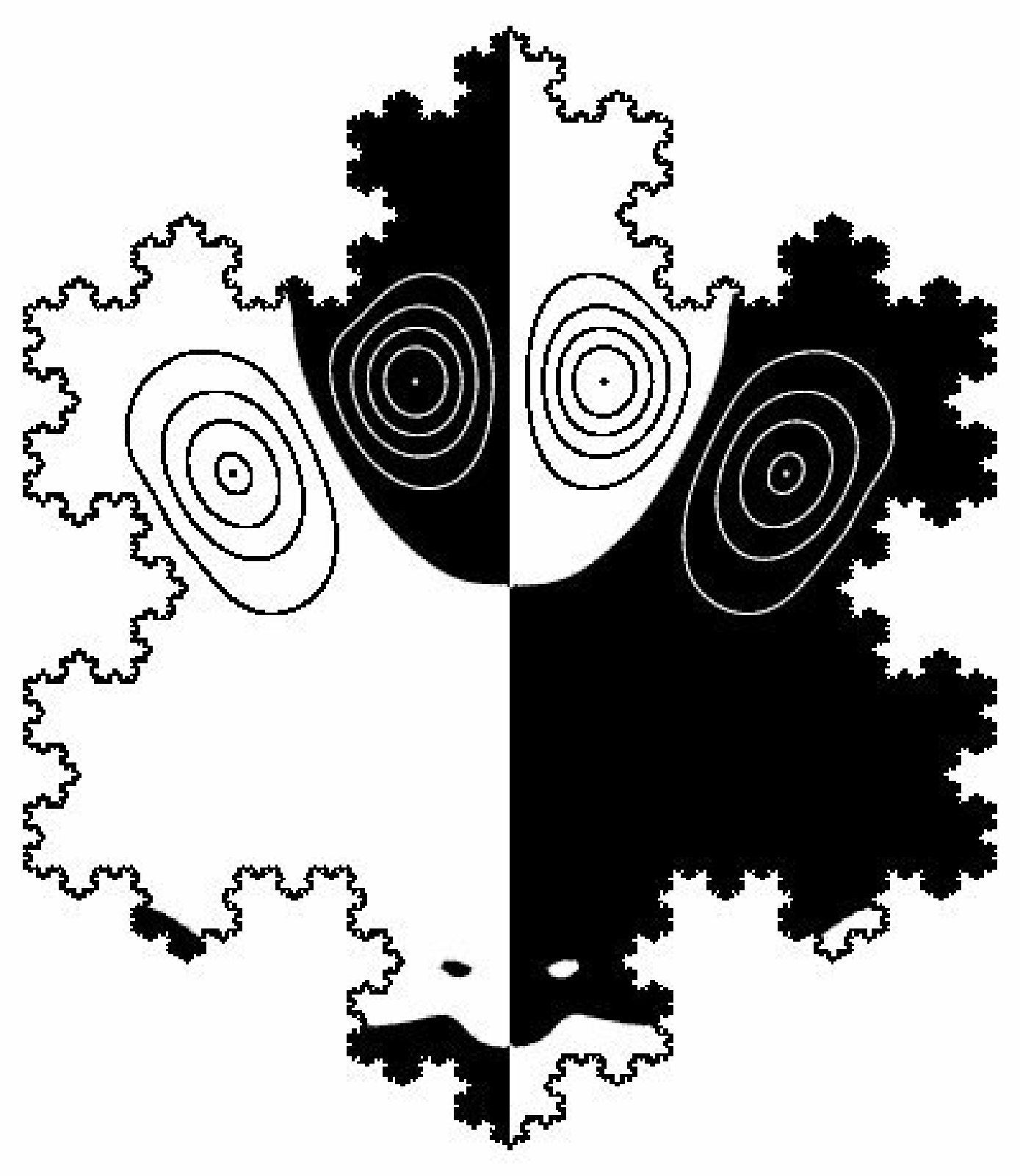}}
&
    \scalebox{\sbsize}{\includegraphics{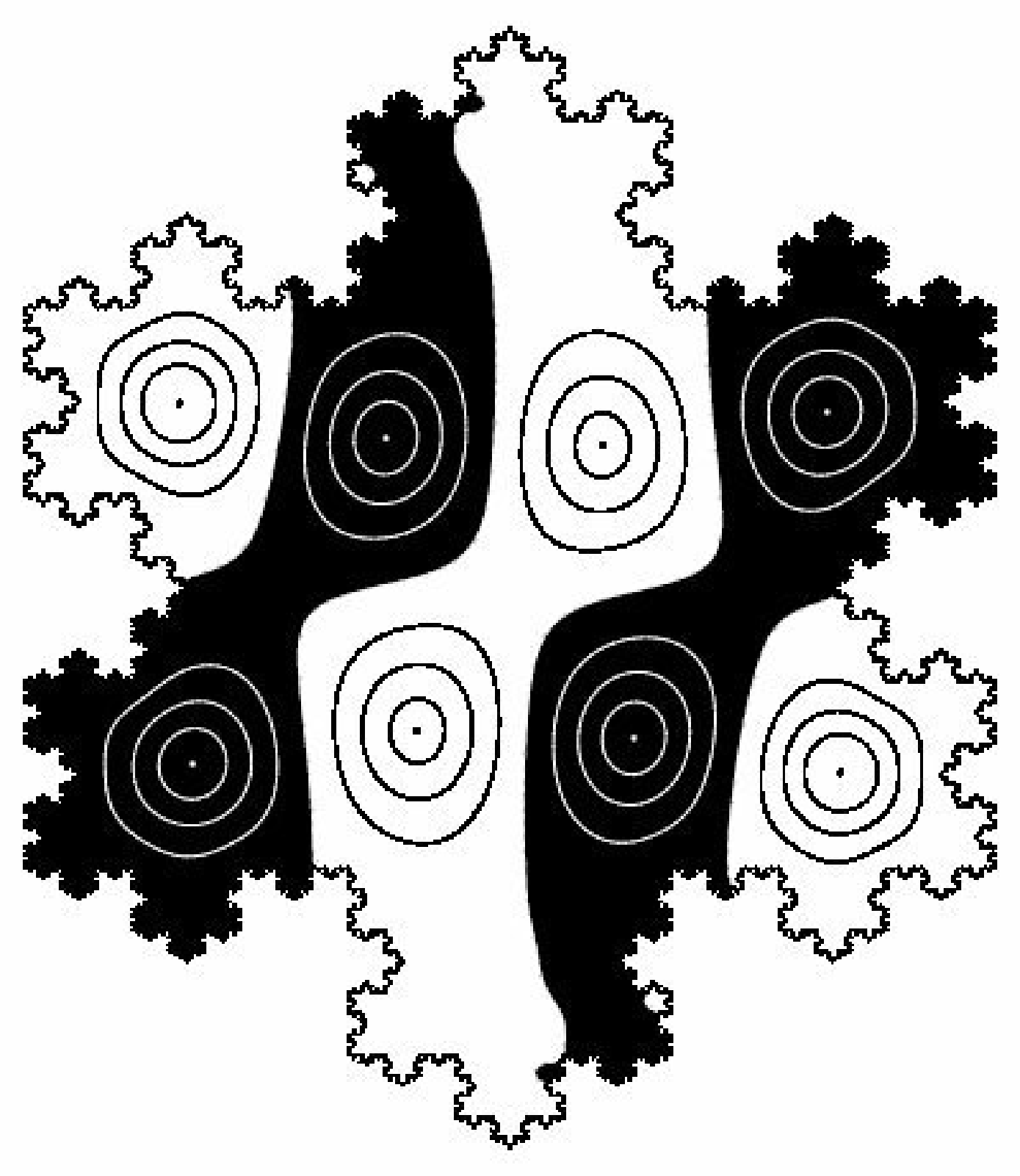}}
\\
$\Gam_{17} = \langle-\tau\rangle \cong \Z_2$ &
$\Gam_{18} = \langle-\sigma\rangle \cong \Z_2$ &
$\Gam_{19} = \langle\rho^3\rangle \cong \Z_2$
\\
\hline
\vphantom{{\rule{0cm}{\figheight}}}
    \scalebox{\sbsize}{\includegraphics{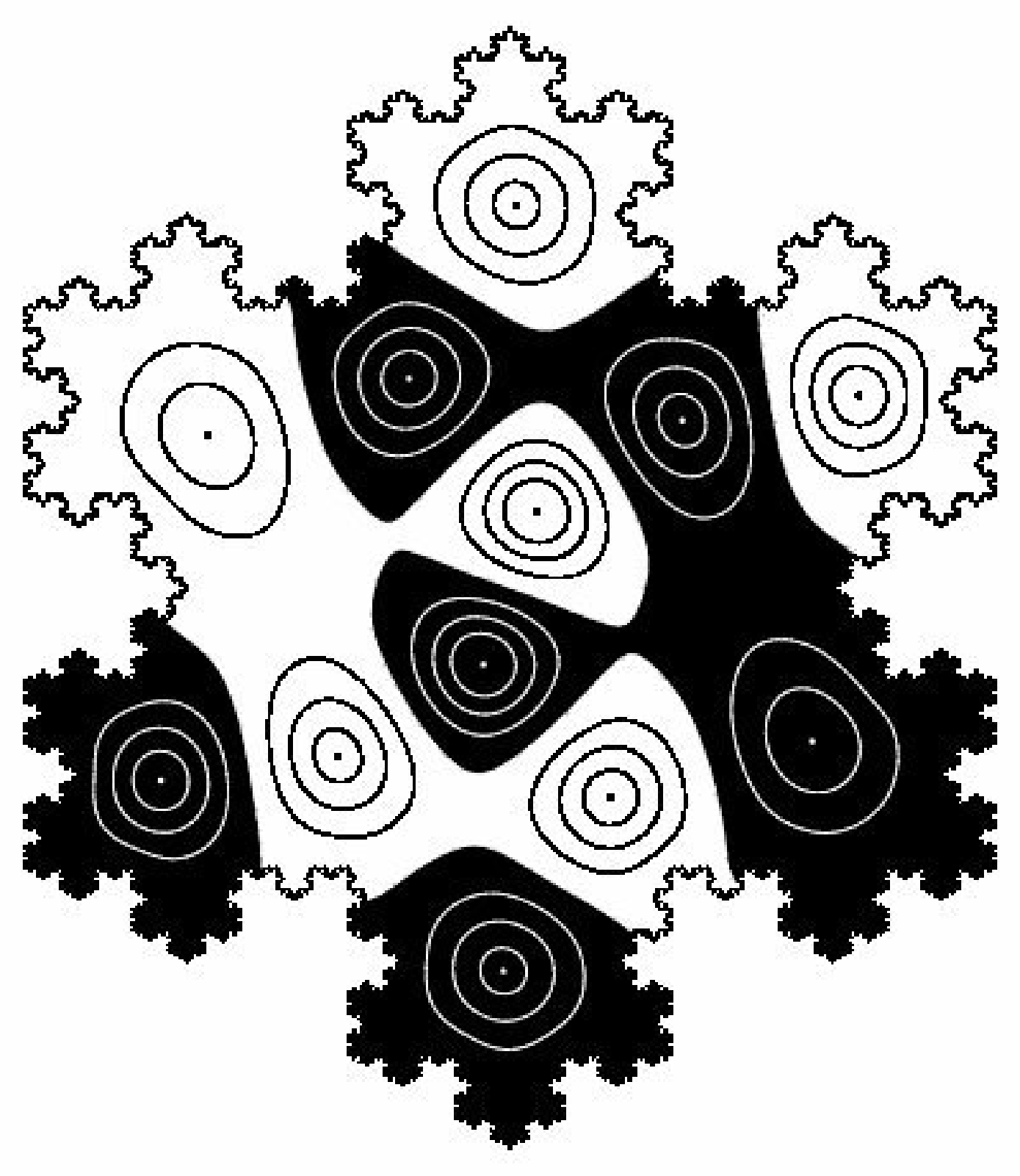}}
&
    \scalebox{\sbsize}{\includegraphics{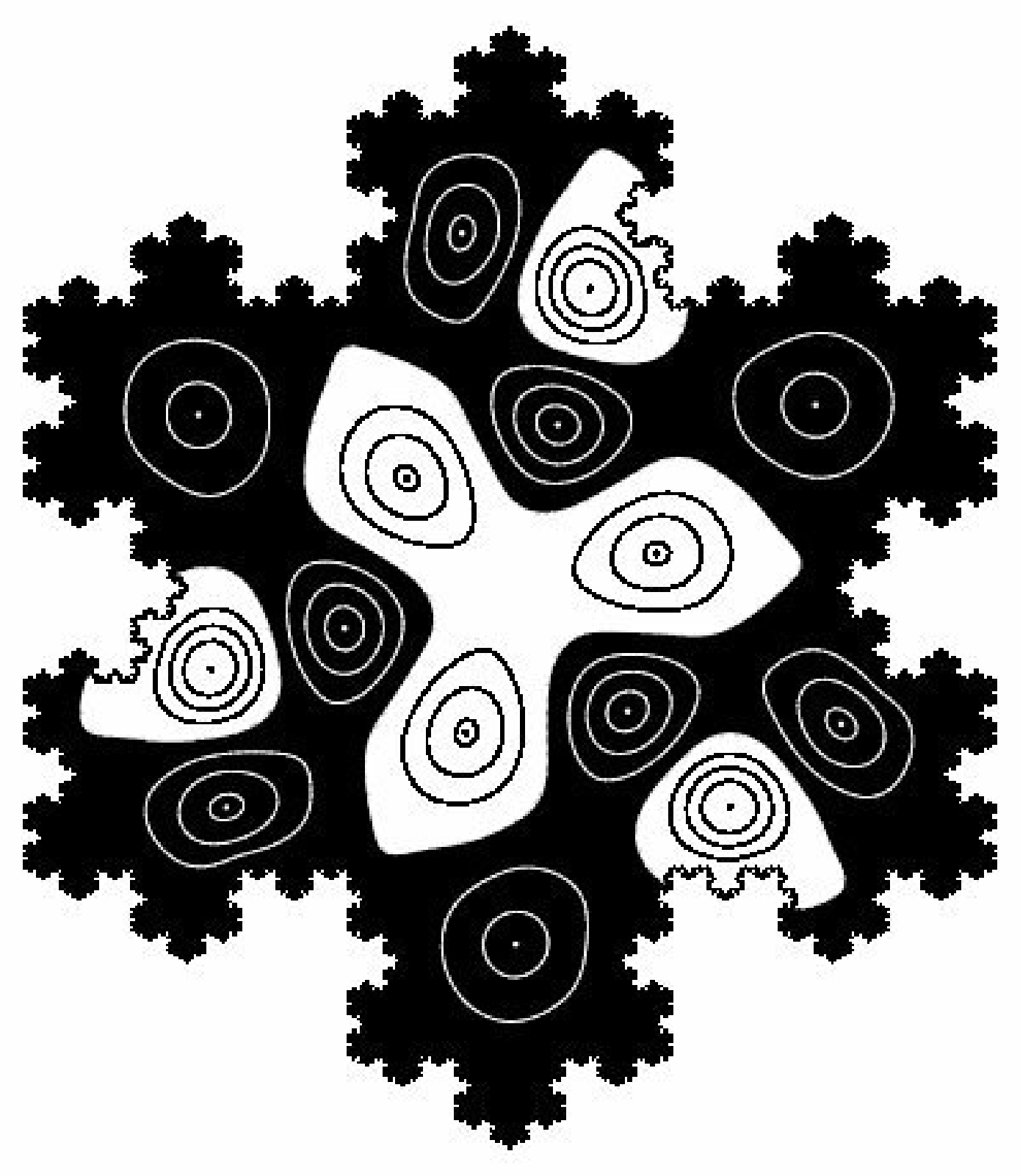}}
&
    \scalebox{\sbsize}{\includegraphics{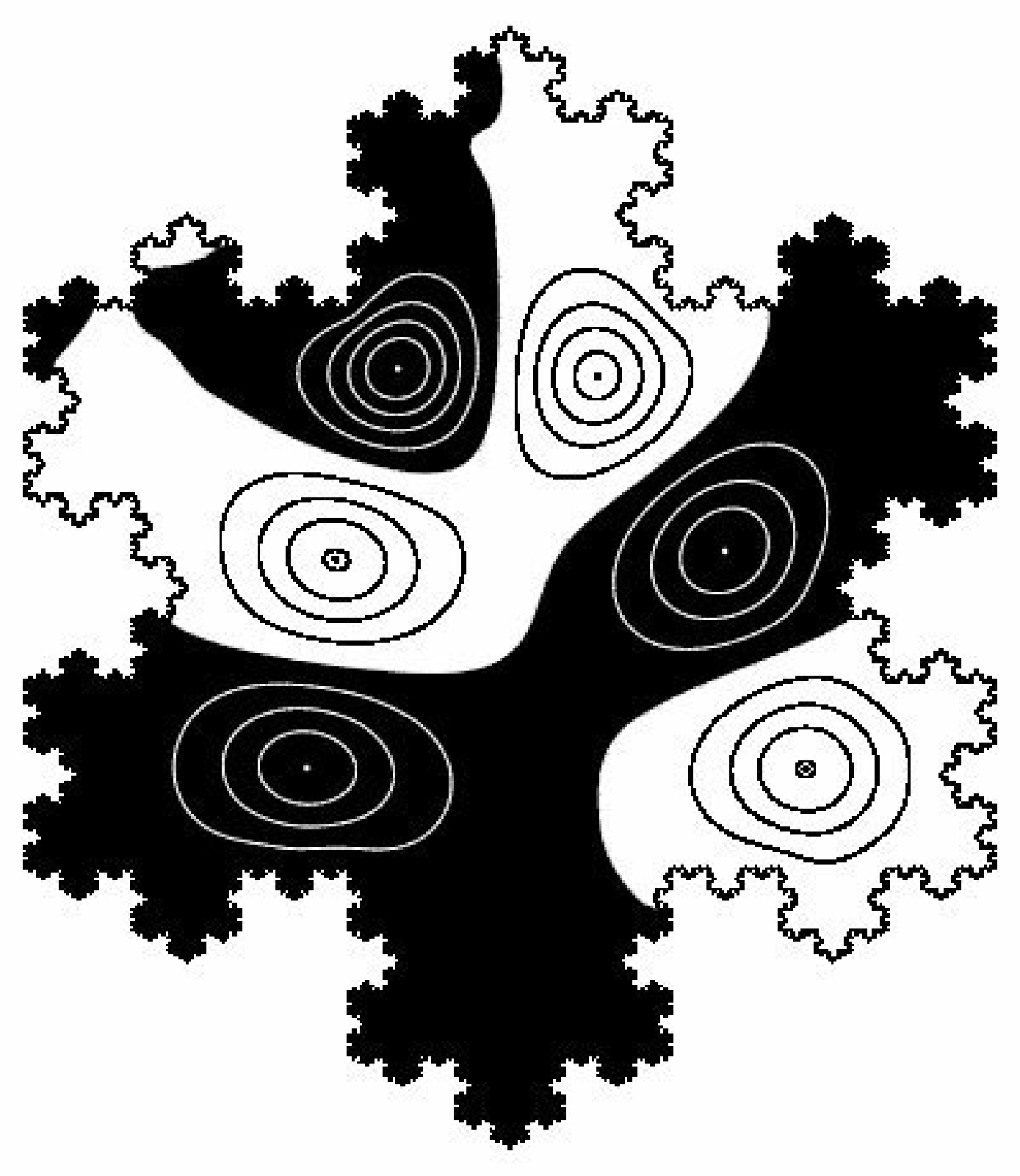}}
\\
$\Gam_{20} = \langle-\rho^3\rangle \cong \Z_2$  &
$\Gam_{21} = \langle\rho^2\rangle \cong \Z_3$   &
$\Gam_{22} = \langle 1 \rangle$
\\
\hline
\end{tabular}
%
%\vspace*{.2in}
\caption{Contour plots of solutions with symmetry types $S_{11},\ldots,S_{22}$ at $\lambda=0$.}
\label{sols2}
\end{center}
\end{figure}

Our goal was to find solutions to (\ref{pde}, \ref{nonlinearity}) at $\lam = 0$ with
each of the 23 symmetry types.
The 24-th primary branch is the first one with symmetry type $S_2$, so we
followed the first 24 primary branches.
With level $\ell = 5$ and $M = 300$ modes, which gave our most accurate results,
this found solutions with all symmetry types except $S_{11}$ and $S_{14}$.  We then
searched the first 100 primary branches, only following solutions with symmetry above $S_{11}$ and $S_{14}$
on the bifurcation digraph (Figure \ref{digraph}.)
In this way we found solutions with all 23 symmetry types.  The bifurcation diagrams which
lead to these solutions are shown in Figures~\ref{bif1-6}--\ref{bifD}.  We chose one solution at $\lam = 0$
with each symmetry type by taking the one descended from the lowest primary branch.  These choices are indicated
by dots in Figures~\ref{bif1-6}--\ref{bifD}, and the corresponding contour diagrams of the solutions
are shown
in Figures~\ref{sols1} and \ref{sols2}.
The contour diagrams use white for $u > 0$ and black for $u < 0$, and
gray indicates $u = 0$.
Equally spaced contours are drawn
along with dots for local extrema.
Details about the technique for generating these contour diagrams are found in \cite{nss}.

We ran our experiments using a range of modes and levels in order to
observe convergence and qualitative stability of the implementation of our algorithm.
At level $\ell = 5$ we have computed
300 eigenfunctions so $M \leq 300$ is possible.  At level $\ell = 6$ we computed only 100 eigenfunctions.
Due to our limited computational resources, using more than 100 modes
on level 6 was not practical.

As an indication of the convergence, consider the bifurcation diagram in Figure \ref{bif1-6}.
The diagram looks qualitatively the same for any choice of $\ell$ and $M$ that we used.
The position of the bifurcation point creating the $S_{10}$ solution (near $\lam = 30$) changes slightly,
according to this table:
\begin{center}
\begin{tabular}{c|ccc}
          & $\ell = 4$ & $\ell = 5$ & $\ell = 6$ \\
\hline
$M = 100$ & 35.3931 & 34.9814 & 34.9252 \\
$M = 200$ & 32.1131 & 32.2964 &  \\
$M = 300$ &         & 32.0518 &
\end{tabular} .
\end{center}
The level 5 and 6 approximations with $M = 100$ modes are very close, but increasing the mode number
has more of an effect.
This indicates that the results with $(\ell, M) = (5, 300)$
are more accurate than those with $(6,100)$.
Figure~{\ref{levelmode}} shows how $u(2/27, 4\sqrt{3}/27)$ varies with mode number and $\ell$
for the solution with $S_{10}$ symmetry at $\lambda=0$ shown in Figures \ref{bif1-6} and \ref{sols1}.
The horizontal segments of the graphs correspond to the addition of modes with zero coefficients
for this solution.
Based on this and other similar convergence results,
we chose to use level 5 with 300 modes in most of our
numerical experiments.
\begin{figure}
\begin{center}
\rotatebox{-90}{\scalebox{.6}{\includegraphics{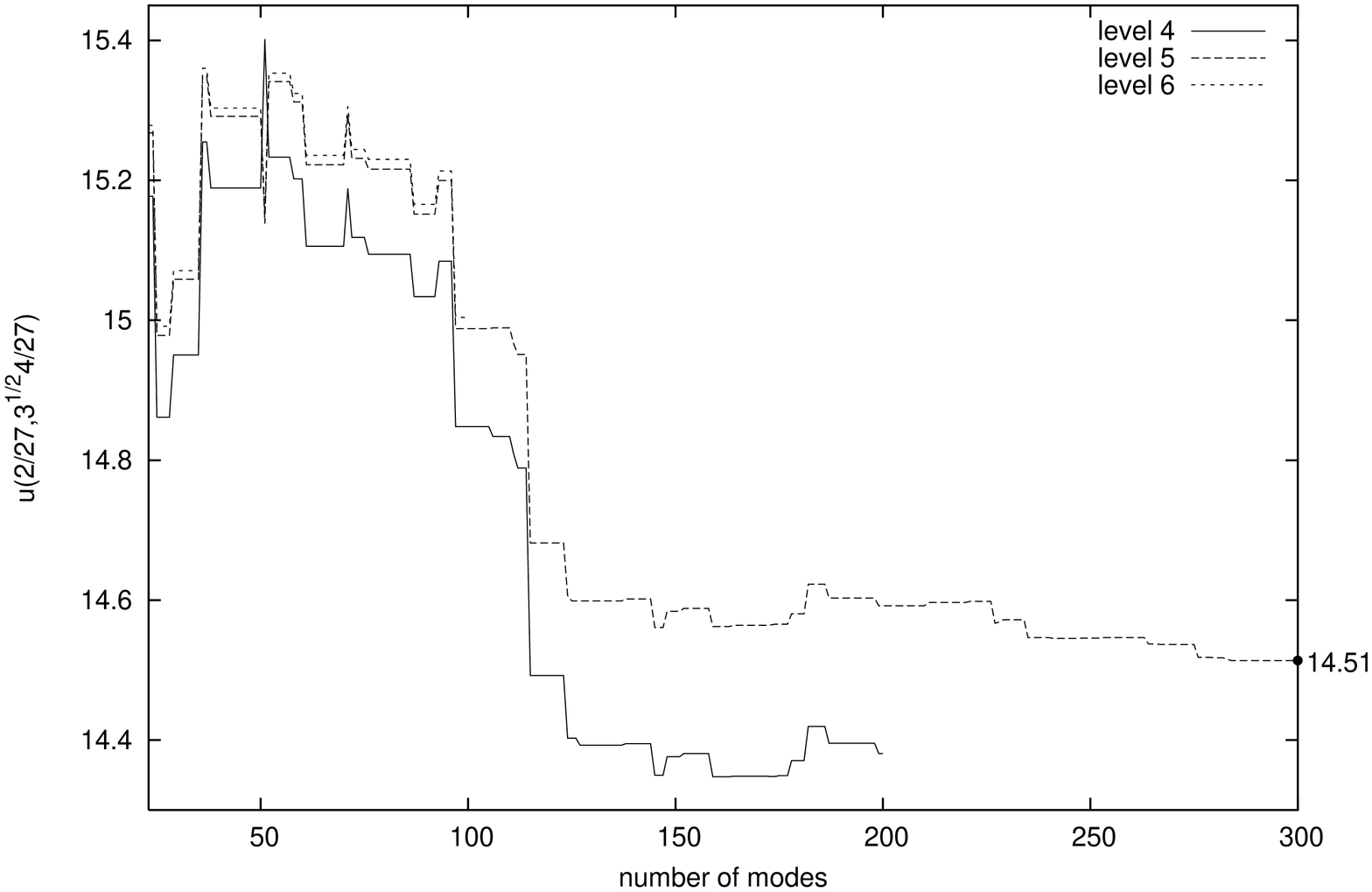}}}
\caption{A plot of $u(2/27,4 \sqrt{3}/27)$ as a function of the number of modes for the
lowest energy solution at $\lambda=0$ with symmetry type $S_{10}$. The point at $M = 300$
matches the point labelled with $S_{10}$ in Figure~\ref{bif1-6}.
%The curves cross near $M = 50$ because $\psi_{51}$ at level $\ell = 4$, and
%$\psi_{52}$ at levels 5 and 6 approximate the same eigenfunction.  Similarly,
%$\psi_{52}$ at level $\ell = 4$, and
%$\psi_{51}$ at levels 5 and 6 approximate the same eigenfunction.
%This unusual labelling happens because the approximate eigenvalues of the two eigenfunctions
%in questions are ordered differently at $\ell = 4$ and $\ell = 5$,
%whereas $\lam_{51} < \lam_{52}$ by definition.
%This ``crossing of eigenvalues'' does not happen elsewhere.
}
\label{levelmode}
\end{center}
\end{figure}

\end{section}

%\begin{section}{MPA}
%
%\input mpa_sym.tex

%\end{section}

\begin{section}{Conclusions.}
\label{conclusion_section}

We are currently working on a more general program for recursive branch
following in symmetric systems.
%Our goal is to write a suite of
%programs that will create the full bifurcation diagram with a
%single command for any one-parameter gradient system $\nabla J(\lam, u) = 0$,
%where $J: \R \times \R^n \rightarrow \R$ is
%($\Gam \times \Z_2$)-invariant, $\Gam$ is a
%permutation group acting on $\R^n$, and the nontrivial element of $\Z_2$ acts on $\R^n$
%as $-I$.  In the current paper, the group $\D_6$ acts as permutations
%on $G_N \cong \R^N$, the space of functions on a grid.
%
Starting with any graph, the analog to Equation
\ref{pde} is the Partial {\em difference} Equation (PdE) $Lu + f(u) = 0$ \cite{graph},
where $L$ is the well-known discrete Laplacian on that graph
and $u$ is a real-valued function on the vertices.
Discretizing a PDE as we have done in this paper leads to a
PdE on a graph with a large number of vertices.
The grid points are the vertices of the graph, and the edges of the graph connect nearest
neighbor grid points.
Starting with an arbitrary graph, our new suite of programs will
analyze the symmetry of the graph and
compute the bifurcation diagrams for the PdE on the graph.

The programs we describe in the current paper will work with other
superlinear odd $f$ and other regions with hexagonal symmetry. The
nonlinearity $f$ needs to be superlinear since our program assumes that
the branches eventually ``go to the left."  Our general program
will not have this restriction; the GNGA and pmGNGA will be replaced by
a single method of branch
following that is able to go through fold points, and has no prejudice
about the parameter increasing or decreasing.  This
new method of branch following has already been successfully implemented
in \cite{thompson}.
We hope to write the new code so that a cluster of computers can
be used in parallel, with each computer following a single branch at one time,
under the control of a central PERL script.

With minor modifications, our program would analyze the PDE (\ref{pde})
even when $f$ is not odd.
The appropriate bifurcation digraph for $\D_6$ acting on $L^2(\Om)$
is a subgraph of the digraph in Figure \ref{digraph},
so the bifurcating branches would be followed properly
unless the symmetry of the mother solution is incorrectly identified.
The Perl scripts which start with the trivial branch would have to be modified, since $u = 0$ is not
a solution when $f$ is not odd (unless $f(0) = 0$).
If $f(0) = 0$, the trivial branch exists, but its bifurcations
are not properly described by the bifurcation digraph in Figure~\ref{digraph},
and some special code would be needed to handle these bifurcations.
%but $f$ is not odd, then the bifurcations from the trivial
%solution would have to be handled by
%With the exception of $u = 0$
%When $f$ is non-odd,
%only subgroups of $\D_6$ are isotropy subgroups.  We would
%have to check for false identification of isotropy subgroups,
%but that would happen rarely.

%The hardest case would be if $f_\lam$ is non-odd and
%$f_\lam( 0) = 0$  (for example $f_\lam(u) = \lam u + u^2$ if $u \geq 0$ and
%$f_\lam(u) =  \lam u + u^3$ if $u < 0$.)
%In this case $u=0$, would be a solution for all
%values of the parameter $\lam$, and the program described in this paper
%would think that
%the symmetry of $u=0$ is $\DZ$ when it is actually $\D_6$.
%At bifurcations of $u=0$, the eigenfunctions $\psi_j$ and $-\psi_j$
%would spawn two non-conjugate branches whereas the program as it stands
%would assume that the two branches are conjugate and only follow one of them.

%If we used our program without modification for non-odd $f_\lam$
%that do not satisfy $f_\lam(0) = 0$
%the problem would be that $u = 0$ is not a solution to the PDE.
%We would have to find at least one solution by trial and error to
%start the first branch.
%The simplest procedure would be to find the positive and
%negative solutions using initial guesses of the form $u = c \psi_1$.
%These two solutions could then be
%used as the starting points for two separate runs of our
%recursive branch following program.

It is valid to ask the question ``does the GNGA converge" (as
implemented in this current research). While we do not have a
complete proof affirming the positive of this conjecture, many
references contain relevant theorems. The GNGA is an
implementation of Newton's method, which indeed converges under
standard assumptions. In \cite{k2}, one finds the classical fixed
point iteration proof that Newton's method in $\R^N$ converges
when the initial guess is sufficiently close to a nondegenerate
zero of the object function. This proof applies almost without
change to the infinite dimensional case. Also addressed in
\cite{k2} are algorithms where the object function and/or its
derivative are only approximated; this would apply to our
implementation due to numerical integration errors, as well as
owing to our imperfect knowledge of the eigenfunctions and
corresponding eigenvalues. While not discussed exactly in the
cited literature, elementary fixed point arguments indicate that
the restriction of our object function $\nabla J$ to sufficiently
large subspaces $B_M$ will still result in convergent iterations.
It would be worthwhile to string these type of results together in
order to obtain a ``best possible'' GNGA convergence theorem.
%We have:
%
%\begin{conjecture}
%Let $u$ be a nondegenerate critical point of $J$ and $\epsilon>0$.
%Then there exists natural numbers $\hat M$ and $\hat N$, and $\delta > 0$ so that
%for $M>\hat M$, $N>\hat N$, eigenfunction and eigenvalue approximations $\tilde \psi_i$ and %$\tilde\lambda_i$,
%$i=1,\ldots,M$ for the first $M$ eigenpairs $\{(\psi_i,\lambda_i)\}_{i=1,\ldots,M}$ with
%$||\tilde\psi_i-\psi_i||_H<\delta$ and $|\tilde\lambda_i-\lambda_i|<\delta$, $i=1,\ldots,M$,
%so that given $u_0\in B_{\delta}(u)\cap {\rm span}\{\tilde\psi_i\}_{i=1,\ldots,M}$ and %assuming that the
%errors in integration in forming the entries of the gradient $g\in R^M$ and the $M\times M$ %Hessian matrix are
%no worse than $\delta$, GNGA converges to a function $\tilde u\in B_{\epsilon}(u)\cap B_M$.
%\end{conjecture}
%
Monograph \cite{k1} gives an easy introduction into some of the details of implementing Newton's method to solve nonlinear problems.
Further, in the spirit of \cite{cdn} and \cite{wangzhou1},
by the invariance of the Newton map, any convergence result should hold in fixed point
subspaces corresponding to a given symmetry type.
The articles \cite{lizhou,wangzhou1} and others by those authors discuss the convergence of algorithms
similar to the GNGA, at times also considering symmetry restrictions.
Finally, the well-known book \cite{berger} contains relevant convergence results for Newton and approximate Newton  iterative fixed point algorithms.

In summary, we have written a suite of programs that automatically
computes the bifurcation diagram of the PDE (\ref{pde}, \ref{nonlinearity}).
The program finds solutions with each of the 23 symmetry types by following
solution branches which are connected to
the trivial branch by a sequence of symmetry-breaking bifurcations.
A thorough understanding of the possible symmetry-breaking bifurcations
is required for this task.  We introduced the bifurcation digraph, which summarizes the results
of the necessary symmetry calculations.  For the group $\DZ$, these calculations were
done by hand and verified by the GAP computer program \cite{GAP, matthews}.
In the future, we plan to implement automated branch following in
systems where the symmetry group is so complicated that GAP is necessary.

\end{section}

%-----BIBLIOGRAPHY-------------------------------------------------------------

\end{document}